\documentclass{article}
\usepackage{amsmath}
\usepackage{amssymb}
\usepackage{amsthm}
\usepackage{bbm}
\usepackage{cite}
\usepackage{stmaryrd}
\SetSymbolFont{stmry}{bold}{U}{stmry}{m}{n}
\usepackage{euscript}
\usepackage[arrow,curve,matrix,arc,2cell]{xy}
\UseAllTwocells
\usepackage[utf8]{inputenc}
\usepackage[unicode]{hyperref}
\usepackage{slashed}
\DeclareFontFamily{U}{rsfs}{} 
\DeclareFontShape{U}{rsfs}{n}{it}{<->
rsfs10}{} \DeclareSymbolFont{mscr}{U}{rsfs}{n}{it}
\DeclareSymbolFontAlphabet{\scr}{mscr}
\def\mathscr{\scr}
\begin{document}
\def\e#1\e{\begin{equation}#1\end{equation}}
\def\ea#1\ea{\begin{align}#1\end{align}}
\def\eq#1{{\rm(\ref{#1})}}
\theoremstyle{plain}
\newtheorem{thm}{Theorem}[section]
\newtheorem{lem}[thm]{Lemma}
\newtheorem{prop}[thm]{Proposition}
\newtheorem{cor}[thm]{Corollary}
\newtheorem{quest}[thm]{Question}
\newtheorem{prob}[thm]{Problem}
\theoremstyle{definition}
\newtheorem{dfn}[thm]{Definition}
\newtheorem{ex}[thm]{Example}
\newtheorem{rem}[thm]{Remark}
\newtheorem{ax}[thm]{Axiom}
\newtheorem{ass}[thm]{Assumption}
\newtheorem{property}[thm]{Property}
\newtheorem{cond}[thm]{Condition}
\numberwithin{figure}{section}
\numberwithin{equation}{section}
\def\dim{\mathop{\rm dim}\nolimits}
\def\codim{\mathop{\rm codim}\nolimits}
\def\vdim{\mathop{\rm vdim}\nolimits}
\def\sign{\mathop{\rm sign}\nolimits}
\def\Im{\mathop{\rm Im}\nolimits}
\def\det{\mathop{\rm det}\nolimits}
\def\Ker{\mathop{\rm Ker}}
\def\Coker{\mathop{\rm Coker}}
\def\Spec{\mathop{\rm Spec}}
\def\Perf{\mathop{\rm Perf}}
\def\Vect{\mathop{\rm Vect}}
\def\LCon{\mathop{\rm LCon}}
\def\Flag{\mathop{\rm Flag}\nolimits}
\def\FlagSt{\mathop{\rm FlagSt}\nolimits}
\def\Iso{\mathop{\rm Iso}\nolimits}
\def\Aut{\mathop{\rm Aut}}
\def\End{\mathop{\rm End}}
\def\Ho{\mathop{\rm Ho}}
\def\PGL{\mathop{\rm PGL}}
\def\GL{\mathop{\rm GL}}
\def\SL{\mathop{\rm SL}}
\def\SO{\mathop{\rm SO}}
\def\SU{\mathop{\rm SU}}
\def\Sp{\mathop{\rm Sp}}
\def\Ch{\mathop{\rm Ch}\nolimits}
\def\Spin{\mathop{\rm Spin}}
\def\Spinc{\mathop{\rm Spin^c}}
\def\SF{\mathop{\rm SF}}
\def\Tr{\mathop{\rm Tr}}
\def\U{{\mathbin{\rm U}}}
\def\vol{\mathop{\rm vol}}
\def\inc{\mathop{\rm inc}}
\def\ind{\mathop{\rm ind}\nolimits}
\def\tind{{\text{\rm t-ind}}}
\def\bdim{{\mathbin{\bf dim}\kern.1em}}
\def\rk{\mathop{\rm rk}}
\def\Pic{\mathop{\rm Pic}}
\def\colim{\mathop{\rm colim}\nolimits}
\def\Stab{\mathop{\rm Stab}\nolimits}
\def\Exact{\mathop{\rm Exact}\nolimits}
\def\Crit{\mathop{\rm Crit}}
\def\Cass{\mathop{\rm Cass}}
\def\supp{\mathop{\rm supp}}
\def\Or{\mathop{\rm Or}}
\def\rank{\mathop{\rm rank}\nolimits}
\def\Hom{\mathop{\rm Hom}\nolimits}
\def\bHom{\mathop{\bf Hom}\nolimits}
\def\Ext{\mathop{\rm Ext}\nolimits}
\def\cExt{\mathop{{\mathcal E}\mathit{xt}}\nolimits}
\def\id{{\mathop{\rm id}\nolimits}}
\def\Id{{\mathop{\rm Id}\nolimits}}
\def\Sch{\mathop{\bf Sch}\nolimits}
\def\Map{{\mathop{\rm Map}\nolimits}}
\def\AlgSp{\mathop{\bf AlgSp}\nolimits}
\def\Art{\mathop{\bf Art}\nolimits}
\def\HSt{\mathop{\bf HSt}\nolimits}
\def\dSt{\mathop{\bf dSt}\nolimits}
\def\dArt{\mathop{\bf dArt}\nolimits}
\def\TopSta{{\mathop{\bf TopSta}\nolimits}}
\def\Topho{{\mathop{\bf Top^{ho}}}}
\def\Gpds{{\mathop{\bf Gpds}\nolimits}}
\def\Aff{\mathop{\bf Aff}\nolimits}
\def\SPr{\mathop{\bf SPr}\nolimits}
\def\SSet{\mathop{\bf SSet}\nolimits}
\def\Parf{\mathop{\bf Perf}}
\def\Ad{\mathop{\rm Ad}}
\def\pl{{\rm pl}}
\def\mix{{\rm mix}}
\def\fd{{\rm fd}}
\def\fpd{{\rm fpd}}
\def\pfd{{\rm pfd}}
\def\coa{{\rm coa}}
\def\rsi{{\rm si}}
\def\rst{{\rm st}}
\def\ss{{\rm ss}}
\def\vi{{\rm vi}}
\def\smq{{\rm smq}}
\def\rsm{{\rm sm}}
\def\cla{{\rm cla}}
\def\rArt{{\rm Art}}
\def\po{{\rm po}}
\def\spo{{\rm spo}}
\def\Kur{{\rm Kur}}
\def\dcr{{\rm dcr}}
\def\top{{\rm top}}
\def\fc{{\rm fc}}
\def\cla{{\rm cla}}
\def\num{{\rm num}}
\def\irr{{\rm irr}}
\def\red{{\rm red}}
\def\sing{{\rm sing}}
\def\virt{{\rm virt}}
\def\qcoh{{\rm qcoh}}
\def\coh{{\rm coh}}
\def\lft{{\rm lft}}
\def\lfp{{\rm lfp}}
\def\cs{{\rm cs}}
\def\dR{{\rm dR}}
\def\Obj{{\rm Obj}}
\def\cdga{{\mathop{\bf cdga}\nolimits}}
\def\Rmod{\mathop{R\text{\rm -mod}}}
\def\Top{{\mathop{\bf Top}\nolimits}}
\def\modKQ{\mathop{\text{\rm mod-}\K Q}}
\def\modKQI{\text{\rm mod-$\K Q/I$}}
\def\modCQ{\mathop{\text{mod-}\C Q}}
\def\modCQI{\text{\rm mod-$\C Q/I$}}
\def\ul{\underline}
\def\bs{\boldsymbol}
\def\ge{\geqslant}
\def\le{\leqslant\nobreak}
\def\boo{{\mathbin{\mathbbm 1}}}
\def\O{{\mathcal O}}
\def\bA{{\mathbin{\mathbb A}}}
\def\bG{{\mathbin{\mathbb G}}}
\def\bL{{\mathbin{\mathbb L}}}
\def\P{{\mathbin{\mathbb P}}}
\def\bT{{\mathbin{\mathbb T}}}
\def\H{{\mathbin{\mathbb H}}}
\def\K{{\mathbin{\mathbb K}}}
\def\R{{\mathbin{\mathbb R}}}
\def\Z{{\mathbin{\mathbb Z}}}
\def\Q{{\mathbin{\mathbb Q}}}
\def\N{{\mathbin{\mathbb N}}}
\def\C{{\mathbin{\mathbb C}}}
\def\CP{{\mathbin{\mathbb{CP}}}}
\def\KP{{\mathbin{\mathbb{KP}}}}
\def\RP{{\mathbin{\mathbb{RP}}}}
\def\fC{{\mathbin{\mathfrak C}\kern.05em}}
\def\fD{{\mathbin{\mathfrak D}}}
\def\fE{{\mathbin{\mathfrak E}}}
\def\fF{{\mathbin{\mathfrak F}}}
\def\A{{\mathbin{\cal A}}}
\def\G{{{\cal G}}}
\def\M{{\mathbin{\cal M}}}
\def\B{{\mathbin{\cal B}}}
\def\ovB{{\mathbin{\smash{\,\overline{\!\mathcal B}}}}}
\def\cC{{\mathbin{\cal C}}}
\def\cD{{\mathbin{\cal D}}}
\def\cE{{\mathbin{\cal E}}}
\def\cF{{\mathbin{\cal F}}}
\def\cG{{\mathbin{\cal G}}}
\def\cH{{\mathbin{\cal H}}}
\def\cI{{\mathbin{\cal I}}}
\def\cJ{{\mathbin{\cal J}}}
\def\cK{{\mathbin{\cal K}}}
\def\cL{{\mathbin{\cal L}}}
\def\bcM{{\mathbin{\bs{\cal M}}}}
\def\cN{{\mathbin{\cal N}\kern .04em}}
\def\cP{{\mathbin{\cal P}}}
\def\cQ{{\mathbin{\cal Q}}}
\def\cR{{\mathbin{\cal R}}}
\def\cS{{\mathbin{\cal S}}}
\def\T{{{\cal T}\kern .04em}}
\def\cW{{\mathbin{\cal W}}}
\def\cX{{\cal X}}
\def\cY{{\cal Y}}
\def\cZ{{\cal Z}}
\def\oM{{\mathbin{\smash{\,\,\overline{\!\!\mathcal M\!}\,}}}}
\def\cV{{\cal V}}
\def\cW{{\cal W}}
\def\g{{\mathfrak g}}
\def\h{{\mathfrak h}}
\def\m{{\mathfrak m}}
\def\u{{\mathfrak u}}
\def\so{{\mathfrak{so}}}
\def\su{{\mathfrak{su}}}
\def\sp{{\mathfrak{sp}}}
\def\fW{{\mathfrak W}}
\def\fX{{\mathfrak X}}
\def\fY{{\mathfrak Y}}
\def\fZ{{\mathfrak Z}}
\def\bM{{\bs M}}
\def\bN{{\bs N}}
\def\bO{{\bs O}}
\def\bQ{{\bs Q}}
\def\bS{{\bs S}}
\def\bU{{\bs U}}
\def\bV{{\bs V}}
\def\bW{{\bs W}\kern -0.1em}
\def\bX{{\bs X}}
\def\bY{{\bs Y}\kern -0.1em}
\def\bZ{{\bs Z}}
\def\al{\alpha}
\def\be{\beta}
\def\ga{\gamma}
\def\de{\delta}
\def\io{\iota}
\def\ep{\epsilon}
\def\la{\lambda}
\def\ka{\kappa}
\def\th{\theta}
\def\ze{\zeta}
\def\up{\upsilon}
\def\vp{\varphi}
\def\si{\sigma}
\def\om{\omega}
\def\De{\Delta}
\def\Ka{{\rm K}}
\def\La{\Lambda}
\def\Om{\Omega}
\def\Ga{\Gamma}
\def\Si{\Sigma}
\def\Th{\Theta}
\def\Up{\Upsilon}
\def\Chi{{\rm X}}
\def\Tau{{\rm T}}
\def\Nu{{\rm N}}
\def\pd{\partial}
\def\ts{\textstyle}
\def\st{\scriptstyle}
\def\sst{\scriptscriptstyle}
\def\w{\wedge}
\def\sm{\setminus}
\def\lt{\ltimes}
\def\bu{\bullet}
\def\sh{\sharp}
\def\di{\diamond}
\def\he{\heartsuit}
\def\od{\odot}
\def\op{\oplus}
\def\ot{\otimes}
\def\bt{\boxtimes}
\def\ov{\overline}
\def\bigop{\bigoplus}
\def\bigot{\bigotimes}
\def\iy{\infty}
\def\es{\emptyset}
\def\ra{\rightarrow}
\def\rra{\rightrightarrows}
\def\Ra{\Rightarrow}
\def\Longra{\Longrightarrow}
\def\ab{\allowbreak}
\def\longra{\longrightarrow}
\def\hookra{\hookrightarrow}
\def\dashra{\dashrightarrow}
\def\lb{\llbracket}
\def\rb{\rrbracket}
\def\ha{{\ts\frac{1}{2}}}
\def\t{\times}
\def\ci{\circ}
\def\ti{\tilde}
\def\d{{\rm d}}
\def\md#1{\vert #1 \vert}
\def\ms#1{\vert #1 \vert^2}
\def\bmd#1{\big\vert #1 \big\vert}
\def\bms#1{\big\vert #1 \big\vert^2}
\def\an#1{\langle #1 \rangle}
\def\ban#1{\bigl\langle #1 \bigr\rangle}
\title{On orientations for gauge-theoretic moduli spaces}
\author{Dominic Joyce, Yuuji Tanaka and Markus Upmeier}
\date{}
\maketitle

\begin{abstract} 
Let $X$ be a compact manifold, $D:\Ga^\iy(E_0)\ra\Ga^\iy(E_1)$ a real elliptic operator on $X$, $G$ a Lie group, $P\ra X$ a principal $G$-bundle, and $\B_P$ the infinite-dimensional moduli space of all connections $\nabla_P$ on $P$ modulo gauge, as a topological stack. For each $[\nabla_P]\in\B_P$, we can consider the twisted elliptic operator $D^{\nabla_{\Ad(P)}}:\Ga^\iy(\Ad(P)\ot E_0)\ra\Ga^\iy(\Ad(P)\ot E_1)$ on $X$. This is a continuous family of elliptic operators over the base $\B_P$, and so has an orientation bundle $O^{E_\bu}_P\ra\B_P$, a principal $\Z_2$-bundle parametrizing orientations of $\Ker D^{\nabla_{\Ad(P)}}\op \Coker D^{\nabla_{\Ad(P)}}$ at each $[\nabla_P]$. An {\it orientation\/} on $(\B_P,E_\bu)$ is a trivialization~$O^{E_\bu}_P\cong\B_P\t\Z_2$.

In gauge theory one studies moduli spaces $\M_P^{\rm ga}$ of connections $\nabla_P$ on $P$ satisfying some curvature condition, such as anti-self-dual instantons on Riemannian 4-manifolds $(X,g)$. Under good conditions $\M_P^{\rm ga}$ is a smooth manifold, and orientations on $(\B_P,E_\bu)$ pull back to orientations on $\M_P^{\rm ga}$ in the usual sense of differential geometry under the inclusion $\M_P^{\rm ga}\hookra\B_P$. This is important in areas such as Donaldson theory, where one needs an orientation on $\M_P^{\rm ga}$ to define enumerative invariants.

We explain a package of techniques, some known and some new, for proving orientability and constructing canonical orientations on $(\B_P,E_\bu)$, after fixing some algebro-topological information on $X$. We use these to construct canonical orientations on gauge theory moduli spaces, including new results for moduli spaces of flat connections on 2- and 3-manifolds, instantons, the Kapustin--Witten equations, and the Vafa--Witten equations on 4-manifolds, and the Haydys--Witten equations on 5-manifolds.
\end{abstract}

\setcounter{tocdepth}{2}
\tableofcontents

\section{Introduction}
\label{or1}

We first set up the problem we wish to discuss.

\begin{dfn}
\label{or1def1}
Suppose we are given the following data:
\begin{itemize}
\setlength{\itemsep}{0pt}
\setlength{\parsep}{0pt}
\item[(a)] A compact, connected manifold $X$, of dimension $n>0$.
\item[(b)] A Lie group $G$, with $\dim G>0$, and centre $Z(G)\subseteq G$, and Lie algebra $\g$.
\item[(c)] A principal $G$-bundle $\pi:P\ra X$. We write $\Ad(P)\ra X$ for the vector bundle with fibre $\g$ defined by $\Ad(P)=(P\t\g)/G$, where $G$ acts on $P$ by the principal bundle action, and on $\g$ by the adjoint action.
\end{itemize}

Write $\A_P$ for the set of connections $\nabla_P$ on the principal bundle $P\ra X$. This is a real affine space modelled on the infinite-dimensional vector space $\Ga^\iy(\Ad(P))$, and we make $\A_P$ into a topological space using the $C^\iy$ topology on $\Ga^\iy(\Ad(P))$. Here if $E\ra X$ is a vector bundle then $\Ga^\iy(E)$ denotes the vector space of smooth sections of $E$. Note that $\A_P$ is contractible.
 
Write $\G_P=\Aut(P)$ for the infinite-dimensional Lie group of $G$-equivariant diffeomorphisms $\ga:P\ra P$ with $\pi\ci\ga=\pi$. Then $\G_P$ acts on $\A_P$ by gauge transformations, and the action is continuous for the topology on~$\A_P$. 

There is an inclusion $Z(G)\hookra\G_P$ mapping $z\in Z(G)$ to the principal bundle action of $z$ on $P$. As $X$ is connected, this identifies $Z(G)$ with the centre $Z(\G_P)$ of $\G_P$, so we may take the quotient group $\G_P/Z(G)$. The action of $Z(G)\subset\G_P$ on $\A_P$ is trivial, so the $\G_P$-action on $\A_P$ descends to a $\G_P/Z(G)$-action. 

Each $\nabla_P\in\A_P$ has a (finite-dimensional) {\it stabilizer group\/} $\Stab_{\G_P}(\nabla_P)\subset\G_P$ under the $\G_P$-action on $\A_P$, with $Z(G)\subseteq\Stab_{\G_P}(\nabla_P)$. As $X$ is connected, $\Stab_{\G_P}(\nabla_P)$ is isomorphic to a Lie subgroup $H$ of $G$ with $Z(G)\subseteq H$. As in \cite[p.~133]{DoKr} we call $\nabla_P$ {\it irreducible\/} if $\Stab_{\G_P}(\nabla_P)=Z(G)$, and {\it reducible\/} otherwise. Write $\A_P^\irr,\A_P^\red$ for the subsets of irreducible and reducible connections in $\A_P$. Then $\A_P^\irr$ is open and dense in $\A_P$, and $\A_P^\red$ is closed and of infinite codimension in the infinite-dimensional affine space $\A_P$. Hence the inclusion $\A_P^\irr\hookra\A_P$ is a weak homotopy equivalence, and $\A_P^\irr$ is weakly contractible.

We write $\B_P=[\A_P/\G_P]$ for the moduli space of gauge equivalence classes of connections on $P$, considered as a {\it topological stack\/} in the sense of Metzler \cite{Metz} and Noohi \cite{Nooh1,Nooh2}. Topological stacks will be discussed in Remark \ref{or2rem1}. Write $\B_P^\irr=[\A_P^\irr/\G_P]$ for the substack $\B_P^\irr\subseteq\B_P$ of irreducible connections. 

Define variations $\ovB_P=[\A_P/(\G_P/Z(G))]$, $\ovB_P^\irr=[\A_P^\irr/(\G_P/Z(G))]$ of $\B_P,\ab\B_P^\irr$. Then $\ovB_P$ is a topological stack, but as $\G_P/Z(G)$ acts freely on $\A_P^\irr$, we may regard $\ovB_P^\irr$ as a topological space (an example of a topological stack).

There are natural morphisms $\Pi_P:\B_P\ra\ovB_P$, $\Pi_P^\irr:\B^\irr_P\ra\ovB^\irr_P$ in the homotopy category $\Ho(\TopSta)$ of the 2-category $\TopSta$ of topological stacks, induced by $\id:\A_P\ra\A_P$ and the projection $\cG_P\ra\cG_P/Z(G)$. These are fibrations of topological stacks, with fibre the quotient stack $[*/Z(G)]$, making $\B_P,\B^\irr_P$ into $Z(G)$-gerbes over~$\ovB_P,\ovB^\irr_P$.

We mostly care about the spaces $\B_P$, but we need $\ovB_P$ to make the connection with orientations on gauge theory moduli spaces as in Remark \ref{or1rem1} and~\S\ref{or4}. 
\end{dfn}

We define orientation bundles $O^{E_\bu}_P,\bar O^{E_\bu}_P$ on the moduli spaces~$\B_P,\ovB_P$:

\begin{dfn}
\label{or1def2}
Work in the situation of Definition \ref{or1def1}, with the same notation. Suppose we are given real vector bundles $E_0,E_1\ra X$, of the same rank $r$, and a linear elliptic partial differential operator $D:\Ga^\iy(E_0)\ra\Ga^\iy(E_1)$, of degree $d$. As a shorthand we write $E_\bu=(E_0,E_1,D)$. With respect to connections $\nabla_{E_0}$ on $E_0\ot\bigot^iT^*X$ for $0\le i<d$, when $e\in\Ga^\iy(E_0)$ we may write
\e
D(e)=\sum_{i=0}^d a_i\cdot \nabla_{E_0}^ie,
\label{or1eq1}
\e
where $a_i\in \Ga^\iy(E_0^*\ot E_1\ot S^iTX)$ for $i=0,\ldots,d$. The condition that $D$ is {\it elliptic\/} is that the {\it symbol\/} $\si(D)_{x,\xi}=a_d\vert_x\cdot\ot^d\xi:E_0\vert_x\ra E_1\vert_x$ is an isomorphism for all $x\in X$ and $0\ne\xi\in T_x^*X$.

Let $\nabla_P\in\A_P$. Then $\nabla_P$ induces a connection $\nabla_{\Ad(P)}$ on the vector bundle $\Ad(P)\ra X$. Thus we may form the twisted elliptic operator
\e
\begin{split}
D^{\nabla_{\Ad(P)}}&:\Ga^\iy(\Ad(P)\ot E_0)\longra\Ga^\iy(\Ad(P)\ot E_1),\\
D^{\nabla_{\Ad(P)}}&:e\longmapsto \sum_{i=0}^d (\id_{\Ad(P)}\ot a_i)\cdot \nabla_{\Ad(P)\ot E_0}^ie,
\end{split}
\label{or1eq2}
\e
where $\nabla_{\Ad(P)\ot E_0}$ are the connections on $\Ad(P)\ot E_0\ot\bigot^iT^*X$ for $0\le i<d$ induced by $\nabla_{\Ad(P)}$ and~$\nabla_{E_0}$.

Since $D^{\nabla_{\Ad(P)}}$ is a linear elliptic operator on a compact manifold $X$, it has finite-dimensional kernel $\Ker(D^{\nabla_{\Ad(P)}})$ and cokernel $\Coker(D^{\nabla_{\Ad(P)}})$, where the {\it index\/} of $D^{\nabla_{\Ad(P)}}$ is $\ind(D^{\nabla_{\Ad(P)}})=\dim\Ker(D^{\nabla_{\Ad(P)}})-\dim\Coker(D^{\nabla_{\Ad(P)}})$. (Unless indicated otherwise, indices and dimensions will be over $\R$.) This index is independent of $\nabla_P\in\A_P$, so we write $\ind^{E_\bu}_P:=\ind(D^{\nabla_{\Ad(P)}})$. The {\it determinant\/} $\det(D^{\nabla_{\Ad(P)}})$ is the 1-dimensional real vector space
\e
\det(D^{\nabla_{\Ad(P)}})=\det\Ker(D^{\nabla_{\Ad(P)}})\ot\bigl(\det\Coker(D^{\nabla_{\Ad(P)}})\bigr)^*,
\label{or1eq3}
\e
where if $V$ is a finite-dimensional real vector space then $\det V=\La^{\dim V}V$.

These operators $D^{\nabla_{\Ad(P)}}$ vary continuously with $\nabla_P\in\A_P$, so they form a family of elliptic operators over the base topological space $\A_P$. Thus as in Knudsen--Mumford \cite{KnMu}, Atiyah--Singer \cite{AtSi6}, and Quillen \cite{Quil}, there is a natural real line bundle $\hat L{}^{E_\bu}_P\ra\A_P$ with fibre $\hat L{}^{E_\bu}_P\vert_{\nabla_P}=\det(D^{\nabla_{\Ad(P)}})$ at each $\nabla_P\in\A_P$. It is equivariant under the actions of $\G_P$ and $\G_P/Z(G)$ on $\A_P$, and so pushes down to real line bundles $L^{E_\bu}_P\ra\B_P$, $\bar L^{E_\bu}_P\ra\ovB_P$ on the topological stacks $\B_P,\ovB_P$, with $L^{E_\bu}_P\cong\Pi_P^*(\bar L_P^{E_\bu})$. We call $L^{E_\bu}_P,\bar L^{E_\bu}_P$ the {\it determinant line bundles\/} of $\B_P,\ovB_P$. The restriction $\bar L^{E_\bu}_P\vert_{\ovB_P^\irr}$ is a topological real line bundle in the usual sense on the topological space~$\ovB_P^\irr$.

Define the {\it orientation bundle\/} $O^{E_\bu}_P$ of $\B_P$ by $O^{E_\bu}_P=(L^{E_\bu}_P\sm 0(\B_P))/(0,\iy)$. That is, we take the complement $L^{E_\bu}_P\sm 0(\B_P)$ of the zero section $0(\B_P)$ in $L^{E_\bu}_P$, and quotient by $(0,\iy)$ acting on the fibres of $L^{E_\bu}_P\sm 0(\B_P)\ra\B_P$ by multiplication. Then $L^{E_\bu}_P\ra\B_P$ descends to $\pi:O^{E_\bu}_P\ra\B_P$, which is a bundle with fibre $(\R\sm\{0\})/(0,\iy)\cong\{1,-1\}=\Z_2$, since $L^{E_\bu}_P\ra\B_P$ is a fibration with fibre $\R$. That is, $\pi:O^{E_\bu}_P\ra\B_P$ is a {\it principal\/ $\Z_2$-bundle}, in the sense of topological stacks. 

Similarly we define a principal $\Z_2$-bundle $\bar\pi:\bar O^{E_\bu}_P\ra\ovB_P$ from $\bar L^{E_\bu}_P$, and as $L^{E_\bu}_P\cong\Pi_P^*(\bar L_P^{E_\bu})$ we have a canonical isomorphism $O^{E_\bu}_P\cong\Pi_P^*(\bar O_P^{E_\bu})$. The fibres of $O^{E_\bu}_P\ra\B_P$, $\bar O^{E_\bu}_P\ra\ovB_P$ are orientations on the real line fibres of $L^{E_\bu}_P\ra\B_P$, $\bar L^{E_\bu}_P\ra\ovB_P$. The restriction $\bar O^{E_\bu}_P\vert_{\ovB^\irr_P}$ is a principal $\Z_2$-bundle on the topological space $\ovB^\irr_P$, in the usual sense.

We say that $(\B_P,E_\bu)$ is {\it orientable\/} if $O^{E_\bu}_P$ is isomorphic to the trivial principal $\Z_2$-bundle $\B_P\t\Z_2\ra\B_P$, and similarly for $(\ovB_P,E_\bu)$ and $\bar O^{E_\bu}_P$. An {\it orientation\/} $\om$ on $(\B_P,E_\bu)$ is an isomorphism $\om:O^{E_\bu}_P\,{\buildrel\cong\over\longra}\,\B_P\t\Z_2$ of principal $\Z_2$-bundles. If $\om$ is an orientation, we write $-\om$ for the opposite orientation. When $E_\bu$ is understood, we usually omit it, and refer just to orientability and orientations of $\B_P$ or $\ovB_P$ rather than $(\B_P,E_\bu)$ or~$(\ovB_P,E_\bu)$.

Since $\Pi_P:\B_P\ra\ovB_P$ is a fibration with fibre $[*/Z(G)]$, which is connected and simply-connected, and $O^{E_\bu}_P\cong\Pi_P^*(\bar O_P^{E_\bu})$, we see that $(\B_P,E_\bu)$ is orientable if and only if $(\ovB_P,E_\bu)$ is, and orientations of $(\B_P,E_\bu)$ and $(\ovB_P,E_\bu)$ correspond.

By characteristic class theory, $\B_P$ is orientable if and only if the first Stiefel--Whitney class $w_1(L^{E_\bu}_P)$ is zero in $H^1(\B_P,\Z_2)$, which may be identified with the equivariant cohomology group $H^1_{\G_P}(\A_P,\Z_2)$. As $\B_P$ is connected, if $\B_P$ is orientable it has exactly two orientations.

We also define the {\it normalized orientation bundle}, or {\it n-orientation bundle}, a principal $\Z_2$-bundle $\check O_P^{E_\bu}\ra\B_P$, by
\e
\check O_P^{E_\bu}=O_P^{E_\bu}\ot_{\Z_2}O_{X\t G}^{E_\bu}\vert_{[\nabla^0]}.
\label{or1eq4}
\e
That is, we tensor the $O_P^{E_\bu}$ with the orientation torsor $O_{X\t G}^{E_\bu}\vert_{[\nabla^0]}$ of the trivial principal $G$-bundle $X\t G\ra X$ at the trivial connection $\nabla^0$. A {\it normalized orientation}, or {\it n-orientation}, of $\B_P$ is an isomorphism $\check\om:\check O^{E_\bu}_P\,{\buildrel\cong\over\longra}\,\B_P\t\Z_2$. There is a natural n-orientation of $\B_{X\t G}$ at~$[\nabla^0]$.

Since we have natural isomorphisms
\e
\Ker(D^{\nabla^0_{\Ad(P)}})\cong\g\ot\Ker D,\qquad 
\Coker(D^{\nabla^0_{\Ad(P)}})\cong\g\ot\Coker D,
\label{or1eq5}
\e
we see that (using an orientation convention) there is a natural isomorphism
\e
L_P^{E_\bu}\vert_{[\nabla^0]}\cong(\det D)^{\ot^{\dim\g}}\ot(\La^{\dim\g}\g)^{\ot^{\ind D}},
\label{or1eq6}
\e
which yields
\e
O_{X\t G}^{E_\bu}\vert_{[\nabla^0]}\cong \Or(\det D)^{\ot^{\dim\g}}\ot_{\Z_2}\Or(\g)^{\ot^{\ind D}},
\label{or1eq7}
\e
where $\Or(\det D),\Or(\g)$ are the $\Z_2$-torsors of orientations on $\det D$ and $\g$. Thus, choosing orientations for $\det D$ and $\g$ gives an isomorphism $\check O_P^{E_\bu}\cong O_P^{E_\bu}$. (But see Remark \ref{or2rem3} for an important technical point about this.) 

N-orientation bundles are convenient because they behave nicely under the Excision Theorem, Theorem \ref{or3thm1} below. Note that $O^{E_\bu}_P$ is trivializable if and only if $\check O_P^{E_\bu}$ is, so for questions of orientability there is no difference.	
\end{dfn}

We can now state the central problem we consider in this paper: 

\begin{prob} In the situation of Definition\/ {\rm\ref{or1def2},} we can ask: 
\begin{itemize}
\setlength{\itemsep}{0pt}
\setlength{\parsep}{0pt}
\item[{\bf(a)}] {\bf(Orientability.)\!} Under what conditions on $X,G,P,E_\bu$ is $\B_P$ orientable?
\item[{\bf(b)}] {\bf(Canonical orientations.)} If\/ $\B_P$ is orientable, then possibly after choo\-sing a small amount of extra data on $X,$ can we construct a natural orientation (or n-orientation) $\om_P$ on~$\B_P$?
\item[{\bf(c)}] {\bf(Relations between canonical orientations.)} Suppose $X$ and\/ $E_\bu$ are fixed, but we consider a family of pairs $(G_i,P_i)$ for $i\in I$. Then there may be natural relations between moduli spaces $\B_{P_i}$ and their orientation bundles $O^{E_\bu}_{P_i},$ which allow us to compare orientations on different $\B_{P_i}$. Can we construct natural orientations (or n-orientations)\/ $\om_{P_i}$ on $\B_{P_i}$ for $i\in I$ as in {\bf(b)\rm,} such that under each relation between moduli spaces $\B_{P_i},$ the $\om_{P_i}$ are related by a sign $\pm 1$ given by an explicit formula?
\end{itemize}

Here is an example of what we have in mind in {\bf(c)}. Consider the family of all principal\/ $\U(m)$-bundles $P\ra X$ for all\/ $m\ge 1$. If\/ $P_1,P_2$ are $\U(m_1)$- and\/ $\U(m_2)$-bundles we can form the direct sum $P_1\op P_2,$ a principal\/ $\U(m_1+m_2)$-bundle. There is a natural morphism $\Phi_{P_1,P_2}:\B_{P_1}\t\B_{P_2}\ra \B_{P_1\op P_2}$ taking direct sums of connections, and we can construct a natural isomorphism
\begin{equation*}
\phi_{P_1,P_2}:O^{E_\bu}_{P_1}\bt_{\Z_2} O^{E_\bu}_{P_2}\longra\Phi_{P_1,P_2}^*\bigl(O^{E_\bu}_{P_1\op P_2}\bigr)
\end{equation*}
of principal\/ $\Z_2$-bundles on $\B_{P_1}\t\B_{P_2}$. Thus, if\/ $\om_{P_1},\om_{P_2},\om_{P_1\op P_2}$ are orientations on $\B_{P_1},\B_{P_2},\B_{P_1\op P_2},$ for some unique $\ep_{P_1,P_2}=\pm 1$ we have
\begin{equation*}
(\phi_{P_1,P_2})_*(\om_{P_1}\bt\om_{P_2})=\ep_{P_1,P_2}\cdot \Phi_{P_1,P_2}^*(\om_{P_1\op P_2}).
\end{equation*}
The aim is to construct orientations\/ $\om_P$ for all\/ $P$ such that\/ $\ep_{P_1,P_2}$ is given by an explicit formula, perhaps involving the Chern classes\/ $c_i(P_1),c_j(P_2)$. In good cases we might just arrange that\/ $\ep_{P_1,P_2}=1$ for all\/~$P_1,P_2$.
\label{or1prob}	
\end{prob}

\begin{rem}{\bf(Orientations on gauge theory moduli spaces.)} We will explain the following in detail in \S\ref{or4}. In gauge theory one studies moduli spaces $\M_P^{\rm ga}$ of (irreducible) connections $\nabla_P$ on a principal bundle $P\ra X$ satisfying some curvature condition, such as moduli spaces of instantons on oriented Riemannian 4-manifolds in Donaldson theory \cite{DoKr}. Under suitable genericity conditions, these moduli spaces $\M_P^{\rm ga}$ will be smooth manifolds.

Problem \ref{or1prob} is important for constructing orientations on such moduli spaces $\M_P^{\rm ga}$. There is a natural inclusion $\io:\M_P^{\rm ga}\hookra\ovB_P$ such that $\io^*(\bar L^{E_\bu}_P)\cong \det T^*\M_P^{\rm ga}$, for an elliptic complex $E_\bu$ on $X$ related to the curvature condition. Hence an orientation on $\ovB_P$, which is equivalent to an orientation on $\B_P$, pulls back under $\io$ to an orientation on $\M_P^{\rm ga}$. We can also use similar ideas to orient moduli spaces of connections $\nabla_P$ plus extra data, such as a Higgs field. 

Thus, constructing orientations as in Problem \ref{or1prob} is an essential part of any programme to define enumerative invariants by `counting' gauge theory moduli spaces, such as Casson invariants of 3-manifolds \cite{BoHe,Taub1}, Donaldson and Seiberg--Witten invariants of 4-manifolds \cite{DoKr,Morg,Nico}, and proposed invariants counting $G_2$-instantons on 7-manifolds with holonomy $G_2$ \cite{DoSe}, or $\Spin(7)$-instantons on 8-manifolds with holonomy $\Spin(7)$ or $\SU(4)$ \cite{BoJo,CaLe1,DoSe}.

There are already various results on Problem \ref{or1prob} in the literature, aimed at orienting gauge theory moduli spaces. The general method was pioneered by Donaldson \cite[Lem.~10]{Dona1}, \cite[\S 3]{Dona2}, \cite[\S 5.4 \& \S 7.1.6]{DoKr}, for moduli of instantons on 4-manifolds. We also mention Taubes \cite[\S 2]{Taub1} for 3-manifolds, Walpuski \cite[\S 6.1]{Walp1} for 7-manifolds, Cao and Leung \cite[\S 10.4]{CaLe1} and Mu\~noz and Shahbazi \cite{MuSh} for 8-manifolds, and Cao and Leung \cite[Th.~2.1]{CaLe2} for $8k$-manifolds.
\label{or1rem1}	
\end{rem}

We will mostly be interested in solving Problem \ref{or1prob} not just for a single choice of $X,G,P,E_\bu$, but for whole classes at once. We make this precise:

\begin{dfn} A {\it Gauge Orientation Problem\/} ({\it GOP\/}) is a problem of the following kind. We consider compact $n$-manifolds $X$ for fixed $n$, equipped with some particular kind of geometric structure $\T$, such that using $\T$ we can define a real elliptic operator $E_\bu$ on $X$ as in Definition \ref{or1def2}. We also choose a family $\cL$ of Lie groups $G$, such as $\cL=\{\SU(m):m=1,2,\ldots\}$. Then we seek to solve Problem \ref{or1prob} for all $X,G,P,E_\bu$ arising from geometric structures $(X,\T)$ of the chosen kind, and Lie groups~$G\in\cL$.

Often we aim to construct canonical (n-)orientations on all such $\B_P$, satisfying compatibility conditions comparing the (n-)orientations for different manifolds $X^+,X^-$ using the Excision Theorem (see Theorem \ref{or3thm1} and Problem~\ref{or3prob}).

\label{or1def3} 	
\end{dfn}

We give some examples of Gauge Orientation Problems:

\begin{ex} Here are some possibilities for $n$, the geometric structure $\T$, and elliptic operator $E_\bu$, which are all Diracians or twisted Diracians: 
\begin{itemize}
\setlength{\itemsep}{0pt}
\setlength{\parsep}{0pt}
\item[{\bf(a)}] Consider compact Riemannian $n$-manifolds $(X,g)$ for any fixed $n$, and let $E_\bu$ be the elliptic operator
\begin{equation*}
\d+\d^*:\Ga^\iy\bigl(\ts\bigop_{i=0}^{[n/2]}\La^{2i}T^*X\bigr)\longra \Ga^\iy\bigl(\ts\bigop_{i=0}^{[(n-1)/2]}\La^{2i+1}T^*X\bigr).
\end{equation*}
\item[{\bf(b)}] Consider compact, oriented Riemannian $n$-manifolds $(X,g)$ for $n=4k$, and let $E_\bu$ be the elliptic operator
\begin{equation*}
\d+\d^*:\Ga^\iy\bigl(\ts\bigop_{i=0}^{k-1}\La^{2i}T^*X\op\La^{2k}_+T^*X\bigr)\longra \Ga^\iy\bigl(\ts\bigop_{i=0}^{k-1}\La^{2i+1}T^*X\bigr),
\end{equation*}
where $\La^{2k}_+T^*X\subset \La^{2k}T^*X$ is the subbundle of $2k$-forms self-dual under the Hodge star.
\item[{\bf(c)}] Consider compact Riemannian $n$-manifolds $(X,g)$ for any fixed $n$ with a spin structure with real spin bundle $S\ra X$, and let $E_\bu$ be the Dirac operator $D:\Ga^\iy(S)\ra\Ga^\iy(S)$.
\item[{\bf(d)}] Consider compact, oriented Riemannian $n$-manifolds $(X,g)$ for $n=4k$ with a spin structure with real spin bundle $S=S_+\op S_-\ra X$, and let $E_\bu$ be the positive Dirac operator $D_+:\Ga^\iy(S_+)\ra\Ga^\iy(S_-)$.
\item[{\bf(e)}] Consider triples $(X,J,g)$ of a compact $n$-manifold $X$ for $n=2k$, an almost complex structure $J$ on $X$, and a Hermitian metric $g$ on $(X,J)$. Let $E_\bu$ be the elliptic operator
\begin{equation*}
\bar\partial+\bar\partial^*:\Ga^\iy\bigl(\ts\bigop_{i=0}^k\La^{0,2i}T^*X\bigr)\longra \Ga^\iy\bigl(\ts\bigop_{i=0}^{k-1}\La^{0,2i+1}T^*X\bigr).
\end{equation*}
\end{itemize}

For example, solving GOP (b) with $n=4$ would give orientations for moduli spaces of anti-self-dual instantons on 4-manifolds \cite{DoKr}. Solving GOP (c) with $n=7$ would give orientations for moduli spaces of $G_2$-instantons, as in \S\ref{or429}. Solving GOP (d) with $n=8$ would give orientations for moduli spaces of $\Spin(7)$-instantons, as in~\S\ref{or4210}. 
\label{or1ex1}	
\end{ex}

In this paper we first collect together in \S\ref{or2} some results and methods for solving Problem \ref{or1prob}. Some of these are new, and some have been used in the literature in particular cases, but we state them in general. Section \ref{or3} discusses techniques for solving Gauge Orientation Problems. Finally, \S\ref{or4} applies the results of \S\ref{or2}--\S\ref{or3} to prove new results on orientability and canonical orientations for interesting families of gauge theory moduli spaces, and reviews the main results of the sequels~\cite{CGJ,JoUp}.

An important motivation for this paper was the first author's new theory \cite{Joyc6} defining vertex algebra structures on the homology $H_*(\M)$ of certain moduli stacks $\M$ in Algebraic Geometry and Differential Geometry. For fixed $X,E_\bu$, an ingredient required in one version of this theory is canonical orientations on moduli stacks $\B_P$ for all principal $\U(m)$-bundles $P\ra X$ and all $m\ge 1$ as in Problem \ref{or1prob}(b), satisfying relations under direct sum as in Problem \ref{or1prob}(c), and the theory dictates the structure of these relations.
\medskip

\noindent{\it Acknowledgements.} This research was partly funded by a Simons Collaboration Grant on `Special Holonomy in Geometry, Analysis and Physics'. The second author was partially supported by JSPS Grant-in-Aid for Scientific Research number JP16K05125. The third author was funded by DFG grants UP 85/2-1 of the DFG priority program SPP 2026 `Geometry at Infinity' and UP 85/3-1. The authors would like to thank Yalong Cao, Aleksander Doan, Simon Donaldson, Sebastian Goette, Andriy Haydys, Vicente Mu\~noz, Johannes Nordstr\"om, Cliff Taubes, Richard Thomas, and Thomas Walpuski for helpful conversations. 

\section{Results and methods for solving Problem \ref{or1prob}}
\label{or2}

\subsection{Remarks on the definitions in \S\ref{or1}}
\label{or21}

Here are some remarks on Definitions \ref{or1def1} and \ref{or1def2}, omitted from \S\ref{or1} for brevity.

\begin{rem}{\bf(i)} There is a theory of topological stacks, due to Metzler \cite{Metz} and Noohi \cite{Nooh1,Nooh2}. Topological stacks form a 2-category $\TopSta$, with homotopy category $\Ho(\TopSta)$. The category of topological spaces $\Top$ has a full and faithful embedding $I:\Top\hookra\Ho(\TopSta)$, so we can consider topological spaces to be examples of topological stacks. There is also a functor $\Pi:\Ho(\TopSta)\ra\Top$ mapping a topological stack $S$ to its {\it coarse moduli space} $S^\coa$ \cite[\S 4.3]{Nooh1}, with $\Pi\ci I\cong\Id_\Top$. Thus, we can regard a topological stack $S$ as a topological space $S^\coa$ with extra structure.

The most important extra structure is {\it isotropy groups}. If $S$ is a topological stack, and $s$ is a point of $S$ (i.e.\ a point of $S^\coa$) we have an isotropy group $\Iso_S(s)$, a topological group, with $\Iso_S(s)=\{1\}$ if $S$ is a topological space. 

If $T$ is a topological space and $H$ a topological group acting continuously on $T$ we can form a {\it quotient stack\/} $[T/H]$ in $\TopSta$, with $[T/H]^\coa$ the quotient topological space $T/H$. Points of $[T/H]$ correspond to $H$-orbits $tH$ in $T$, and the isotropy groups are $\Iso_{[T/H]}(tH)\cong\Stab_H(t)$.

For the quotient topological stacks $\B_P=[\A_P/\G_P]$, $\ovB_P=[\A_P/(\G_P/Z(G))]$ the points are $\G_P$-orbits $[\nabla_P]$ in $\A$, and the isotropy groups are 
\begin{equation*}
\Iso_{\B_P}([\nabla_P])\cong \Stab_{\G_P}(\nabla_P),\quad
\Iso_{\ovB_P}([\nabla_P])\cong\Stab_{\G_P}(\nabla_P)/Z(G).
\end{equation*}
Since $\ovB_P^\irr$ has trivial isotropy groups as a topological stack, it is actually a topological space, and we do not need topological stacks to study $\ovB_P^\irr$. 
\smallskip

\noindent{\bf(ii)} As in Definition \ref{or1def1} the inclusion $\A_P^\irr\hookra\A_P$ is a weak homotopy equivalence, so the inclusions $\B_P^\irr\hookra\B_P$, $\ovB_P^\irr\hookra\ovB_P$ are weak homotopy equivalences of topological stacks in the sense of Noohi \cite{Nooh2}. Also $\Pi_P:\B_P\ra\ovB_P$ identifies orientations on $\B_P$ and $\ovB_P$. Therefore, for the algebraic topological questions that concern us, working on $\ovB^\irr_P$ and on $\B_P$ are essentially equivalent, so we could just restrict our attention to the topological space $\ovB^\irr_P$, and not worry about topological stacks at all, following most other authors in the area.

The main reason we do not do this is that to relate orientations on different moduli spaces we consider direct sums of connections, which are generally reducible, so restricting to irreducible connections would cause problems.
\smallskip

\noindent{\bf(iii)} Here is why we sometimes need $\B_P,\ovB_P$ to be topological stacks rather than topological spaces. We will be studying certain {\it real line bundles\/} $L\ra\ovB_P$. A line bundle $L\ra\ovB_P$ is equivalent to a $\G_P/Z(G)$-equivariant line bundle $L'\ra\A_P$. At each point $[\nabla_P]$ in $\ovB_P$ the fibre $L_{\nabla_P}$ is a 1-dimensional real vector space, and the isotropy group $\Iso_{\ovB_P}([\nabla_P])$ has a natural representation on~$L_{\nabla_P}$.

Under some circumstances, this representation of $\Iso_{\ovB_P}([\nabla_P])$ on $L_{\nabla_P}$ may not be trivial. Then $L$ does not descend to the coarse moduli space $\ovB_P^\coa$. That is, if we consider $\ovB_P=[\A_P/(\G_P/Z(G))]$ as a topological space rather than a topological stack, then the orientation line bundles we are interested in {\it may not exist on the topological space\/} $\ovB_P$, though they are defined on~$\ovB_P^\irr\subset\ovB_P$.
\label{or2rem1}
\end{rem}

\begin{rem}{\bf(i)} Up to continuous isotopy (and hence up to isomorphism), $L^{E_\bu}_P,O^{E_\bu}_P$ and $\bar L^{E_\bu}_P,\bar O^{E_\bu}_P$ in Definition \ref{or1def2} depend on the elliptic operator $D:\Ga^\iy(E_0)\ra\Ga^\iy(E_1)$ up to continuous deformation amongst elliptic operators, and hence only on the symbol $\si(D)$ of~$D$.

This can mean that superficially different geometric problems have isomorphic orientation bundles, or that orientations depend on less data than you think. For example, as in \S\ref{or429}--\S\ref{or4210} orientations for moduli spaces of $G_2$-instantons on a $G_2$-manifold $(X,\vp,g)$, or of $\Spin(7)$-instantons on a $\Spin(7)$-manifold $(X,\Om,g)$, depend only on the underlying compact spin 7- or 8-manifold $X$, not on the $G_2$-structure $(\vp,g)$ or $\Spin(7)$-structure~$(\Om,g)$.
\smallskip

\noindent{\bf(ii)} For orienting moduli spaces of `instantons' in gauge theory, as in \S\ref{or4}, we usually start not with an elliptic operator on $X$, but with an {\it elliptic complex\/}
\e
\xymatrix@C=28pt{ 0 \ar[r] & \Ga^\iy(E_0) \ar[r]^{D_0} & \Ga^\iy(E_1) \ar[r]^(0.55){D_1} & \cdots \ar[r]^(0.4){D_{k-1}} & \Ga^\iy(E_k) \ar[r] & 0. }
\label{or2eq1}
\e
If $k>1$ and $\nabla_P$ is an arbitrary connection on a principal $G$-bundle $P\ra X$ then twisting \eq{or2eq1} by $(\Ad(P),\nabla_{\Ad(P)})$ as in \eq{or1eq2} may not yield a complex (that is, we may have $D^{\nabla_{\Ad(P)}}_{i+1}\ci D^{\nabla_{\Ad(P)}}_i\ne 0$), so the definition of $\det(D_\bu^{\nabla_{\Ad(P)}})$ does not work, though it does work if $\nabla_P$ satisfies the appropriate instanton-type curvature condition. To get round this, we choose metrics on $X$ and the $E_i$, so that we can take adjoints $D_i^*$, and replace \eq{or2eq1} by the elliptic operator
\e
\xymatrix@C=90pt{ \Ga^\iy\bigl(\bigop_{0\le i\le k/2}E_{2i}\bigr) \ar[r]^(0.48){\sum_i(D_{2i}+D_{2i-1}^*)} & \Ga^\iy\bigl(\bigop_{0\le i< k/2}E_{2i+1}\bigr), }
\label{or2eq2}
\e
and then Definition \ref{or1def2} works with \eq{or2eq2} in place of~$E_\bu$.
\label{or2rem2}	
\end{rem}

\begin{rem} In \eq{or1eq4} we defined the n-orientation bundle $\check O_P^{E_\bu}$ in terms of $O_{X\t G}^{E_\bu}\vert_{[\nabla^0]}$, for which we gave a formula in \eq{or1eq7} involving $\Or(\g)$, and said that choosing orientations on $\det D$ and $\g$ gives an isomorphism~$\check O_P^{E_\bu}\cong O_P^{E_\bu}$.

While all this makes sense, for it to be well behaved, {\it we need the orientation on\/ $\g$ to be invariant under the adjoint action of\/ $G$ on\/} $\g$, and this is not true for all Lie groups $G$. For example, if $G={\rm O}(2m)$ and $\ga\in{\rm O}(2m)\sm\SO(2m)$ then $\Ad(\ga)$ is orientation-reversing on $\g$, so no $\Ad(G)$-invariant orientation exists on $\g$. If we restrict to {\it connected\/} Lie groups $G$ then $\Ad(G)$ is automatically orientation-preserving on $\g$, and this problem does not arise.

Let $X$ and $E_\bu$ be as in Definition \ref{or1def2}. Take $P$ to be the trivial principal ${\rm O}(2m)$-bundle over $X$. Consider the topological stack $\B_P$, determinant line bundle $L_P^{E_\bu}\ra\B_P$, and orientation bundle $O_P^{E_\bu}\ra\B_P$ from \S\ref{or1}. The isotropy group of the stack $\B_P$ at $[\nabla^0]$ is $\Iso_{\B_P}([\nabla^0])={\rm O}(2m)$, and its action on $L_P^{E_\bu}\vert_{[\nabla^0]}$ in \eq{or1eq6} is induced by the action of $\Ad(G)$ on $\g$. Thus $\ga\in{\rm O}(2m)\sm\SO(2m)$ acts on $L_P^{E_\bu}\vert_{[\nabla^0]}$ and $O_P^{E_\bu}\vert_{[\nabla^0]}$ by multiplication by $(-1)^{\ind D}$, where~$E_\bu=(E_0,E_1,D)$.

Now suppose $\ind D$ is odd. Then $\ga\in{\rm O}(2m)\sm\SO(2m)$ acts on $O_P^{E_\bu}\vert_{[\nabla^0]}$ by multiplication by $-1$. Any orientation on $\B_P$ must restrict at $[\nabla^0]$ to an ${\rm O}(2m)$-invariant trivialization of $O_P^{E_\bu}\vert_{[\nabla^0]}$. Thus $\B_P$ {\it is not orientable}.
\label{or2rem3}
\end{rem}

\subsection{Elementary results on orientation bundles}
\label{or22}

We now give some results and constructions for orientation bundles $O^{E_\bu}_P$ in Definition \ref{or1def2}, and for answering Problem \ref{or1prob}. Many of these are fairly obvious, or are already used in the references in Remark \ref{or1rem1}, but some are new.

\subsubsection{Simply-connected moduli spaces $\B_P$ are orientable}
\label{or221}

As principal $\Z_2$-bundles on $\B_P$ are trivial if $H^1(\B_P,\Z_2)=0$, we have:

\begin{lem} In Definition\/ {\rm\ref{or1def2},} if\/ $\B_P$ is simply-connected, or more generally if\/ $H^1(\B_P,\Z_2)=0,$ then $\B_P$ is orientable, and n-orientable.
\label{or2lem1}	
\end{lem}

Thus, if we can show $\pi_1(\B_P)=\{1\}$ using algebraic topology, then orientability in Problem \ref{or1prob}(a) follows. This is used in Donaldson \cite[Lem.~10]{Dona1}, \cite[\S 5.4]{DoKr}, Cao and Leung \cite[\S 10.4]{CaLe1}, \cite[Th.~2.1]{CaLe2}, and Mu\~noz and Shahbazi~\cite{MuSh}.

\subsubsection{Standard orientations for trivial connections}
\label{or222}

In Definition \ref{or1def2}, let $P=X\t G$ be the trivial principal $G$-bundle over $X$, and write $\nabla^0\in\A_P$ for the trivial connection. Then \eq{or1eq7} gives a formula for $O_{X\t G}^{E_\bu}\vert_{[\nabla^0]}$. Thus, if we fix an orientation for $\g$ if $\ind D$ is odd, and an orientation for $\det D$ if $\dim\g$ is odd, then we obtain an orientation on $\B_P=\B_{X\t G}$ at the trivial connection $[\nabla^0]$. We will call this the {\it standard orientation}. If $\B_P$ is orientable, the standard orientation determines an orientation on all of~$\B_P$.

\subsubsection{Natural orientations when $G$ is abelian}
\label{or223}

In Definition \ref{or1def2}, suppose the Lie group $G$ is abelian (e.g.\ $G=\U(1)$). Then the adjoint action of $G$ on $\g$ is trivial, so $\Ad(P)\ra X$ is the trivial vector bundle $X\t\g\ra X$, and $\nabla_{\Ad(P)}$ is the trivial connection. Thus as in \eq{or1eq5}
\begin{equation*}
\Ker(D^{\nabla_{\Ad(P)}})=\g\ot\Ker D\quad\text{and}\quad \Coker(D^{\nabla_{\Ad(P)}})=\g\ot\Coker D.
\end{equation*}
Hence as in \eq{or1eq6}, $L_P^{E_\bu}\ra\B_P$ is the trivial line bundle with fibre
\begin{equation*}
(\det D)^{\ot^{\dim\g}}\ot(\La^{\dim\g}\g)^{\ot^{\ind D}},
\end{equation*}
so $\B_P$ is orientable. If we choose an orientation for $\g$ if $\ind D$ is odd, and an orientation for $\det D$ (equivalently, an orientation for $\Ker D\op\Coker D$) if $\dim\g$ is odd, then we obtain a natural orientation on $\B_P$ for any principal $G$-bundle $P\ra X$. Also $\B_P$ has a canonical n-orientation, independent of choices.

\subsubsection{Natural orientations from complex structures on $E_\bu$ or $G$}
\label{or224}

The next theorem is easy to prove, but very useful.

\begin{thm} In Definition\/ {\rm\ref{or1def2},} suppose that\/ $E_0,E_1$ have the structure of complex vector bundles, such that the symbol of\/ $D$ is complex linear. We will call this a \begin{bfseries}complex structure on\end{bfseries} $E_\bu$. Then for any Lie group $G$ and principal\/ $G$-bundle $P\ra X,$ we can define a canonical orientation $\om_P$ and a canonical n-orientation $\check\om_P$ on $\B_P,$ that is, we define trivializations\/ $\om_P:O_P^{E_\bu}\,{\buildrel\cong\over\longra}\,\B_P\t\Z_2$ and\/~$\check\om_P:\check O_P^{E_\bu}\,{\buildrel\cong\over\longra}\,\B_P\t\Z_2$.
\label{or2thm1}
\end{thm}

\begin{proof} As $E_0,E_1$ are complex vector bundles, $\Ga^\iy(E_0),\Ga^\iy(E_1)$ are complex vector spaces. First suppose $D:\Ga^\iy(E_0)\ra\Ga^\iy(E_1)$ is $\C$-linear. Then $\Ad(P)\ot E_0,\Ad(P)\ot E_1$ are also complex vector bundles, and $D^{\nabla_{\Ad(P)}}$ in \eq{or1eq2} is $\C$-linear, so $\Ker(D^{\nabla_{\Ad(P)}})$ and $\Coker(D^{\nabla_{\Ad(P)}})$ are finite-dimensional complex vector spaces. With an appropriate orientation convention, the complex structures induce a natural orientation on $\det(D^{\nabla_{\Ad(P)}})$ in \eq{or1eq3}, which varies continuously with $\nabla_P$ in $\B_P$. This gives a canonical orientation for $\B_P$. To  get a canonical n-orientation, combine the orientations for $\B_P$ and $\B_{X\t G}$ using~\eq{or1eq4}.

If $D$ is not $\C$-linear, though $\si(D)$ is, we can deform $D=D^0$ continuously through elliptic operators $D^t:\Ga^\iy(E_0)\ra\Ga^\iy(E_1)$, $t\in[0,1]$ with symbols $\si(D^t)=\si(D)$ to $D^1$ which is $\C$-linear. As in Remark \ref{or2rem2}(i), the orientation bundle $O^{E^t_\bu}_P$ deforms continuously with $D^t$, so $O^{E_\bu}_P=O^{E^0_\bu}_P\cong O^{E^1_\bu}_P$, and the trivialization of $O^{E^1_\bu}_P$ from $D^1$ complex linear induces a trivialization of $O^{E_\bu}_P$. It is independent of choices, as the space of all $D^t$ with $\si(D^t)=\si(D)$ is an infinite-dimensional affine space, and so contractible, and the subset of $\C$-linear $D^1$ is connected. 
\end{proof}

\begin{ex} Let $(X,g)$ be a compact, oriented Riemannian manifold of dimension $4n+2$, and take $E_\bu$ to be the elliptic operator on $X$
\begin{equation*}
D=\d+\d^*:\Ga^\iy\bigl(\ts\bigop_{i=0}^{2n+1}\La^{2i}T^*X\bigr)\longra\Ga^\iy\bigl(\ts\bigop_{i=0}^{2n}\La^{2i+1}T^*X\bigr).	
\end{equation*}
Using the Hodge star $*$ we can define complex structures on the bundles $E_0=\La^{\rm even}T^*X$, $E_1=\La^{\rm odd}T^*X$ such that the symbol of $D$ is complex linear. So for these $X,E_\bu$ we have canonical (n-)orientations on $\B_P$ for all~$G,P$. 
\label{or2ex1}	
\end{ex}

\begin{ex}{\bf(a)} Let $(X,g)$ be a compact, oriented Riemannian $n$-manifold with a spin structure with real spinor bundle $S\ra X$, and let $E_\bu$ be the Dirac operator $D:\Ga^\iy(S)\ra\Ga^\iy(S)$. If $n\equiv 1,2,3,4$ or 5 mod 8 there is a complex structure on $S$ with the symbol of $D$ complex linear. Also if $n\equiv 6$ mod 8 there is a complex structure on the real spinor bundle $S$ with the symbol of $D$ complex {\it anti\/}-linear, so that $D:\Ga^\iy(S)\ra\Ga^\iy(\bar S)$ is complex linear. Hence for these $X,E_\bu$ we have canonical (n-)orientations on $\B_P$ for all $G,P$.
\smallskip

\noindent{\bf(b)} If $n\equiv 0$ or 4 mod 8 then $S=S_+\op S_-$, and we can take $E_\bu$ to be the positive Dirac operator $D_+:\Ga^\iy(S_+)\ra\Ga^\iy(S_-)$. If $n\equiv 4$ mod 8 there are complex structures on $S_\pm$ with the symbol of $D_+$ complex linear, and again we get canonical (n-)orientations.
\label{or2ex2}	
\end{ex}

See Theorems \ref{or4thm1}, \ref{or4thm2}, \ref{or4thm4} and \ref{or4thm8} for more applications of Theorem~\ref{or2thm1}.

In a similar way, if $G$ is a {\it complex\/} Lie group, such as $\SL(m,\C)$, then $\g$ is a complex vector space, $\Ad(P)$ is a complex vector bundle, and $\nabla_{\Ad(P)}$ is complex linear, so $D^{\nabla_{\Ad(P)}}$ in \eq{or1eq2} is complex linear, and as in Theorem \ref{or2thm1} we obtain a canonical orientation on $\B_P$ for all $X,E_\bu$ and principal $G$-bundles~$P$.

\subsubsection{Another case with natural orientations}
\label{or225}

In Definition \ref{or1def2}, suppose that $E_\bu$ is of the form $E_\bu=\ti E_\bu\op\ti E_\bu^*$, where $\ti E_\bu$ is a real linear elliptic operator on $X$, and $\ti E_\bu^*$ is the formal adjoint of $\ti E_\bu$ under some metrics on $X,\ti E_0,\ti E_1$. Then we have
\begin{align*}
L_P^{E_\bu}&\cong L_P^{\ti E_\bu}\ot_\R L_P^{\ti E_\bu^*}\cong L_P^{\ti E_\bu}\ot_\R (L_P^{\ti E_\bu})^*	\cong \B_P\t\R,\\
O_P^{E_\bu}&\cong O_P^{\ti E_\bu}\ot_{\Z_2} O_P^{\ti E_\bu^*}\cong O_P^{\ti E_\bu}\ot_{\Z_2} (O_P^{\ti E_\bu})^* \cong \B_P\t\Z_2.
\end{align*}
Thus $\B_P$ has a canonical orientation for any principal $G$-bundle $P\ra X$, for any $G$. Since orientation bundles depend only on the symbol of $E_\bu$, and this up to continuous isotopy, this is also true if $E_\bu=\ti E_\bu\op\ti E_\bu^*$ holds only at the level of symbols, up to continuous isotopy.

In \S\ref{or425} we will use this method to show that moduli spaces $\M_P^{\rm VW}$ of solutions to the Vafa--Witten equations on 4-manifolds have canonical orientations.

\subsubsection{\texorpdfstring{Quotienting $X$ by a free $\U(1)$-action}{Quotienting X by a free U(1)-action}}
\label{or226}

As in Definition \ref{or1def2}, let $X$ be a compact $n$-manifold and $E_\bu=(E_0,E_1,D)$ a real elliptic operator on $X$. Suppose the Lie group $\U(1)$ acts freely on $X$, and the action lifts to $E_0,E_1$ making $D$ $\U(1)$-equivariant. Then $Y=X/\U(1)$ is a compact $(n-1)$-manifold, with projection $\pi:X\ra Y$ a principal $\U(1)$-bundle. 

Now $E_\bu$ pushes down to a real elliptic operator $F_\bu=(F_0,F_1,\hat D)$ on $Y$, such that there are natural isomorphisms $E_i\cong \pi^*(F_i)$ for $i=0,1$ inducing isomorphisms $\Ga^\iy(F_i)\cong\Ga^\iy(E_i)^{\U(1)}$ between sections of $F_i$ on $Y$ and $\U(1)$-invariant sections of $E_i$ on $X$, which identify $\hat D:\Ga^\iy(F_0)\ra\Ga^\iy(F_1)$ with $D^{\U(1)}:\Ga^\iy(E_0)^{\U(1)}\ra\Ga^\iy(E_1)^{\U(1)}$. Note that $\hat D$ does not determine $D$, as it has no information on the derivatives in $D$ in the direction of the fibres of~$\pi$.

Let $G$ be a Lie group, and $Q\ra Y$ a principal $G$-bundle. Define $P=\pi^*(Q)$. Then $P\ra X$ is a principal $G$-bundle, with a lift of the $\U(1)$-action on $X$ to $P$. If $\nabla_Q$ is a connection on $Q$ then $\nabla_P=\pi^*(\nabla_Q)$ is a connection on $P$, which is $\U(1)$-equivariant. This defines an injective map $\pi^*:\B_Q\ra\B_P$ of topological stacks. We have orientation bundles $O_P^{E_\bu}\ra\B_P$ and~$O_Q^{F_\bu}\ra\B_Q$.

The next proposition will be used in \cite[Ex.~1.14]{CGJ} to give interesting examples of non-orientable moduli spaces in 8 dimensions.

\begin{prop} In the situation above, there is a natural isomorphism $O_Q^{F_\bu}\cong(\pi^*)^*(O_P^{E_\bu})$ of principal\/ $\Z_2$-bundles on $\B_Q$. Hence, if\/ $\B_P$ is orientable then\/ $\B_Q$ is orientable. Conversely, if\/ $\B_Q$ is not orientable, then\/ $\B_P$ is not orientable.
\label{or2prop1}	
\end{prop}
 
\begin{proof} Let $\nabla_Q$ be a connection on $Q$, and $\nabla_P=\pi^*(\nabla_Q)$. As in \eq{or1eq2}, consider the twisted elliptic operators 
\begin{align*}
D^{\nabla_{\Ad(P)}}&:\Ga^\iy(\Ad(P)\ot E_0)\longra\Ga^\iy(\Ad(P)\ot E_1),\\
\hat D{}^{\nabla_{\Ad(Q)}}&:\Ga^\iy(\Ad(Q)\ot F_0)\longra\Ga^\iy(\Ad(Q)\ot F_1).
\end{align*}
Since $D^{\nabla_{\Ad(P)}}$ is $\U(1)$-equivariant, $\Ker D^{\nabla_{\Ad(P)}}$ and $\Coker D^{\nabla_{\Ad(P)}}$ are finite-dimensional $\U(1)$-representations, and there are natural isomorphisms
\begin{equation*}
\Ker \hat D{}^{\nabla_{\Ad(Q)}}\cong(\Ker D^{\nabla_{\Ad(P)}})^{\U(1)},\quad
\Coker \hat D{}^{\nabla_{\Ad(Q)}}\cong(\Coker D^{\nabla_{\Ad(P)}})^{\U(1)}.
\end{equation*}

There are natural splittings of $\U(1)$-representations
\begin{align*}
\Ker D^{\nabla_{\Ad(P)}}&\cong(\Ker D^{\nabla_{\Ad(P)}})^{\U(1)}\op(\Ker D^{\nabla_{\Ad(P)}})^{\rm nt},\\
\Coker D^{\nabla_{\Ad(P)}}&\cong(\Coker D^{\nabla_{\Ad(P)}})^{\U(1)}\op(\Coker D^{\nabla_{\Ad(P)}})^{\rm nt},
\end{align*}
where $(\cdots)^{\rm nt}$ are nontrivial $\U(1)$-representations (have no trivial component). Now every real nontrivial $\U(1)$-representation has a unique complex vector space structure, such that $e^{i\th}\in\U(1)$ has eigenvalues $e^{ki\th}\in\C$ for $k>0$ only. Thus 
\begin{align*}
\Ker D^{\nabla_{\Ad(P)}}&\cong\Ker \hat D{}^{\nabla_{\Ad(Q)}}\op\text{(complex vector space),}\\
\Coker D^{\nabla_{\Ad(P)}}&\cong\Coker \hat D{}^{\nabla_{\Ad(Q)}}\op\text{(complex vector space).}
\end{align*}
As the complex vector spaces have natural orientations, we obtain natural isomorphisms of $\Z_2$-torsors
\begin{equation*}
(\pi^*)^*(O_P^{E_\bu})\vert_{[\nabla_Q]}=O_P^{E_\bu}\vert_{[\nabla_P]}\cong O_Q^{F_\bu}\vert_{[\nabla_Q]}.
\end{equation*}
These depend continuously on $[\nabla_Q]\in\B_Q$, and so give the required isomorphism $O_Q^{F_\bu}\cong(\pi^*)^*(O_P^{E_\bu})$. The rest of the proposition is immediate.	
\end{proof}

\subsubsection{Orientations on products of moduli spaces}
\label{or227}

Let $X$ and $E_\bu$ be fixed, and suppose $G,H$ are Lie groups, and $P\ra X$, $Q\ra X$ are principal $G$- and $H$-bundles respectively. Then $P\t_XQ$ is a principal $G\t H$ bundle over $X$. There is a natural 1-1 correspondence between pairs $(\nabla_P,\nabla_Q)$ of connections $\nabla_P,\nabla_Q$ on $P,Q$, and connections $\nabla_{P\t_XQ}$ on $P\t_XQ$. This induces an isomorphism of topological stacks $\La_{P,Q}:\B_P\t\B_Q\ra\B_{P\t_XQ}$.

For $(\nabla_P,\nabla_Q)$ and $\nabla_{P\t_XQ}$ as above, there are also natural isomorphisms
\begin{align*}
\Ker(D^{\nabla_{\Ad(P)}})\op \Ker(D^{\nabla_{\Ad(Q)}})&\cong\Ker(D^{\nabla_{\Ad(P\t_XQ)}}),\\
\Coker(D^{\nabla_{\Ad(P)}})\op \Coker(D^{\nabla_{\Ad(Q)}})&\cong\Coker(D^{\nabla_{\Ad(P\t_XQ)}}).\end{align*}
With the orientation conventions described in \cite[\S 3.1.1 \& Prop.~3.5(ii)]{Upme}, these induce a natural isomorphism
\begin{equation*}
\det(D^{\nabla_{\Ad(P)}})\ot \det(D^{\nabla_{\Ad(Q)}})\cong\det(D^{\nabla_{\Ad(P\t_XQ)}}),
\end{equation*}
which is the fibre at $(\nabla_P,\nabla_Q)$ of an isomorphism of line bundles on $\B_P\t\B_Q$
\begin{equation*}
L_P^{E_\bu}\bt L_Q^{E_\bu}\cong
\La_{P,Q}^*(L_{P\t_XQ}^{E_\bu}).
\end{equation*}

This induces an isomorphism of orientation bundles
\begin{equation*}
\la_{P,Q}:O_P^{E_\bu}\bt_{\Z_2} O_Q^{E_\bu}\,{\buildrel\cong\over\longra}\,\La_{P,Q}^*(O_{P\t_XQ}^{E_\bu}).
\end{equation*}
Therefore $\B_{P\t_XQ}$ is orientable if and only if $\B_P,\B_Q$ are both orientable, and then there is a natural correspondence between pairs $(\om_P,\om_Q)$ of orientations for $\B_P,\B_Q$, and orientations $\om_{P\t_XQ}$ for $\B_{P\t_XQ}$.

By exchanging $G,H$ and $P,Q$, we get an isomorphism on $\B_Q\t\B_P$:
\begin{equation*}
\la_{Q,P}:O_Q^{E_\bu}\bt_{\Z_2} O_P^{E_\bu}\,{\buildrel\cong\over\longra}\,
\La_{Q,P}^*(O_{Q\t_XP}^{E_\bu}).
\end{equation*}
Under the natural isomorphisms $\B_P\t\B_Q\cong\B_Q\t\B_P$, $\B_{P\t_XQ}\cong\B_{Q\t_XP}$, using the orientation convention we can show that 
\e
\la_{P,Q}=(-1)^{\ind^{E_\bu}_P\cdot\ind^{E_\bu}_Q}\cdot\la_{Q,P}.
\label{or2eq3}
\e
This gives a {\it commutativity property\/} of the isomorphisms~$\la_{P,Q}$.

If $K$ is another Lie group and $R\ra X$ a principal $K$-bundle, then we have
\e
\begin{split}
&\La_{P\t_XQ,R}\ci(\La_{P,Q}\t\id_{\B_R})=\La_{P,Q\t_XR}\ci(\id_{\B_P}\t\La_{Q,R})
:\\
&\qquad\qquad\B_P\t\B_Q\t\B_R\longra\B_{P\t_XQ\t_XR}.	
\end{split}
\label{or2eq4}
\e
Using this, we can show the following {\it associativity property\/} of the isomorphisms $\la_{P,Q}$ on $\B_P\t\B_Q\t\B_R$, where the sign is trivial:
\ea
&(\La_{P,Q}\t\id_{\B_R})^*(\la_{P\t_XQ,R})\ci(\pi_{\B_P\t\B_Q}^*(\la_{P,Q})\ot\id_{\pi_{\B_R}^*(O_R^{E_\bu})})
\label{or2eq5}\\
&=(\id_{\B_P}\t\La_{Q,R})^*(\la_{P,Q\t_XR})\ci(\id_{\pi_{\B_P}^*(O_P^{E_\bu})}\ot\pi_{\B_Q\t\B_R}^*(\la_{Q,R})):
\nonumber\\
&O_P^{E_\bu}\bt_{\Z_2}O_Q^{E_\bu}\bt_{\Z_2}O_R^{E_\bu}\,{\buildrel\cong\over\longra}
\,\bigl(\La_{P\t_XQ,R}\ci(\La_{P,Q}\t\id_{\B_R})\bigr)^*(O_{P\t_XQ\t_XR}^{E_\bu}).
\nonumber
\ea
Equations \eq{or2eq3} and \eq{or2eq5} are examples of the kind of explicit formula relating orientations referred to in Problem \ref{or1prob}(c). See also~\cite[Prop.~3.5(ii)]{Upme}.

The analogue of the above also works for n-orientation bundles.

\subsubsection{Relating moduli spaces for discrete quotients $G\twoheadrightarrow H$}
\label{or228}

Suppose $G$ is a Lie group, $K\subset G$ a discrete (closed and dimension zero) normal subgroup, and set $H=G/K$ for the quotient Lie group. Let $X,E_\bu$ be fixed. If $P\ra X$ is a principal $G$-bundle, then $Q:=P/K$ is a principal $H$-bundle over $X$. Each $G$-connection $\nabla_P$ on $P$ induces a natural $H$-connection $\nabla_Q$ on $Q$, and the map $\nabla_P\mapsto\nabla_Q$ induces a natural morphism $\De_P^Q:\B_P\ra\B_Q$ of topological stacks, which is an isomorphism. If $\nabla_P,\nabla_Q$ are as above then the natural isomorphism $\g\cong\h$ induces an isomorphism $\Ad(P)\cong\Ad(Q)$ of vector bundles on $X$, which identifies the connections $\nabla_{\Ad(P)},\nabla_{\Ad(Q)}$. Hence the twisted elliptic operators $D^{\nabla_{\Ad(P)}},D^{\nabla_{\Ad(Q)}}$ are naturally isomorphic, and so are their determinants \eq{or1eq3}. This easily gives canonical isomorphisms
\begin{equation*}
L_P^{E_\bu}\cong(\De_P^Q)^*(L_Q^{E_\bu}) \quad\text{and}\quad \de_P^Q:O_P^{E_\bu}\,{\buildrel\cong\over\longra}\,(\De_P^Q)^*(O_Q^{E_\bu}),
\end{equation*}
which induce a 1-1 correspondence between orientations on $\B_P,\B_Q$. For example, we can apply this when $G=\SU(2)$ and~$H=\SO(3)=\SU(2)/\{\pm 1\}$.

Note however that not every principal $H$-bundle $Q\ra X$ need come from a principal $G$-bundle $P\ra X$ by $Q\cong P/K$. For example, a principal $\SO(3)$-bundle $Q\ra X$ lifts to a principal $\SU(2)$-bundle $P\ra X$ if and only if the second Stiefel-Whitney class $w_2(Q)$ is zero.

\begin{ex} Take $G=\SU(m)\t\U(1)$, and define $K\subset G$ by
\begin{equation*}
K=\bigl\{(e^{2\pi ik/m}\Id_m,e^{-2\pi ik/m}):k=1,\ldots,m\bigr\}\cong\Z_m.
\end{equation*}
Then $K$ lies in the centre $Z(G)$, so is normal in $G$, and $H=G/K\cong\U(m)$. To see this, note that the morphism $G=\SU(m)\t\U(1)\ra\U(m)=H$ mapping $(A,e^{i\th})\mapsto e^{i\th}A$ is surjective with kernel~$K$.

For fixed $X,E_\bu$, let $P\ra X$ be a principal $\SU(m)$-bundle, and write $P'=X\t\U(1)\ra X$ for the trivial $\U(1)$-bundle over $X$. Set $P''=P\t_XP'$ for the associated principal $\SU(m)\t\U(1)$-bundle over $X$, and define $Q=P''/K$ for the quotient principal $\U(m)$-bundle. We now have isomorphisms of moduli spaces
\begin{equation*}
\smash{\xymatrix@C=60pt{ \B_P\t\B_{P'} \ar[r]^(0.45){\La_{P,P'}} & \B_{P\t_XP'}=\B_{P''} \ar[r]^(0.55){\De_{P''}^Q} & \B_Q, }}
\end{equation*}
and isomorphisms of orientation bundles
\begin{equation*}
\smash{\xymatrix@C=60pt{ O_P^{E_\bu}\bt_{\Z_2} O_{P'}^{E_\bu} \ar[r]^{\la_{P,P'}} &
\La_{P,P'}^*(O_{P''}^{E_\bu}) \ar[r]^{\La_{P,P'}^*(\de_{P''}^Q)} & (\De_{P''}^Q\ci\La_{P,P'})^*(O_Q^{E_\bu}),
}}
\end{equation*}
where $\La_{P,P'},\la_{P,P'}$ are as in \S\ref{or227}, and $\De_{P''}^Q,\de_{P''}^Q$ are as above. 

As $P'$ is the trivial $\U(1)$-bundle, it carries the trivial connection $\nabla^0_{P'}$. Fixing an orientation for $\det D$, as in \S\ref{or222} we have the standard orientation for $\B_{P'}$ at $[\nabla_{P'}^0]$, giving an isomorphism $\si_{P'}:\Z_2\ra O_{P'}^{E_\bu}\vert_{[\nabla^0_{P'}]}$. Thus, we have a morphism 
\begin{equation*}
\Ka_P^Q:\B_P\longra\B_Q,\qquad \Ka_P^Q:[\nabla_P]\longmapsto \De_{P''}^Q\ci \La_{P,P'}\bigl([\nabla_P],[\nabla^0_{P'}]\bigr),
\end{equation*}
and an isomorphism of orientation bundles
\begin{equation*}
\ka_P^Q:=\La_{P,P'}^*(\de_{P''}^Q)\ci \la_{P,P'}\ci(\id\bt\si_{P'}):O_P^{E_\bu}\longra (\Ka_P^Q)^*(O_Q^{E_\bu}).	
\end{equation*}
Hence orientations for the $\U(m)$-bundle moduli space $\B_Q$ induce orientations for the $\SU(m)$-bundle moduli space $\B_P$. Our conclusion is:
\begin{quotation}
\noindent{\it For fixed\/ $X,E_\bu,$ if we have orientability, or canonical orientations, on $\B_Q$ for all principal\/ $\U(m)$-bundles $Q\ra X,$ then we have orientability, or canonical orientations, on $\B_P$ for all\/ $\SU(m)$-bundles $P\ra X$. The analogue holds for n-orientations.}	
\end{quotation}
Example \ref{or2ex5} will give a kind of converse to this.
\label{or2ex3}	
\end{ex}

The method of the next proposition was used by Donaldson and Kronheimer \cite[\S 5.4.3]{DoKr} for simply-connected 4-manifolds~$X$.

\begin{prop} Let\/ $X,E_\bu$ be fixed as in Definition\/ {\rm\ref{or1def2},} and suppose that for all principal\/ $\U(2)$-bundles $Q\ra X,$ the moduli space $\B_Q$ is orientable. Then for all principal\/ $\SO(3)$-bundles $P\ra X$ such that\/ $w_2(P)\in H^2(X,\Z_2)$ lies in the image of\/ $H^2(X,\Z)\ra H^2(X,\Z_2),$ the moduli space $\B_P$ is orientable. This holds for all\/ $\SO(3)$-bundles $P\ra X$ if\/ $H^3(X,\Z)$ has no $2$-torsion.

The analogue holds with\/ $\U(2),\SO(3)$ replaced by\/ $\Spinc(n),\SO(n),$ $n\ge 2$.
\label{or2prop2}
\end{prop}

\begin{proof} We apply the above construction with $G=\U(2)$, $K=\{\pm 1\}\subset\U(2)$, and $H=\U(2)/\{\pm 1\}\cong\SO(3)\t\U(1)$. Let $Q\ra X$ be a principal $\U(2)$-bundle. Then $R=Q/\{\pm 1\}$ is a principal $\SO(3)\t\U(1)$-bundle $R\ra X$. Hence there are principal $\SO(3)$- and $\U(1)$-bundles $P,S\ra X$ with $R\cong P\t_XS$. We now have isomorphisms of moduli spaces
\begin{equation*}
\smash{\xymatrix@C=60pt{ \B_P\t\B_S \ar[r]^(0.45){\La_{P,S}} & \B_{P\t_XS}\cong \B_R \ar[r]^(0.55){\De_R^Q} & \B_Q, }}
\end{equation*}
and isomorphisms of orientation bundles
\begin{equation*}
\smash{\xymatrix@C=55pt{ 
O_P^{E_\bu}\bt_{\Z_2} O_S^{E_\bu} \ar[r]^{\la_{P,S}} &
\La_{P,S}^*(O_R^{E_\bu}) \ar[r]^(0.45){\La_{P,S}^*(\de_R^Q)} & (\De_R^Q\ci\La_{P,S})^*(O_Q^{E_\bu}).
}}
\end{equation*}
By assumption $O_Q^{E_\bu}$ is orientable. Restricting to a point of $\B_S$ in the above equations, we see that $\B_P$ is orientable.

Since $\U(2)\cong\Spinc(3)$, it is known from the theory of $\Spinc$-structures that an $\SO(3)$-bundle $P\ra X$ extends to a $\U(2)$-bundle $Q\ra X$ as above if and only if the second Stiefel--Whitney class $w_2(P)\in H^2(X,\Z_2)$ lies in the image of $H^2(X,\Z)\ra H^2(Z,\Z_2)$, since $w_2(P)$ must be the image of $c_1(Q)$. The exact sequence $0\ra\Z\,{\buildrel 2\cdot\over\longra}\,\Z\ra\Z_2\ra 0$ gives a long exact sequence in cohomology
\begin{equation*}
\xymatrix@C=20pt{ \cdots \ar[r] & H^2(X,\Z) \ar[r] & H^2(X,\Z_2) \ar[r] & H^3(X,\Z) \ar[r]^{2\cdot} & H^3(X,\Z) \ar[r] & \cdots. }  
\end{equation*}
This implies that $H^2(X,\Z)\ra H^2(X,\Z_2)$ is surjective if and only if $H^3(X,\Z)$ has no 2-torsion. The same arguments work with $\U(2),\SO(3)$ replaced by $\Spinc(n)$ and $\SO(n)$. The proposition follows.
\end{proof}

The analogues of all the above also work for n-orientation bundles.

\subsubsection{Relating moduli spaces for Lie subgroups $G\subset H$}
\label{or229}

Let $X,E_\bu$ be fixed, and let $H$ be a Lie group and $G\subset H$ a Lie subgroup, with Lie algebras $\g\subset\h$. If $P\ra X$ is a principal $G$-bundle, then $Q:=(P\t H)/G$ is a principal $H$-bundle over $X$. Each $G$-connection $\nabla_P$ on $P$ induces a natural $H$-connection $\nabla_Q$ on $Q$, and the map $\nabla_P\mapsto\nabla_Q$ induces a natural morphism $\Xi_P^Q:\B_P\ra\B_Q$ of topological stacks. Thus, we can try to compare the line bundles $L_P^{E_\bu},(\Xi_P^Q)^*(L_Q^{E_\bu})$ on $\B_P$, and the principal $\Z_2$-bundles~$O_P^{E_\bu},(\Xi_P^Q)^*(O_Q^{E_\bu})$.

Write $\m=\h/\g$, and $\rho:G\ra\Aut(\m)$ for the representation induced by the adjoint representation of $H\supset G$. Then we have an exact sequence 
\e
\xymatrix@C=30pt{ 0 \ar[r] & \Ad(P) \ar[r] & \Ad(Q) \ar[r] & \rho(P)=(P\t\m)/G \ar[r] & 0 }	
\label{or2eq6}
\e
of vector bundles on $X$, induced by $0\ra\g\ra\h\ra\m\ra 0$. If $\nabla_P,\nabla_Q$ are as above, we have connections $\nabla_{\Ad(P)},\nabla_{\Ad(Q)},\nabla_{\rho(P)}$ on $\Ad(P),\Ad(Q),\rho(P)$ compatible with \eq{or2eq6}. Twisting $E_\bu$ by $\Ad(P),\ab\Ad(Q),\ab\rho(P)$ and their connections and taking determinants, as in \cite[Prop.~3.5(ii)]{Upme} we define an isomorphism
\begin{equation*}
\det(D^{\nabla_{\Ad(P)}})\ot \det(D^{\nabla_{\rho(P)}})\cong \det(D^{\nabla_{\Ad(Q)}}),
\end{equation*}
which is the fibre at $\nabla_P$ of an isomorphism of line bundles on $\B_P$
\e
L_P^{E_\bu}\ot L_{P,\rho}^{E_\bu}\cong (\Xi_P^Q)^*(L_Q^{E_\bu}),
\label{or2eq7}
\e
where $L_{P,\rho}^{E_\bu}\ra\B_P$ is the determinant line bundle associated to the family of elliptic operators $\nabla_P\mapsto D^{\nabla_{\rho(P)}}$ on $\B_P$. We will write $\ind_{P,\rho}^{E_\bu}:=\ind(D^{\nabla_{\rho(P)}})$ for the index of these operators, which is independent of~$[\nabla_P]\in\B_P$.

Now suppose that we can give $\m$ the structure of a {\it complex\/} vector space, such that $\rho:G\ra\Aut(\m)$ is complex linear. This happens if $H/G$ has an (almost) complex structure homogeneous under $H$. Then as in \S\ref{or224} for complex $G$, we can define a natural orientation on $L_{P,\rho}^{E_\bu}$, so taking orientations in \eq{or2eq7} gives a natural isomorphism of principal $\Z_2$-bundles on~$\B_P$:
\e
\xi_P^Q:O_P^{E_\bu}\,{\buildrel\over\longra}\, (\Xi_P^Q)^*(O_Q^{E_\bu}).
\label{or2eq8}
\e
An easy special case is if $\m=0$, e.g. for $\SO(m)\subset{\rm O}(m)$, when $L_{P,\rho}^{E_\bu}$ is trivial.

This gives a method for proving orientability in Problem \ref{or1prob}(a). Suppose we can show that $H^1(\B_Q,\Z_2)=0$, using homotopy-theoretic properties of $X,H$. Then $\B_Q$ is orientable by Lemma \ref{or2lem1}, so \eq{or2eq8} shows that $\B_P$ is orientable, even if $H^1(\B_P,\Z_2)\ne 0$. The method is used by Donaldson \cite[Lem.~10]{Dona1}, \cite[\S 5.4.2]{DoKr}, and by Mu\~noz and Shahbazi \cite{MuSh} using the inclusion of Lie groups $\SU(9)/\Z_3\subset E_8$. Here are some examples of suitable~$G\subset H$:

\begin{ex} We have an inclusion $G=\U(m_1)\t\U(m_2)\subset\U(m_1+m_2)=H$ for $m_1,m_2\ge 1$, with $\u(m_1+m_2)/(\u(m_1)\op\u(m_2))=\m\cong\C^{m_1}\ot_\C\ov{\C^{m_2}}$, where $G=\U(m_1)\t\U(m_2)$ acts on $\C^{m_1}\ot_\C\ov{\C^{m_2}}$ via the usual representations of $\U(m_1),\U(m_2)$ on $\C^{m_1},\C^{m_2}$, with $\ov{\C^{m_2}}$ the complex conjugate of $\C^{m_2}$, so the representation $\rho$ is complex linear.

Suppose $X,E_\bu$ are fixed, and $P_1\ra X$, $P_2\ra X$ are principal $\U(m_1)$- and $\U(m_2)$-bundles. Define a principal $\U(m_1+m_2)$-bundle $P_1\op P_2\ra X$ by 
\begin{equation*}
P_1\op P_2=(P_1\t_XP_2\t\U(m_1+m_2))/(\U(m_1)\t\U(m_2)).
\end{equation*}
Then combining the material of \S\ref{or227} for the product of $\U(m_1),\U(m_2)$ with the above, we have a morphism
\e
\Phi_{P_1,P_2}:=\Xi_{P_1\t_XP_2}^{P_1\op P_2}\ci\La_{P_1,P_2}:\B_{P_1}\t\B_{P_2}\longra\B_{P_1\op P_2},
\label{or2eq9}
\e
and a natural isomorphism of principal $\Z_2$-bundles on~$\B_{P_1}\t\B_{P_2}$:
\e
\begin{split}
&\phi_{P_1,P_2}=\La_{P_1,P_2}^*(\xi_{P_1\t_XP_2}^{P_1\op P_2})\ci\la_{P_1,P_2}:\\
&\qquad O_{P_1}^{E_\bu}\bt_{\Z_2} O_{P_2}^{E_\bu}\,{\buildrel\cong\over\longra}\,
\Phi_{P_1,P_2}^*(O_{P_1\op P_2}^{E_\bu}).
\end{split}
\label{or2eq10}
\e

As for \eq{or2eq3}--\eq{or2eq5}, we can consider commutativity and associativity properties of the isomorphisms $\phi_{P_1,P_2}$. For commutativity, under the natural isomorphisms $\B_{P_1}\t\B_{P_2}\cong \B_{P_2}\t\B_{P_1}$, $\B_{P_1\op P_2}\cong\B_{P_2\op P_1}$ we have
\e
\begin{split}
\Phi_{P_2,P_1}&\cong\Phi_{P_1,P_2}, \\
\phi_{P_2,P_1}&\cong(-1)^{\ind^{E_\bu}_{P_1}\cdot\ind^{E_\bu}_{P_2}}\cdot (-1)^{\frac{1}{2}\ind^{E_\bu}_{P_1\t_XP_2,\rho}}
\cdot\phi_{P_1,P_2}.
\end{split}
\label{or2eq11}
\e
Here the first sign $(-1)^{\ind^{E_\bu}_{P_1}\cdot\ind^{E_\bu}_{P_2}}$ in \eq{or2eq11} comes from \eq{or2eq3}, and exchanges the $\U(m_1)\t\U(m_2)$-bundle $P_1\t_XP_2$ with the $\U(m_2)\t\U(m_1)$-bundle~$P_2\t_XP_1$.

The second sign $(-1)^{\frac{1}{2}\ind^{E_\bu}_{P_1\t_XP_2,\rho}}$ in \eq{or2eq11} comes in as $\phi_{P_1,P_2},\phi_{P_2,P_1}$ in \eq{or2eq10} depend on choices of complex structure on 
\begin{equation*}
\m_{1,2}\!=\!\u(m_1+m_2)/(\u(m_1)\!\op\u(m_2))\;\>\text{and}\;\>\m_{2,1}\!=\!\u(m_2+m_1)/(\u(m_2)\!\op\!\u(m_1)).
\end{equation*} 
Under the natural isomorphism $\m_{1,2}\cong\m_{2,1}$, these complex structures are {\it complex conjugate}, as $\m_{1,2}\cong\C^{m_1}\ot_\C\ov{\C^{m_2}}$, $\m_{2,1}\cong\C^{m_2}\ot_\C\ov{\C^{m_1}}\cong\ov{\C^{m_1}}\ot_\C\C^{m_2}$. Because of this, under the natural isomorphism $L_{P_1\t_XP_2,\rho_{12}}^{E_\bu}\cong L_{P_2\t_XP_1,\rho_{21}}^{E_\bu}$, the orientations on $L_{P_1\t_XP_2,\rho_{12}}^{E_\bu},L_{P_2\t_XP_1,\rho_{21}}^{E_\bu}$ used to define $\phi_{P_1,P_2},\ab\phi_{P_2,P_1}$ differ by a factor of $(-1)^{\ind_\C(D^{\nabla_{\rho(P_1\t_XP_2)}})}$, regarding $D^{\nabla_{\rho(P_1\t_XP_2)}}$ as a complex elliptic operator. As $\ind_\C(D^{\nabla_{\rho(P_1\t_XP_2)}})=\ha\ind_\R(D^{\nabla_{\rho(P_1\t_XP_2)}})=\frac{1}{2}\ind^{E_\bu}_{P_1\t_XP_2,\rho}$, equation \eq{or2eq11} follows.

For associativity, if $P_3\ra X$ is a principal $\U(m_3)$-bundle then we have
\begin{align*}
&\Phi_{P_1\op P_2,P_3}\ci(\Phi_{P_1,P_2}\t\id_{\B_{P_3}})=\Phi_{P_1,P_2\op P_3}\ci(\id_{\B_{P_1}}\t\Phi_{P_2,P_3}):\\
&\qquad\qquad\B_{P_1}\t\B_{P_2}\t\B_{P_3}\longra\B_{P_1\op P_2\op P_3},	
\end{align*}
as for \eq{or2eq4}, and then as for \eq{or2eq5}, we have
\ea
&(\Phi_{P_1,P_2}\t\id_{\B_{P_3}})^*(\phi_{P_1\op P_2,P_3})\ci(\phi_{P_1,P_2}\bt\id_{O_{P_3}^{E_\bu}})
\label{or2eq12}\\
&=(\id_{\B_{P_1}}\t\Phi_{P_2,P_3})^*(\phi_{P_1,P_2\op P_3})\ci(\id_{O_{P_1}^{E_\bu}}\bt\phi_{P_2,P_3}):
\nonumber\\
&O_{P_1}^{E_\bu}\bt_{\Z_2}O_{P_2}^{E_\bu}\bt_{\Z_2}O_{P_3}^{E_\bu}\,{\buildrel\cong\over\longra}\,\bigl(\Phi_{P_1\op P_2,P_3}\ci(\Phi_{P_1,P_2}\t\id_{\B_{P_3}})\bigr)^*(O_{P_1\op P_2\op P_3}^{E_\bu}).
\nonumber
\ea
The sign is trivial as there is no sign in \eq{or2eq5}, and the natural isomorphism $\m_{12,3}\op\m_{1,2}\cong\m_{1,23}\op\m_{2,3}$ is complex linear.

The analogue of the above also holds for n-orientation bundles, giving isomorphisms
$\check\phi_{P_1,P_2}:\check O_{P_1}^{E_\bu}\bt_{\Z_2}\check O_{P_2}^{E_\bu}\ra
\Phi_{P_1,P_2}^*(\check O_{P_1\op P_2}^{E_\bu})$ satisfying \eq{or2eq11} and~\eq{or2eq12}.
\label{or2ex4}	
\end{ex}

\begin{rem} The analogue of Example \ref{or2ex4} does not work for the families of groups ${\rm O}(m),\SO(m),\Spin(m)$ or $\Sp(m)$, since for example under the inclusion $\SO(m_1)\t\SO(m_2)\hookra\SO(m_1+m_2)$, there is no $\SO(m_1)\t\SO(m_2)$-invariant complex structure on $\m=\so(m_1+m_2)/(\so(m_1)\op\so(m_2))$ unless $m_1=2$ or $m_2=2$. So the theory of \S\ref{or25} below works only for the unitary groups.
\label{or2rem4}	
\end{rem}

\begin{ex} Define an inclusion $\U(m)\hookra\SU(m+1)$ by mapping
\begin{equation*}
A\longmapsto \begin{pmatrix} A & \begin{matrix} 0 \\ \vdots \\ 0 \end{matrix} \\ \begin{matrix} 0 & \cdots & 0 \end{matrix} & (\det A)^{-1} \end{pmatrix},\qquad A\in\U(m).	
\end{equation*}
There is an isomorphism $\m=\su(m+1)/\u(m)\cong\C^m$, such that $A\in\U(m)$ acts on $\m\cong\C^m$ by $A:\bs x\mapsto \det A\cdot A\bs x$, which is complex linear on~$\m$. 

For fixed $X,E_\bu$, let $P\ra X$ be a principal $\U(m)$-bundle, and $Q=(P\t\SU(m+1))/\U(m)$ the associated principal $\SU(m+1)$-bundle. Then as above we have a morphism of moduli spaces $\Xi_P^Q:\B_P\ra\B_Q$ and an isomorphism of orientation bundles $\xi_P^Q:O_P^{E_\bu}\ra(\Xi_P^Q)^*(O_Q^{E_\bu})$. Hence orientations for the $\SU(m+1)$-bundle moduli space $\B_Q$ induce orientations for the $\U(m)$-bundle moduli space $\B_P$. In a converse to Example \ref{or2ex3}, our conclusion is:
\begin{quotation}
\noindent{\it For fixed\/ $X,E_\bu,$ if we have orientability, or canonical orientations, on $\B_Q$ for all principal\/ $\SU(m+1)$-bundles $Q\ra X,$ then we have orientability, or canonical orientations, on $\B_P$ for all principal\/ $\U(m)$-bundles $P\ra X$. The analogue holds for n-orientations.}	
\end{quotation}
\label{or2ex5}	
\end{ex}

\begin{ex} Define an inclusion $\U(m)\hookra\Sp(m)$ by mapping complex matrices to quaternionic matrices using the inclusion $\C=\an{1,i}_\R\hookra\H=\an{1,i,j,k}_\R$. There is an isomorphism of $\m=\sp(m)/\u(m)$ with the complex vector space of $m\t m$ complex symmetric matrices $B$, such that $A\in\U(m)$ acts on $\m$ by $A:B\mapsto ABA^t$, which is complex linear on $\m$. 

For fixed $X,E_\bu$, let $P\ra X$ be a $\U(m)$-bundle, and $Q=(P\t\Sp(m))/\U(m)$ the associated $\Sp(m)$-bundle. Then as above we have a morphism of moduli spaces $\Xi_P^Q:\B_P\ra\B_Q$ and an isomorphism of orientation bundles $\xi_P^Q:O_P^{E_\bu}\ra(\Xi_P^Q)^*(O_Q^{E_\bu})$. Hence orientations for the $\Sp(m)$-bundle moduli space $\B_Q$ induce orientations for the $\U(m)$-bundle moduli space $\B_P$. Our conclusion is:
\begin{quotation}
\noindent{\it For fixed\/ $X,E_\bu,$ if we have orientability, or canonical orientations, on $\B_Q$ for all principal\/ $\Sp(m)$-bundles $Q\ra X,$ then we have orientability, or canonical orientations, on $\B_P$ for all principal\/ $\U(m)$-bundles $P\ra X$. The analogue holds for n-orientations.}	
\end{quotation}
The converse is false. In \S\ref{or429}--\S\ref{or4210} we discuss problems in which orientability holds for $\U(m)$-bundles, but fails for $\Sp(m)$-bundles.
\label{or2ex6}
\end{ex}

\begin{ex} We have an inclusion $\SO(m)\t\SO(2)\subset\SO(m+2)$ for $m\ge 1$. There is a natural identification $\so(m+2)/(\so(m)\op\so(2))=\m\cong\R^m\ot_\R\R^2$, where $G=\SO(m)\t\SO(2)$ acts on $\R^m\ot_\R\R^2$ by the tensor product of the obvious representations of $\SO(m),\SO(2)$ on $\R^m,\R^2$. Identifying $\R^2\cong\C$ and $\SO(2)\cong\U(1)$ gives $\m\cong\R^m\ot_\R\C=\C^m$, where $\rho$ is complex linear.
\label{or2ex7}	
\end{ex}

\begin{ex} We have an inclusion $G=\Sp(m)\t\U(1)\subset\Sp(m+1)=H$ for $m\ge 1$, by combining $\U(1)\subset\Sp(1)$ and $\Sp(m)\t\Sp(1)\subset\Sp(m+1)$. There is a natural identification $\sp(m+1)/(\sp(m)\op\u(1))=\m\cong\H^m\op\C$, where $G=\Sp(m)\t\U(1)$ acts on $\H^m\op\C$ by
\begin{equation*}
\rho(A,e^{i\th}):(\bs x,y)\longmapsto(A\bs xe^{-i\th},ye^{-2i\th})	
\end{equation*}
for $A\in\Sp(m)$, $e^{i\th}\in\U(1)$, $\bs x\in\H^m$ and $y\in\C$, regarding $A,\bs x,e^{i\th}$ as $m\t m$, $m\t 1$ and $1\t 1$ matrices over $\H$ to define $A\bs xe^{i\th}$. Identifying $\H^m\cong\C^{2m}$ using right multiplication by $i\in\H$, we see that $\rho$ is complex linear on~$\m\cong\C^{2m+1}$. 
\label{or2ex8}	
\end{ex}

\subsection{Background on K-theory and homotopy theory}
\label{or23}

We explain some Algebraic Topology background material needed in~\S\ref{or24}--\S\ref{or25}.

\subsubsection{Background on K-theory}
\label{or231}

We briefly summarize some notation and results from topological K-theory. Some references are Atiyah \cite{Atiy}, Karoubi \cite{Karo} and Switzer~\cite[\S 11]{Swit}.

Let $X$ be a compact topological space. Write $K^0(X)$ for the abelian group generated by isomorphism classes $\lb F \rb$ of complex vector bundles $F\ra X$ (which may have different ranks on different components of $X$) with the relation that $\lb F\op G\rb=\lb F\rb +\lb G\rb $ in $K^0(X)$ for all complex vector bundles $F,G\ra X$. If $f:X\ra Y$ is continuous, define a group morphism $K^0(f):K^0(Y)\ra K^0(X)$ with $K^0(f):\lb F\rb \mapsto\lb f^*(F)\rb $ for all $F\ra Y$. This defines a functor $K^0:\Top^{\bf cpt}\ra\mathop{\bf AbGp}$.

If $P\ra X$ is a principal $\U(m)$-bundle, it has an associated complex vector bundle $F\ra X$ with fibre $\C^m$ given by $F=(P\t\C^m)/\U(m)$. We write $\lb P\rb=\lb F\rb$ in~$K^0(X)$.

If $(X,x)$ is a compact topological space with base-point $x\in X$, define $\ti K^0(X,x)=\Ker \bigl(K^0(x):K^0(X)\ra K^0(*)\bigr)$, regarding $x$ as a map $*\ra X$. We can make any topological space $X$ into a space with basepoint $(X\amalg\{\iy\},\iy)$ by adding a disjoint extra point $\iy$, and then~$K^0(X)\cong \ti K^0(X\amalg\{\iy\},\iy)$.

Define $K^{-n}(X)=\ti K^0\bigl(S^n(X\amalg\{\iy\},\iy)\bigr)$ for $n=1,2,\ldots,$ where $S^n(-)$ is the $n$-fold suspension of pointed topological spaces. Then {\it Bott periodicity\/} gives canonical isomorphisms $K^{-n}(X)\cong K^{-n-2}(X)$, so we can extend to $K^n(X)$ for $n\in\Z$ periodic of period 2. Reducing from $\Z$ to $\Z_2$, we have the {\it complex K-theory\/} $K^*(X)=K^0(X)\op K^1(X)$, graded over~$\Z_2$.

Write $\ul\C^k$ for the trivial vector bundle $X\t\C^k\ra X$. Write $\boo_X\in K^0(X)$ for the class $\lb\ul\C\rb$, so that $\lb\ul\C^k\rb=k\,\boo_X$. Tensor product induces a product $\cdot:K^0(X)\t K^0(X)\ra K^0(X)$ with $\lb F\rb \cdot\lb G\rb =\lb F\ot_\C G\rb $, which extends to a graded product on $K^*(X)$, and is commutative and associative with identity $\boo_X$, making $K^*(X)$ into a $\Z_2$-graded commutative ring.
All this is contravariantly functorial under continuous maps~$f:X\ra Y$.

The Chern character gives isomorphisms
\begin{align*}
\Ch^0:K^0(X)\ot_\Z\Q\longra H^{\rm even}(X,\Q),\quad
\Ch^1:K^1(X)\ot_\Z\Q\longra H^{\rm odd}(X,\Q).
\end{align*}
There is an {\it Atiyah--Hirzebruch spectral sequence\/} $H^{i+2\Z}(X,\Z)\Ra K^i(X)$, which can be used to compute $K^*(X)$.

Now suppose $X$ is a compact, connected manifold. Then there is a morphism $\rank:K^0(X)\ra\Z$ mapping $\lb F\rb \mapsto\rank F$. If $\al\in K^0(X)$ with $2\rank\al\ge\dim X$ (the {\it stable range\/}) then there exists a complex vector bundle $F\ra X$ with $\lb F\rb=\al$ in $K^0(X)$, which is unique up to (noncanonical) isomorphism. Choosing a metric $h$ on the fibres of $F$ gives a principal $\U(\rank\al)$-bundle $P\ra X$ with $\lb P\rb=\al$, also unique up to isomorphism.

Instead of working with complex vector bundles $F\ra X$, we can work with real vector bundles, giving {\it real K-theory\/} $KO^*(X)$, or with quaternionic vector bundles, yielding {\it quaternionic K-theory\/} $KSp^*(X)$. In these cases {\it Bott periodicity\/} gives isomorphisms $KO^{-n}(X)\cong KO^{-n-8}(X)$, $KSp^{-n}(X)\cong KSp^{-n-8}(X)$, so $KO^*(X),KSp^*(X)$ are both graded over $\Z_8$. There are canonical isomorphisms $KSp^n(X)\cong KO^{n+4}(X)$ for all $X,n$. Here $KO^*(X)$ is a $\Z_8$-graded commutative ring, but we do not have a natural graded product on $KSp^*(X)$, as there is no good notion of tensor product of quaternionic vector bundles.

\subsubsection{Homotopy and the category $\Topho$}
\label{or232}

Continuous maps $f_0,f_1:X\ra Y$ are called {\it homotopic\/} if there is a continuous $h:X\t[0,1]\ra Y$ with $h(x,0)=f_0(x)$ and $h(x,1)=f_1(x)$. Writing $f_t(x)=h(x,t)$, this means there is a continuous family
$(f_t:X\ra Y)_{t\in[0,1]}$ interpolating between $f_0$ and $f_1$. Homotopy is an equivalence relation on $f:X\ra Y$.

Write $\Topho$ for the category with objects topological spaces $X,Y$ and morphisms homotopy equivalence classes $[f]:X\ra Y$ of continuous maps $f:X\ra Y$. Then $f:X\ra Y$ is a {\it homotopy equivalence\/} in $\Top$ if $[f]:X\ra Y$ is an isomorphism in~$\Topho$.

We will deal with constructions which yield a space $X$ unique up to homotopy equivalence, or a map $f:X\ra Y$ unique up to homotopy. These are conveniently stated in the category $\Topho$: $X$ is unique up to isomorphism (or perhaps canonical isomorphism) in $\Topho$, and $[f]:X\ra Y$ is unique in~$\Topho$.

\subsubsection{Background on classifying spaces}
\label{or233}

We summarize some well known material which can be found in Milnor and Stasheff \cite{MiSt}, May \cite[\S\S 16.5, 23, 24]{May2}, and Husem\"oller et al.\ \cite[Part II]{HJJS}. Let $G$ be a topological group. A {\it classifying space\/} for $G$ is a topological space $BG$ and a principal $G$-bundle $\pi:EG\ra BG$ such that $EG$ is contractible. Classifying spaces exist for any $G$, and are unique up to homotopy equivalence, so $BG$ is unique up to (canonical) isomorphism in $\Topho$. There is a functorial construction for them involving simplicial complexes.  

Classifying spaces have the property that if $X$ is a paracompact topological space and $P\ra X$ is a principal $G$-bundle, then there exists a continuous map $f_P:X\ra BG$, unique up to homotopy, so that $[f_P]:X\ra BG$ is unique in $\Topho$, and an isomorphism $P\cong f_P^*(EG)$ of principal $G$-bundles on $X$. Thus, there is a 1-1 correspondence between the set of isomorphism classes of principal $G$-bundles on $X$, and $[X,BG]$. Here for topological spaces $X,Y$ we write $[X,Y]=\pi_0\bigl(\Map_{C^0}(X,Y)\bigr)$ for the set of homotopy classes of maps $X\ra Y$, the connected components of $\Map_{C^0}(X,Y)$.

To prove this property, note that $(P\t EG)/G\ra X$ is a bundle with fibre $EG$, which is contractible. Sections $X\ra (P\t EG)/G$ of this bundle project to maps $X\ra EG/G=BG$ as required. A bundle with contractible fibre over a paracompact space $X$ has sections, and any two sections are homotopic.

Any $\al\in H^k(BG,\Z)$ defines a characteristic class $c_\al$ of principal $G$-bundles $P\ra X$ by $c_\al(P)=H^k(f_P)(\al)\in H^k(X,\Z)$. Thus, by computing $H^*(BG)$ we learn about topological invariants of $G$-bundles. A morphism of topological groups $G\ra H$ induces a morphism $BG\ra BH$, unique up to homotopy.

One choice for $B\U(1)$ is the infinite projective space $\CP^\iy$, and for $B\U(m)$ is the infinite Grassmannian $\mathop{\rm Gr}(\C^m,\C^\iy)$. The inclusion $\U(m)\hookra\U(m+1)$ by $A\mapsto\bigl(\begin{smallmatrix} A & 0 \\ 0 & 1 \end{smallmatrix}\bigr)$ induces a morphism $B\U(m)\ra B\U(m+1)$. Write $B\U$ for the (homotopy) direct limit $B\U=\varinjlim_{m\ra\iy}B\U(m)$. Then $B\U\t\Z$ {\it is the classifying space for complex K-theory}. That is, as in May \cite[p.~204-5]{May2}, for compact topological spaces $X$ there is a natural bijection
\e
K^0(X)\cong [X,B\U\t\Z]=\pi_0\bigl(\Map_{C^0}(X,B\U\t\Z)\bigr).
\label{or2eq13}
\e
More generally, if $X$ is noncompact, it is usual to take \eq{or2eq13} to be the definition of $K^0(X)$. We can also define higher K-theory groups as the higher homotopy groups $K^{-k}(X)=\pi_k\bigl(\Map_{C^0}(X,B\U\t\Z)\bigr)$. So by Bott periodicity $K^k(X)\cong K^{k+2}(X)$ we have
\e
K^1(X)\cong\pi_1\bigl(\Map_{C^0}(X,B\U\t\Z)\bigr).
\label{or2eq14}
\e

Define $\Pi_m:B\U(m)\ra B\U\t\Z$ for $m\ge 0$ to map $B\U(m)\ra B\U$ from the direct limit $B\U=\varinjlim_{m\ra\iy}B\U(m)$, and to map $B\U(m)\ra m\in\Z$. Then if a principal $\U(m)$-bundle $P\ra X$ corresponds to $f_P:X\ra B\U(m)$, its K-theory class $\lb P\rb\in K^0(X)$ corresponds to~$\Pi_m\ci f_P:X\ra B\U\t\Z$.

It is known (this is essentially equivalent to the `stable range' in \S\ref{or231}) that the morphisms $B\U(m)\ra B\U(m+1)$ and $B\U(m)\ra B\U$ induce isomorphisms on homotopy groups $\pi_k(-)$ when $k\le 2m$ and are surjective for $\pi_{2m+1}(-)$. It follows that if $X$ is a compact $n$-manifold then
\e
\begin{gathered}
\pi_k(\Pi_m\ci):\pi_k\bigl(\Map_{C^0}(X,B\U(m))\bigr)\longra\pi_k\bigl(\Map_{C^0}(X,B\U\t\Z)\bigr)\\
\text{is an isomorphism if $0<k\le 2m-n$,}
\end{gathered}
\label{or2eq15}
\e
where we exclude $k=0$ because of the $\Z$ factor in~$B\U\t\Z$.

The inclusion $\U(m)\t\U(m')\ra\U(m+m')$ mapping $(A,B)\mapsto\bigl(\begin{smallmatrix} A & 0 \\ 0 & B \end{smallmatrix}\bigr)$ induces a morphism $\mu_{m,m'}:B\U(m)\t B\U(m')\ra B\U(m+m')$. We interpret this in terms of direct sums: if $P\ra X$, $Q\ra X$ are principal $\U(m)$, $\U(m')$-bundles corresponding to $f_P:X\ra B\U(m)$, $f_Q:X\ra B\U(m')$ then $\mu_{m,m'}\ci(f_P,f_Q):X\ra B\U(m+m')$ corresponds to the principal $\U(m+m')$-bundle~$P\op Q\ra X$.

Let $\mu=\varinjlim_{m,m'\ra\iy}\mu_{m,m'}:B\U\t B\U\ra B\U$. Then $\mu$ is homotopy commutative and associative, and makes $B\U$ into an H-space. We can also define $\mu':(B\U\t\Z)\t(B\U\t\Z)\ra B\U\t\Z$ as the product of $\mu:B\U\t B\U\ra B\U$ and $+:\Z\t\Z\ra\Z$. Then $\mu'$ induces the operation of addition on $K^0(X)\cong[X,B\U\t\Z]$, which comes from direct sum of vector bundles.

Similarly $B{\rm O}\t\Z$, $B\Sp\t\Z$ are the classifying spaces for $KO^0(-),KSp^0(-)$, where $B{\rm O}=\varinjlim_{m\ra\iy}B{\rm O}(m)$ and~$B\Sp=\varinjlim_{m\ra\iy}B\Sp(m)$.

\subsubsection{Homotopy theory of topological stacks}
\label{or234}

In \S\ref{or24} we will deal with the topological stacks $\B_P=[\A_P/\G_P]$ of Definition \ref{or1def1} in a homotopy-theoretic way. Fortunately, Noohi \cite{Nooh2} provides a homotopy theory for topological stacks. As $\B_P=[\A_P/\G_P]$ is a global quotient with $\A_P,\G_P$ paracompact, it is {\it hoparacompact\/} in the sense of \cite[Def.~8.5]{Nooh2}.

Noohi defines a {\it classifying space\/} for a hoparacompact topological stack $S$ to be a paracompact topological space $S^\cla$ with a morphism $\pi^\cla:S^\cla\ra S$ in $\Ho(\TopSta)$ which is `parashrinkable' (i.e.\ in a weak sense a fibration with contractible fibres). Then $S^\cla$ is homotopy equivalent to $S$ in the category $\Ho(\TopSta_{\bf hp})$ of hoparacompact topological spaces.

Classifying spaces are functorial: there is a functor $(-)^\cla:\TopSta_{\bf hp}^{\bf ho}\ra\Top_{\bf pa}^{\bf ho}$ from the category of hoparacompact topological stacks with homotopy classes of 1-morphisms, to the category of paracompact topological spaces with homotopy classes of morphisms, which is right adjoint to the inclusion $\Top_{\bf pa}^{\bf ho}\hookra\TopSta_{\bf hp}^{\bf ho}$, with $\pi^\cla$ the unit of the adjunction. We also write $(-)^\cla$ for the composition $\Ho(\TopSta_{\bf hp})\ra\TopSta_{\bf hp}^{\bf ho}\,{\buildrel(-)^\cla\over\longra}\,\Top_{\bf pa}^{\bf ho}\hookra\Topho$.

The moral is that we can replace the topological stacks $\B_P$ with their classifying spaces $\B_P^\cla$, and then apply homotopy theory of topological spaces to $\B_P^\cla$ in the usual way. A possible model for $\B_P^\cla$ is $\B_P^\cla=(\A_P\t E\G_P)/\G_P$, so that $\pi^\cla:\B_P^\cla\ra\B_P$ is a genuine fibration with contractible fibre $E\G_P$.

\subsubsection{On pullbacks of principal $\Z_2$-bundles by homotopic morphisms}
\label{or235}

\begin{rem}
\label{or2rem5}
Let $[f]:X\ra Y$ be a morphism in $\Topho$, and $\pi:P\ra Y$ a principal $\Z_2$-bundle. Can we make sense of a `pullback $\Z_2$-bundle'~$[f]^*(P)\ra X$?

Let $f_0,f_1:X\ra Y$ represent $[f]$. Then there is a homotopy $(f_t)_{t\in[0,1]}$ interpolating between $f_0,f_1$. We have two principal $\Z_2$-bundles $f_0^*(P)\ra X$, $f_1^*(P)\ra X$, which are joined by a continuous path of $\Z_2$-bundles $f_t^*(P)$ for $t\in[0,1]$. Parallel translation along this path induces an isomorphism $f_0^*(P)\cong f_1^*(P)$ of principal $\Z_2$-bundles on $X$. Hence if $f$ represents $[f]$ then $f^*(P)\ra X$ depends only on $[f],P$ {\it up to isomorphism\/} of principal $\Z_2$-bundles. But when does $f^*(P)$ depend only on $[f],P$ {\it up to canonical isomorphism}?

If $P\ra Y$ is a trivializable principal $\Z_2$-bundle, it is easy to show that the isomorphism $f_0^*(P)\cong f_1^*(P)$ above is independent of the choice of homotopy $(f_t)_{t\in[0,1]}$, so $[f]^*(P):=f^*(P)$ is natural up to canonical isomorphism. However, as in Example \ref{or2ex9} below, if $P$ is nontrivial then $f_0^*(P)\cong f_1^*(P)$ may depend on $(f_t)_{t\in[0,1]}$, so $[f]^*(P)$ is {\it not\/} natural up to canonical isomorphism.

We have two answers to the problems this raises in our theory:
\begin{itemize}
\setlength{\itemsep}{0pt}
\setlength{\parsep}{0pt}
\item[(a)] As in Problem \ref{or1prob}, here and in the sequels \cite{JoUp,CGJ}, given an orientation $\Z_2$-bundle such as $O^{E_\bu}_P\ra\B_P$ or $O_\al^{E_\bu}\ra\cC_\al$ below, we first want to know if $O^{E_\bu}_P,O_\al^{E_\bu}$ are trivializable (so $O^{E_\bu}_P,O_\al^{E_\bu}$ matter only up to isomorphism), and secondly, if it is trivializable we want to know about canonical isomorphisms (so $O^{E_\bu}_P,O_\al^{E_\bu}$ matter up to canonical isomorphism).

So in practice, if an orientation bundle $P$ is not trivializable, we do not care that pullbacks $[f]^*(P)$ are not natural up to canonical isomorphism. 
\item[(b)] Actually, in the situation of \S\ref{or24}--\S\ref{or25}, there is always additional structure allowing us to specify pullbacks $[f]^*(P)$ uniquely up to canonical isomorphism. So the problem goes away. We illustrate this in Example~\ref{or2ex10}. 
\end{itemize}
\end{rem}

\begin{ex} Let $X=*$ be the point, $Y=\cS^1=\bigl\{z\in\C:\md{z}=1\bigr\}$, and $P\ra Y$ be the nontrivial $\Z_2$-bundle, and $f_0,f_1:X\ra Y$ be $f_0(*)=1$, $f_1(*)=-1$. Consider the homotopies 
$(f_t)_{t\in[0,1]}$, $(f'_t)_{t\in[0,1]}$ given by $f_t(*)=e^{i\pi t}$, $f'_t(*)=e^{-i\pi t}$. As the monodromy of $P$ around $\cS^1$ is $-1$, we see that $(f_t)_{t\in[0,1]}$ and $(f'_t)_{t\in[0,1]}$ induce different isomorphisms~$f_0^*(P)\cong f_1^*(P)$.
\label{or2ex9}	
\end{ex}

\begin{ex}{\bf(a)} Let $X$ be a topological space and $P\ra X$ a principal $G$-bundle. As in \S\ref{or233}, there exists a map $f_P:X\ra BG$, unique up to homotopy, with an isomorphism $P\cong f_P^*(EG)$. Choices of $f_P$ correspond to sections $s\in\Ga\bigl((P\t EG)/G\bigr)$ of the bundle 
$(P\t EG)/G\ra X$, with fibre~$EG$.

Let $f_P^0,f_P^1$ be choices of $f_P$, corresponding to $s^0,s^1\in\Ga\bigl((P\t EG)/G\bigr)$. As $EG$ is contractible, $\Ga\bigl((P\t EG)/G\bigr)$ is connected, so there is a path $(s^t)_{t\in[0,1]}$ from $s^0$ to $s^1$ in $\Ga\bigl((P\t EG)/G\bigr)$, giving a homotopy $(f^t_P)_{t\in[0,1]}$ from $f_P^0$ to $f_P^1$. But we can say more: $\Ga\bigl((P\t EG)/G\bigr)$ is contractible, so two such paths $(s^t)_{t\in[0,1]},$ $(\ti s^t)_{t\in[0,1]}$ are joined by a homotopy of paths, and the corresponding $(f^t_P)_{t\in[0,1]}$, $(\ti f^t_P)_{t\in[0,1]}$ are joined by a homotopy of homotopies from $f^0_P$ to $f^1_P$. Thus, any isomorphism of principal $\Z_2$-bundles constructed using a homotopy $(f^t_P)_{t\in[0,1]}$ defined using $(s^t)_{t\in[0,1]}$ is independent of the choice of~$(f^t_P)_{t\in[0,1]}$.
\smallskip

\noindent{\bf(b)} In \S\ref{or233} we defined $\mu':(B\U\t\Z)\t(B\U\t\Z)\ra B\U\t\Z$, which is commutative and associative up to homotopy, so $[\mu']$ is commutative and associative in $\Topho$, and makes $B\U\t\Z$ into a commutative, associative H-space.

It is natural to enhance the H-space structure on $B\U\t\Z$ to a $\Ga$-{\it space}, as in Segal \cite[\S 1]{Sega}, or more-or-less equivalently an $E_\iy$-{\it space}, as in May \cite{May1}. The $\Ga$- or $E_\iy$-space structures provide choices of homotopies in $\Top$ lifting all the abelian group identities of $[\mu']$ in $\Topho$, and these homotopies are natural up to homotopies of homotopies. Thus, any isomorphism of principal $\Z_2$-bundles constructed using homotopies realizing commutativity or associativity of $[\mu']$ can be made independent of choices.
\label{or2ex10}	
\end{ex}

\subsection{\texorpdfstring{$\U(m)$-bundles and the mapping space $\Map_{C^0}(X,B\U\t\Z)$}{U(m)-bundles and the mapping space Mapᶜ⁰(X,BU x ℤ)}}
\label{or24}

Let $X$ and $E_\bu$ be fixed. We now explain a useful framework for studying orientations on $\B_P$ simultaneously for all principal $\U(m)$-bundles $P\ra X$ and all $m\ge 1$, using the complex K-theory groups $K^0(X),K^1(X)$ and the topological mapping space $\cC:=\Map_{C^0}(X,B\U\t\Z)$. This can then be used to study orientations on $\B_Q$ for all principal $\SU(m)$-bundles $Q\ra X$. Parts of the theory also work for $\SO(m)$-bundles and $\Sp(m)$-bundles.

\subsubsection{The Euler form $\chi^{E_\bu}:K^0(X)\t K^0(X)\ra\Z$}
\label{or241}

\begin{dfn}
\label{or2def1}
Work in the situation of Definition \ref{or1def2}, and use the K-theory notation of \S\ref{or231}. Let $P_1\ra X$, $P_2\ra X$ be principal $\U(m_1)$- and $\U(m_2)$-bundles with $\lb P_1\rb=\al$ and $\lb P_2\rb=\be$ in $K^0(X)$. Choose connections $\nabla_{P_1},\nabla_{P_2}$ on $P_1,P_2$. Let $\rho_{12}:\U(m_1)\t\U(m_2)\ra\Aut_\C(\C^{m_1}\ot_\C\overline{\C^{m_2}})$ be the tensor product of the usual representation of $\U(m_1)$ on $\C^{m_1}$ and the complex conjugate representation of $\U(m_2)$ on $\C^{m_2}$. We have a vector bundle $\rho_{12}(P_1\t_XP_2)=(P_1\t_XP_2 \t\C^{m_1}\ot_\C\overline{\C^{m_2}})/\U(m_1)\t\U(m_2)$ over $X$ with fibre $\C^{m_1}\ot_\C\overline{\C^{m_2}}$, with a connection $\nabla_{\rho_{12}(P_1\t_XP_2)}$ induced by $\nabla_{P_1},\nabla_{P_2}$. Thus as in Definition \ref{or1def2} we may form the twisted complex elliptic operator
\begin{equation*}
D^{\nabla_{\rho_{12}(P_1\t_XP_2)}}:\Ga^\iy(\rho_{12}(P_1\t_XP_2)\ot E_0)\longra\Ga^\iy(\rho_{12}(P_1\t_XP_2)\ot E_1).
\end{equation*}

The complex index $\ind_\C(D^{\nabla_{\rho_{12}(P_1\t_XP_2)}})$ is independent of $\nabla_{P_1},\nabla_{P_2}$, and so depends only on $X,E_\bu,\al,\be$. Replacing $P_1$ or $P_2$ by a direct sum $P_1'\op P_1''$, $P_2'\op P_2''$ gives a direct sum of the corresponding elliptic operators. Hence $\ind_\C(D^{\nabla_{\rho_{12}(P_1\t_XP_2)}})$ is biadditive in $\al,\be$. Therefore there exists a unique biadditive map $\chi^{E_\bu}:K^0(X)\t K^0(X)\ra\Z$ which we call the {\it Euler form}, such that $\ind_\C(D^{\nabla_{\rho_{12}(P_1\t_XP_2)}})=\chi^{E_\bu}(\al,\be)$ for all $P_1,P_2,\al,\be$ as above.

Swapping round $P_1,P_2$, $m_1,m_2$ and $\al,\be$ replaces $\C^{m_1}\ot\overline{\C^{m_2}}$ by $\C^{m_2}\ot\overline{\C^{m_1}}$, and so complex conjugates $\rho_{12}(P_1\t_XP_2)$ and $D^{\nabla_{\rho_{12}(P_1\t_XP_2)}}$, which does not change the index. Hence $\chi^{E_\bu}(\al,\be)=\chi^{E_\bu}(\be,\al)$ for all~$\al,\be\in K^0(X)$.

When $P_1=P_2=P$ and $\al=\be$, we have $\rho_{12}(P\t_XP)\cong\Ad(P)\ot_\R\C$, so the complex index $\ind_\C(D^{\nabla_{\rho_{12}(P\t_XP)}})$ equals the real index $\ind_\R(D^{\nabla_{\Ad(P)}})$. Therefore for $\ind^{E_\bu}_P$ as in Definition \ref{or1def2}, we have
\e
\ind^{E_\bu}_P=\chi^{E_\bu}(\lb P\rb,\lb P\rb).
\label{or2eq16}
\e

If $P_1\ra X$ and $P_2\ra X$ are principal $\U(m_1)$- and $\U(m_2)$-bundles with $\lb P_1\rb=\al$ and $\lb P_2\rb=\be$, then Example \ref{or2ex4} defines an isomorphism $\phi_{P_1,P_2}:O_{P_1}^{E_\bu}\bt_{\Z_2} O_{P_2}^{E_\bu}\ra
\Phi_{P_1,P_2}^*(O_{P_1\op P_2}^{E_\bu})$, and \eq{or2eq11} relates $\phi_{P_2,P_1},\phi_{P_1,P_2}$ under the isomorphism $\B_{P_2}\t\B_{P_1}\cong\B_{P_1}\t\B_{P_2}$. The first sign in \eq{or2eq11} is written in terms of $\chi^{E_\bu}$ by \eq{or2eq16}, and for the second we have
\begin{equation*}
\ha\ind^{E_\bu}_{P_1\t_XP_2,\rho}=\ha\ind_\R(D^{\nabla_{\rho_{12}(P_1\t_XP_2)}})=\ind_\C(D^{\nabla_{\rho_{12}(P_1\t_XP_2)}})=\chi^{E_\bu}(\al,\be).
\end{equation*}
Hence \eq{or2eq11} may be rewritten
\e
\phi_{P_2,P_1}=(-1)^{\chi^{E_\bu}(\al,\be)+\chi^{E_\bu}(\al,\al)\chi^{E_\bu}(\be,\be)}\cdot\phi_{P_1,P_2}.
\label{or2eq17}
\e	
\end{dfn}

\subsubsection{The mapping spaces $\cC,\cC_\al$ and orientation bundles $O_\al^{E_\bu}$}
\label{or242}

\begin{dfn}
\label{or2def2}
Let $X$ be a compact, connected manifold of dimension $n$, and use the notation of \S\ref{or23}. Write $\cC=\Map_{C^0}(X,B\U\t\Z)$ for the topological space of continuous maps $X\ra B\U\t\Z$, with the compact-open topology. Equation \eq{or2eq13} identifies the set $\pi_0(\cC)$ of path-connected components of $\cC$ with $K^0(X)$. Write $\cC_\al$ for the connected component of $\cC$ corresponding to $\al\in K^0(X)$ under \eq{or2eq13}, so that $\cC=\coprod_{\al\in K^0(X)}\cC_\al$.

Define $\Phi:\cC\t\cC\ra\cC$ by $\Phi:(f,g)\mapsto\mu'\ci(f,g)$, where $\mu':(B\U\t\Z)^2\ra B\U\t\Z$ is as in \S\ref{or233}. Then $\Phi$ is natural, commutative, and associative, up to homotopy, as $\mu'$ is, so $[\Phi]$ is natural, commutative, and associative in $\Topho$. Write $\Phi_{\al,\be}=\Phi\vert_{\cC_\al\t\cC_\be}:\cC_\al\t\cC_\be\ra\cC_{\al+\be}$ for $\al,\be\in K^0(X)$. Then the following diagrams commute up to homotopy for all $\al,\be,\ga\in K^0(X)$, where $\si:\cC_\al\t\cC_\be\ra\cC_\be\t\cC_\al$ exchanges the two factors: 
\ea
\begin{gathered}
\xymatrix@C=90pt@R=15pt{ *+[r]{\cC_\al\t\cC_\be} \ar@{}[dr]^\simeq \ar@/^1pc/[drr]^(0.6){\Phi_{\al,\be}} \ar[d]^{\si} \\
*+[r]{\cC_\be\t\cC_\al} \ar[rr]^(0.4){\Phi_{\be,\al}} && *+[l]{\cC_{\al+\be},\!} 
}
\end{gathered}
\label{or2eq18}\\
\begin{gathered}
\xymatrix@C=160pt@R=15pt{ *+[r]{\cC_\al\t\cC_\be\t\cC_\ga} \ar@{}[dr]^(0.6)\simeq \ar[r]_(0.43){\id_{\cC_\al}\t\Phi_{\be,\ga}} \ar[d]^{\Phi_{\al,\be}\t\id_{\cC_\ga}} &
*+[l]{\cC_\al\t\cC_{\be+\ga}} \ar[d]_{\Phi_{\al,\be+\ga}}\\
*+[r]{\cC_{\al+\be}\t\cC_\ga} \ar[r]^(0.7){\Phi_{\al+\be,\ga}} & *+[l]{\cC_{\al+\be+\ga}.\!} 
}
\end{gathered}
\label{or2eq19}
\ea

Now let $P\ra X$ be a principal $\U(m)$-bundle. Choose a classifying space $\pi^\cla:\B_P^\cla\ra\B_P$ for the topological stack $\B_P$, as in \S\ref{or234}. There is a universal principal $\U(m)$-bundle $U_P=(P\t\A_P)/\G_P\ra X\t\B_P$, so $(\id_X\t\pi^\cla)^*(U_P)\ra X\t\B_P^\cla$ is a principal $\U(m)$-bundle over a paracompact topological space. As in \S\ref{or233} this corresponds to some $f_P:X\t\B_P^\cla\ra B\U(m)$. Write $\Si_P:\B_P^\cla\ra\Map_{C^0}(X,B\U(m))$ for the corresponding map. Then $f_P,\Si_P$ are unique up to homotopy, so $[\Si_P]$ is unique in~$\Topho$.

Connected components of $\Map_{C^0}(X,B\U(m))$ correspond to isomorphism classes $[Q]$ of principal $\U(m)$-bundles $Q\ra X$. Write $\Map_{C^0}(X,B\U(m))_{[P]}$ for the connected component corresponding to $[P]$. Using the arguments of Donaldson--Kronheimer \cite[Prop.~5.1.4]{DoKr} and Atiyah--Bott \cite[Prop.~2.4]{AtBo}, we see that $\Si_P:\B_P^\cla\ra\Map_{C^0}(X,B\U(m))_{[P]}$
is a homotopy equivalence.

Define $\Si_P^\cC:\B_P^\cla\ra\cC$ by $\Si_P^\cC:b\mapsto \Pi_m\ci\Si_P(b)$ for $\Pi_m:B\U(m)\ra B\U\t\Z$ as in \S\ref{or233}. Then $\Si_P^\cC$ maps $\B_P^\cla\ra\cC_\al$, where $\al=\lb P\rb\in K^0(X)$. Equation \eq{or2eq15} and $\Si_P$ a homotopy equivalence yield
\e
\pi_k(\Si_P^\cC):\pi_k(\B_P^\cla)\longra\pi_k(\cC_\al)\quad \text{is an isomorphism if $0\le k\le 2m-n$,}
\label{or2eq20}
\e
where the case $k=0$ is trivial as $\B_P^\cla,\cC_\al$ are connected.

Suppose $Q\ra X$ is a principal $\U(m')$-bundle with $\lb Q\rb=\be\in K^0(X)$, so $P\op Q\ra X$ is a principal $U(m+m')$-bundle with $\lb P\op Q\rb=\al+\be$, and \eq{or2eq9} defines a morphism $\Phi_{P,Q}:\B_P\t\B_Q\ra\B_{P\op Q}$. Then the following diagram commutes up to homotopy:
\e
\begin{gathered}
\xymatrix@!0@C=145pt@R=40pt{
*+[r]{\B_P^\cla\t\B_Q^\cla} \ar@<1.5ex>@/^.7pc/[rr]^{\Si_P^\cC\t\Si_Q^\cC} \ar@<-1ex>@{}[dr]^\simeq \ar[d]^{\Phi_{P,Q}^\cla} \ar[r]_(0.44){\Si_P\t\Si_Q} & {\begin{subarray}{l}\ts \Map_{C^0}(X,B\U(m))_{[P]} \t \\ \ts \Map_{C^0}(X,B\U(m'))_{[Q]}\end{subarray}} \ar@<-1ex>@{}[dr]^\simeq\ar[r]_(0.55){(\Pi_m\ci)\t(\Pi_{m'}\ci)} \ar[d]^(0.55){\mu_{m,m'}\ci} & *+[l]{\cC_\al\t\cC_\be} \ar[d]_{\Phi_{\al,\be}=\mu'\ci}
\\
*+[r]{\B^\cla_{P\op Q}} \ar@<-1ex>@/_.5pc/[rr]_{\Si_{P\op Q}^\cC} \ar[r]^(0.33){\Si_{P\op Q}} & \Map_{C^0}(X,B\U(m\!+\!m'))_{[P\op Q]} \ar[r]^(0.6){\Pi_{m+m'}\ci} & *+[l]{\cC_{\al+\be},}
}
\end{gathered}
\label{or2eq21}
\e
where the left hand square commutes as $\Phi_{P,Q},\mu_{m,m'}$ both come from direct sums, the right hand square commutes as $\mu=\varinjlim_{m,m'\ra\iy}\mu_{m,m'}$ in \S\ref{or233}, and the semicircles commute by definition of~$\Si_P^\cC,\Si_Q^\cC,\Si_{P\op Q}^\cC$.
\end{dfn}

\begin{dfn}
\label{or2def3}
Continue in the situation of Definition \ref{or2def2}, and let $E_\bu$ be a real elliptic operator on $X$. Let $P\ra X$ be a principal $\U(m)$-bundle, with $\lb P\rb=\al\in K^0(X)$, so that $\rank\al=m$, and suppose that~$2m>n=\dim X$. 

As in Definition \ref{or1def2} we have (n-)orientation bundles $O_P^{E_\bu}\ra \B_P$, $\check O_P^{E_\bu}\ra \B_P$, so pulling back gives principal $\Z_2$-bundles $(\pi^\cla)^*(O_P^{E_\bu}),(\pi^\cla)^*(\check O_P^{E_\bu})$ on $\B_P^\cla$. Equation \eq{or2eq20} implies that $\Si_P^\cC:\B_P^\cla\ra\cC_\al$ is an isomorphism on $\pi_0$ and $\pi_1$. Since principal $\Z_2$-bundles depend only on $\pi_0$ and $\pi_1$, this means that $(\Si_P^\cC)^*$ is an equivalence of categories from principal $\Z_2$-bundles on $\cC_\al$ to principal $\Z_2$-bundles on $\B_P^\cla$.

Thus, there exist principal $\Z_2$-bundles $O^{E_\bu}_\al\ra\cC_\al$, $\check O^{E_\bu}_\al\ra\cC_\al$, unique up to canonical isomorphism, with given isomorphisms
\e
\si_P^\cC:(\pi^\cla)^*(O_P^{E_\bu})\ra(\Si_P^\cC)^*(O_\al^{E_\bu}),\;\> \check\si_P^\cC:(\pi^\cla)^*(\check O_P^{E_\bu})\ra(\Si_P^\cC)^*(\check O_\al^{E_\bu}).
\label{or2eq22}
\e
As in \S\ref{or231}, as $2m\ge n$ the $\U(m)$-bundle $P\ra X$ with $\lb P\rb=\al$ is determined by $\al\in K^0(X)$ up to isomorphism, and such $P$ exist for any $\al$ with~$\rank\al=m$. 

We claim that $O^{E_\bu}_\al,\check O^{E_\bu}_\al$ are independent of all choices $P,\B_P^\cla,\Si_P^\cC$ up to canonical isomorphism, and so depend only on $X,E_\bu,\al$. To see this, note that $\Si_P^\cC=(\Pi_m\ci)\ci\Si_P$ depends on a choice of $f_P:X\t\B_P^\cla\ra B\U(m)$. As in Example \ref{or2ex10}(a), this $f_P$ lies in a contractible space. Thus, not only is $\Si_P^\cC$ unique up to homotopy, but the homotopies between two choices $\Si_{P,0}^\cC,\Si_{P,1}^\cC$ are themselves unique up to homotopy. Hence, in the discussion of Remark \ref{or2rem5}, the $\Z_2$-bundles $O_\al^{E_\bu},\check O_\al^{E_\bu}$ are independent of the choice of $\Si_P^\cC$ up to canonical isomorphism. Independence of $P,\B_P^\cla$ is also straightforward.

For all $\al\in K^0(X)$ with $2\rank\al>n$, we have now constructed principal $\Z_2$-bundles $O^{E_\bu}_\al\ra\cC_\al$, $\check O^{E_\bu}_\al\ra\cC_\al$, with natural isomorphisms $\si_P^\cC,\check\si_P^\cC$ in \eq{or2eq22} on $\B_P^\cla$ whenever $P\ra X$ is a principal $\U(m)$-bundle with $\lb P\rb=\al$. Shortly we will extend this to all $\al\in K^0(X)$, omitting the condition~$2\rank\al>n$.

Let $Q\ra X$ be a principal $\U(m')$-bundle with $\lb Q\rb=\be\in K^0(X)$, and suppose $2m'>n$. Then we have a homotopy commutative diagram \eq{or2eq21}. Consider the diagram of principal $\Z_2$-bundles 
on $\B_P^\cla\t\B_Q^\cla$, parallel to~\eq{or2eq21}:
\e
\begin{gathered}
\!\!\!\!\!\!\!\!\!\xymatrix@!0@C=148pt@R=45pt{
*+[r]{(\pi^\cla)^*(O_P^{E_\bu})\bt_{\Z_2}(\pi^\cla)^*(O_Q^{E_\bu})} \ar[rr]_{\si_P^\cC\bt\si_Q^\cC} \ar[d]^(0.4){(\pi^\cla)^*(\phi_{P,Q})} && *+[l]{(\Si_P^\cC)^*(O^{E_\bu}_\al)\bt_{\Z_2}(\Si_Q^\cC)^*(O^{E_\bu}_\be)} \ar@{.>}[d]_(0.4){(\Si_P^\cC\t\Si_Q^\cC)^*(\phi_{\al,\be})}
\\
*+[r]{\!\!\quad\begin{subarray}{l}\ts (\pi^\cla)^*(\Phi_{P,Q}^*(O_{P\op Q}^{E_\bu}))\!\simeq \\ \ts (\Phi_{P,Q}^\cla)^*\!\ci\!(\pi^\cla)^*(O_{P\op Q}^{E_\bu}) \end{subarray}} \ar[r]^(0.8){\begin{subarray}{l}(\Phi_{P,Q}^\cla)^* \\ (\si_{P\op Q}^\cC)\end{subarray}} & *+[r]{\begin{subarray}{l}\ts \qquad(\Phi_{P,Q}^\cla)^*\ci \\ \ts (\Si_{P\op Q}^\cC)^*(O^{E_\bu}_{\al+\be}) \end{subarray}} \ar[r]^\simeq_{\eq{or2eq21}}
& *+[l]{\begin{subarray}{l}\ts (\Si_P^\cC\t\Si_Q^\cC)^*\ci \\ \ts \Phi_{\al,\be}^*(O^{E_\bu}_{\al+\be}).\end{subarray}\quad\!\!}
}\!\!\!\!\!\!\!\!\!
\end{gathered}
\label{or2eq23}
\e
Here the two `$\simeq$' are isomorphisms relating pullbacks of the same bundle by homotopic morphisms, as in Remark \ref{or2rem5}. As in Remark \ref{or2rem5}(b) and Example \ref{or2ex10}, we can choose the homotopies canonically up to homotopies of homotopies, so the two `$\simeq$' are independent of choices.

Since $\Si_P^\cC\t\Si_Q^\cC$ is an isomorphism on $\pi_0$ and $\pi_1$ by \eq{or2eq20}, there is a unique isomorphism of principal $\Z_2$-bundles on $\cC_\al\t\cC_\be$ making \eq{or2eq23} commute:
\e
\phi_{\al,\be}:O_\al^{E_\bu}\bt_{\Z_2}O_\be^{E_\bu}\longra\Phi_{\al,\be}^*(O_{\al+\be}^{E_\bu}).
\label{or2eq24}
\e
The analogous argument using n-orientation bundles gives an isomorphism
\e
\check\phi_{\al,\be}:\check O_\al^{E_\bu}\bt_{\Z_2}\check O_\be^{E_\bu}\longra\Phi_{\al,\be}^*(\check O_{\al+\be}^{E_\bu}).
\label{or2eq25}
\e
The proof above showing $O^{E_\bu}_\al,\check O^{E_\bu}_\al$ are independent of choices implies that $\phi_{\al,\be},\check\phi_{\al,\be}$ are independent of choices $P,Q,\ldots,$ and depend only on $X,E_\bu,\al,\be$.

Let $\al,\be,\ga\in K^0(X)$ with $2\rank\al,2\rank\be,2\rank\ga>n$. Then we have homotopy commutative diagrams \eq{or2eq18}--\eq{or2eq19}. The $\phi_{\al,\be}$ in \eq{or2eq24} satisfy:
\begin{gather}
\si^*(\phi_{\be,\al})\,{\buildrel\eq{or2eq18}\over \simeq}\, (-1)^{\chi^{E_\bu}(\al,\be)+\chi^{E_\bu}(\al,\al)\chi^{E_\bu}(\be,\be)}\cdot\phi_{\al,\be},
\label{or2eq26}\\
\begin{split}
&(\Phi_{\al,\be}\t\id_{\cC_\ga})^*(\phi_{\al+\be,\ga})\ci(\phi_{\al,\be}\bt\id_{O_\ga^{E_\bu}})
\\
&\quad \,{\buildrel\eq{or2eq19}\over \simeq}\, (\id_{\cC_\al}\t\Phi_{\be,\ga})^*(\phi_{\al,\be+\ga})\ci(\id_{O_\al^{E_\bu}}\bt\phi_{\be,\ga}).
\end{split}
\label{or2eq27}
\end{gather}

Here two sides of \eq{or2eq26} are isomorphisms $O_\al^{E_\bu}\bt_{\Z_2}O_\be^{E_\bu}\ra(\Phi_{\be,\al}\ci\si)^*(O_{\al+\be}^{E_\bu})$ and $O_\al^{E_\bu}\bt_{\Z_2}O_\be^{E_\bu}\ra\Phi_{\al,\be}^*(O_{\al+\be}^{E_\bu})$, where $\Phi_{\be,\al}\ci\si\simeq\Phi_{\al,\be}$ by \eq{or2eq18}, and \eq{or2eq26} means the two sides are identified by the isomorphism $(\Phi_{\be,\al}\ci\si)^*(O_{\al+\be}^{E_\bu})\cong \Phi_{\al,\be}^*(O_{\al+\be}^{E_\bu})$ from parallel translation along the homotopy in \eq{or2eq18}. As in Remark \ref{or2rem5}(b) and Example \ref{or2ex10}, we can choose this homotopy canonically up to homotopies of homotopies, so the isomorphism $(\Phi_{\be,\al}\ci\si)^*(O_{\al+\be}^{E_\bu})\cong \Phi_{\al,\be}^*(O_{\al+\be}^{E_\bu})$ is independent of choices. Equation \eq{or2eq27} is interpreted the same way. We prove \eq{or2eq26}--\eq{or2eq27} by combining \eq{or2eq23} with the analogues \eq{or2eq17} and \eq{or2eq12} for the $\phi_{P,Q}$. The analogues of \eq{or2eq26}--\eq{or2eq27} also hold for the $\check\phi_{\al,\be}$ in~\eq{or2eq25}.

Let $\al,\be\in K^0(X)$ with $2\rank\al,2\rank\be>n$, and pick a base-point $b$ in $\cC_\be$. Then restricting \eq{or2eq24} to $\cC_\al\t\{b\}\cong\cC_\al$ gives a canonical isomorphism of principal $\Z_2$-bundles on $\cC_\al$:
\e
O_\al^{E_\bu}\cong (\Phi_{\al,\be}(-,b))^*(O_{\al+\be}^{E_\bu})\ot_{\Z_2}(O_\be^{E_\bu}\vert_b)^*.
\label{or2eq28}
\e
Observe that the right hand side of \eq{or2eq28} makes sense even if $2\rank\al\le n$, provided $2\rank\be,2\rank(\al+\be)>n$, for example if $\be=N\boo_X$ for $N\gg 0$. Thus we can take \eq{or2eq28} to be the {\it definition\/} of $O_\al^{E_\bu}$ when $2\rank\al\le n$. Straightforward arguments using the associativity property \eq{or2eq27} show that \eq{or2eq28} is independent of the choice of $\be$ and $b\in\cC_\be$ up to canonical isomorphism. We define $\check O_\al^{E_\bu}\ra\cC_\al$ when $2\rank\al\le n$ using \eq{or2eq25} in the same way.

Define principal $\Z_2$-bundles $O^{E_\bu}\ra\cC$, $\check O^{E_\bu}\ra\cC$ by $O^{E_\bu}\vert_{\cC_\al}=O^{E_\bu}_\al$ and $\check O^{E_\bu}\vert_{\cC_\al}=\check O^{E_\bu}_\al$ for all~$\al\in K^0(X)$.

We can also show that the $\phi_{\al,\be}$ in \eq{or2eq24} for $2\rank\al,2\rank\be>n$ extend uniquely to $\phi_{\al,\be}$ for all $\al,\be\in K^0(X)$, such that \eq{or2eq26}--\eq{or2eq27} hold for all $\al,\be,\ga\in K^0(X)$, and the definition \eq{or2eq28} of $O_\al^{E_\bu}$ for $2\rank\al\le n$ is identified with $\phi_{\al,\be}\vert_{C_\al\t\{b\}}$. We extend $\check\phi_{\al,\be}$ in \eq{or2eq25} to all $\al,\be$ in the same way.

Suppose $P\ra X$ is a principal $\U(m)$-bundle for $2m\le n$, and $\al=\lb P\rb\in K^0(X)$. Choose $\be\in K^0(X)$ with $2\rank\be,2\rank(\al+\be)>n$, set $m'=\rank\be$, and let $Q\ra X$ be a principal $\U(m')$-bundle with $\lb Q\rb=\be$. Then in \eq{or2eq23}, all morphisms are defined except $\si_P^\cC$. As for \eq{or2eq28}, picking a base-point $b\in\B_Q^\cla$ there is a unique isomorphism $\si_P^\cC$ in \eq{or2eq22} such that the restriction of \eq{or2eq23} to $\B_P^\cla\t\{b\}$ commutes. Using \eq{or2eq12} and \eq{or2eq27} we can show $\si_P^\cC$ is independent of $\be,Q,b$, and that \eq{or2eq23} commutes for all $P,Q$ without supposing $2m,2m'>n$. We construct $\check\si_P^\cC$ in \eq{or2eq22} in the same way.

In the obvious way, we say that $\cC,\cC_\al$ are {\it orientable\/} (or {\it n-orientable\/}) for $\al$ in $K^0(X)$ if $O^{E_\bu},O_\al^{E_\bu}$ (or $\check O^{E_\bu},\check O_\al^{E_\bu}$) are trivializable, and an {\it orientation\/} $\om_\al$ (or {\it n-orientation\/} $\check\om_\al$) for $\cC_\al$ is a trivialization $O_\al^{E_\bu}\cong\cC_\al\t\Z_2$ (or~$\check O_\al^{E_\bu}\cong\cC_\al\t\Z_2$).
\end{dfn}         

\begin{rem}
\label{or2rem6}	
{\bf(a)} The importance of \eq{or2eq22} is that it shows that if $\cC_\al$ is orientable then $\B_P$ is orientable for any principal $\U(m)$-bundle $P\ra X$ with $\lb P\rb =\al$ in $K^0(X)$, and an orientation for $\cC_\al$ induces orientations on $\B_P$ for all such $P$. Hence, if we can construct orientations on $\cC_\al$ for all $\al\in K^0(X)$, we obtain orientations on $\B_P$ for all $\U(m)$-bundles $P\ra X$, for all~$m\ge 0$.
\smallskip

\noindent{\bf(b)} Here is an example of how to apply equations \eq{or2eq24} and \eq{or2eq26} above to orientations. Suppose $\al,\be\in K^0(X)$ with $\cC_\al,\cC_\be,\cC_{\al+\be}$ orientable, and choose orientations $\om_\al,\om_\be,\om_{\al+\be}$ on $\cC_\al,\cC_\be,\cC_{\al+\be}$, where $\om_\al:O_\al^{E_\bu}\,{\buildrel\cong\over\longra}\,\cC_\al\t\Z_2$, and so on. Then in \eq{or2eq24}, $\om_\al\bt\om_\be$ is a trivialization of $O_\al^{E_\bu}\bt_{\Z_2}O_\be^{E_\bu}\ra\cC_\al\t\cC_\be$, so $\phi_{\al,\be}(\om_\al\bt\om_\be)$ is a trivialization of $\Phi_{\al,\be}^*(O_{\al+\be}^{E_\bu})\ra\cC_\al\t\cC_\be$, as is $\Phi_{\al,\be}^*(\om_{\al+\be})$. Since $\cC_\al\t\cC_\be$ is connected, there is a unique $\ep_{\al,\be}=\pm 1$ with
\begin{align*}
\phi_{\al,\be}(\om_\al\bt\om_\be)&=\ep_{\al,\be}\cdot \Phi_{\al,\be}^*(\om_{\al+\be}),\quad\text{and similarly}\\
\phi_{\be,\al}(\om_\be\bt\om_\al)&=\ep_{\be,\al}\cdot \Phi_{\be,\al}^*(\om_{\al+\be}).
\end{align*}
Then \eq{or2eq26} implies that $\ep_{\al,\be},\ep_{\be,\al}$ are related by
\e
\ep_{\be,\al}=(-1)^{\chi^{E_\bu}(\al,\be)+\chi^{E_\bu}(\al,\al)\chi^{E_\bu}(\be,\be)}\cdot\ep_{\al,\be}.
\label{or2eq29}
\e
The same methods work for \eq{or2eq27}, and for n-orientations.
\smallskip

\noindent{\bf(c)} Equations \eq{or2eq26}--\eq{or2eq27} need to be interpreted carefully as they relate isomorphisms by homotopic morphisms. However, when we apply them to orientations on $\cC_\al,\cC_\be,\cC_\ga$ as in {\bf(b)}, they simplify, as we need not worry about homotopies. Also, as in Remark \ref{or2rem5}, in the orientable case the issue of pullbacks by homotopic morphisms being non-canonically isomorphic disappears.
\end{rem}

\subsubsection{The isomorphism $\pi_1(\cC_\al)\cong K^1(X)$, and orientability}
\label{or243}
 
\begin{prop}
\label{or2prop3}	
In Definition\/ {\rm\ref{or2def3},} for each\/ $\al\in K^0(X),$ we have:
\begin{itemize}
\setlength{\itemsep}{0pt}
\setlength{\parsep}{0pt}
\item[{\bf(a)}] $\cC_\al$ is homotopy-equivalent to $\cC_0$.
\item[{\bf(b)}] $\cC_\al$ is orientable if and only if\/ $\cC_0$ is orientable.
\item[{\bf(c)}] The fundamental group is $\pi_1(\cC_\al)\cong K^1(X),$ for $K^1(X)$ as in\/ {\rm\S\ref{or231}}.
\end{itemize}
\end{prop}

\begin{proof} For (a), let $\al\in K^0(X)$, choose points $p\in\cC_\al$, $q\in\cC_{-\al}$, and set $r=\Phi_{\al,-\al}(p,q)$ and $s=\Phi_{-\al,\al}(p,q)$ in $\cC_0$. Then we have morphisms
\begin{align*}
\Phi_{0,\al}\vert_{\cC_0\t\{p\}}&:\cC_0\cong \cC_0\t\{p\}\longra \cC_\al,\\
\Phi_{\al,-\al}\vert_{\cC_\al\t\{q\}}&:\cC_\al\cong \cC_\al\t\{q\}\longra \cC_0.
\end{align*}
By \eq{or2eq19}, we have homotopies
\begin{align*}
\Phi_{\al,-\al}\vert_{\cC_\al\t\{q\}}\ci\Phi_{0,\al}\vert_{\cC_0\t\{p\}}&\simeq\Phi_{0,0}\vert_{\cC_0\t\{r\}}:\cC_0\longra \cC_0,\\
\Phi_{0,\al}\vert_{\cC_0\t\{p\}}\ci\Phi_{\al,-\al}\vert_{\cC_\al\t\{q\}}&\simeq\Phi_{\al,0}\vert_{\cC_\al\t\{s\}}:\cC_\al\longra \cC_\al.
\end{align*}
As $B\U\t\Z$ is an H-space in the sense of May \cite[\S 22.2]{May2}, $\cC=\Map_{C^0}(X,B\U\t\Z)$ is also an H-space, so it has a homotopy identity, which may be any point in $\cC_0$, such as $r,s$. Thus $\Phi\vert_{\cC\t\{r\}}\simeq\id_\cC\simeq\Phi\vert_{\cC\t\{s\}}$, so $\Phi_{0,0}\vert_{\cC_0\t\{r\}}\simeq\id_{\cC_0}$ and $\Phi_{\al,0}\vert_{\cC_\al\t\{s\}}\simeq\id_{\cC_\al}$. Therefore $\Phi_{0,\al}\vert_{\cC_0\t\{p\}}$ and $\Phi_{\al,-\al}\vert_{\cC_\al\t\{q\}}$ are homotopy inverses, and $\cC_\al$ is homotopy-equivalent to~$\cC_0$.

For (b), if $\om_\al$ is an orientation for $\cC_\al$ then $\phi_{0,\al}^{-1}\vert_{\cC_0\t\{p\}}(\om_\al)$ from \eq{or2eq24} is an orientation for $O_0^{E_\bu}\ot_{\Z_2}\bigl(O_\al^{E_\bu}\vert_p\bigr)$, so choosing an identification $O_\al^{E_\bu}\vert_p\cong\Z_2$ gives an orientation for $\cC_0$. Thus, if $\cC_\al$ is orientable, then $\cC_0$ is orientable. Similarly, if $\cC_0$ is orientable then $\cC_\al$ is orientable. Part (c) follows from~\eq{or2eq14}. 
\end{proof}

Combining Proposition \ref{or2prop3}(c) with Lemma \ref{or2lem1} yields a criterion for orientability of the $\cC_\al,\B_P$, as in Problem~\ref{or1prob}(a):

\begin{cor} In the situation of Definition\/ {\rm\ref{or2def3},} if\/ $K^1(X)\ot_\Z\Z_2=0$ then $\cC_\al$ is orientable for all\/ $\al\in K^0(X),$ and hence by\/ {\rm\eq{or2eq22},} $\B_P$ is orientable for all principal\/ $\U(m)$-bundles $P\ra X$. By the Atiyah--Hirzebruch spectral sequence, a sufficient condition for $K^1(X)\ot_\Z\Z_2=0$ is that\/~$H^{\rm odd}(X,\Z_2)=0$.

\label{or2cor1}	
\end{cor}

A version of this corollary is used by Cao and Leung \cite[\S 10.4]{CaLe1}, \cite[Th.~2.1]{CaLe2}.

\subsubsection{An index-theoretic expression for obstruction to orientability}
\label{or244}

In Definition \ref{or2def3}, there is a group morphism $\pi_1(\cC_0)\ra\{\pm 1\}$ mapping $[\ga]$ in $\pi_1(\cC_0)$ to the monodromy of the principal $\Z_2$-bundle $O_0^{E_\bu}\ra\cC_0$ round the loop $\ga$, and $\cC_0$ is orientable if and only if this morphism is the constant map 1. So by Proposition \ref{or2prop3}(b),(c), this gives a natural morphism $\Th:K^1(X)\ra\{\pm 1\}$, such that $\cC_\al$ is orientable for all $\al\in K^0(X)$ if and only if~$\Th\equiv 1$.

By a calculation in index theory following Atiyah and Singer \cite{AtSi1,AtSe,AtSi3,AtSi4,AtSi5}, using the Atiyah--Singer Index Theorem for families \cite{AtSi4} over a base $\cS^1$, one can show that $\Th$ is given by the commutative diagram:
\begin{equation*}
\xymatrix@C=90pt@R=15pt{
*+[r]{K^1(X)} \ar[d]^\Th \ar@{=}[r] & K^{-1}(X) \ar[r]_(0.45){\Ad^{-1}} & *+[l]{KO^{-1}(X)} \ar[d]_{\be\mapsto \pi^*(\be)\cup \si(E_\bu)}
\\
*+[r]{\{\pm 1\}} & KO^{-1}(*) \ar[l]_\cong & *+[l]{KO^{-1}_{\rm cs} (TX).\!} \ar[l]_(0.55){\tind^{-1}}
}	
\end{equation*}

Here $\Ad^i:K^i(X)\ra KO^i(X)$ is a natural quadratic map which when $i=0$ maps $\Ad^0:\lb P\rb\mapsto \lb\Ad(P)\rb$ for any principal $\U(m)$-bundle $P\ra X$. Also $KO^i_{\rm cs}(TX)$ is the compactly-supported real K-theory of the tangent bundle $TX$, and $\si(E_\bu)\in KO^0_{\rm cs}(TX)$ is defined using the symbol of $E_\bu$, and $\pi^*:KO^{-1}(X)\ra KO^{-1}(TX)$ is pullback by $\pi:TX\ra X$, and $\cup:KO^{-1}(TX)\t KO^0_{\rm cs}(X)\ra KO^{-1}_{\rm cs}(X)$ is the cup product, and $\tind^i:KO^i_{\rm cs}(TX)\ra KO^i(*)$ is the {\it topological index\/} morphism of Atiyah and Singer~\cite[\S 3]{AtSi1}.

\subsection{Comparing orientations under direct sums}
\label{or25}

We define the {\it orientation group\/} of $X,E_\bu$:

\begin{dfn} In the situation of Definition \ref{or2def3}, suppose that $\cC_0$ is orientable, so that $\cC_\al$ is orientable for all $\al\in K^0(X)$ by Proposition \ref{or2prop3}(b). Thus, each $\cC_\al$ has two possible orientations, as it is connected. 

There is a natural orientation $\bar\om_0$ on $\cC_0$, defined as follows. Let $P=X\t\U(0)\ra X$ be the trivial $\U(0)$-bundle with $\lb P\rb=0$ in $K^0(X)$. Then $\B_P$ in Definition \ref{or1def1} is a point, and $O_P^{E_\bu}=\{1,-1\}$ is naturally trivial. We fix $\bar\om_0$ by requiring that $(\si_P^\cC)_*(1)=(\Si_P^\cC)^*(\bar\om_0)$, for $\si_P^\cC$ as in \eq{or2eq22}. Equivalently, $\bar\om_0$ is characterized by $(0,\bar\om_0)\star(0,\bar\om_0)=(0,\bar\om_0)$, using the multiplication $\star$ below.

Define the {\it orientation group\/} $\Om(X)$, initially just as a set, by
\begin{equation*}
\Om(X)=\bigl\{(\al,\om_\al):\text{$\al\in K^0(X)$, $\om_\al$ is an orientation on $\cC_\al$}\bigr\}.
\end{equation*}
Define a map $\pi:\Om(X)\ra K^0(X)$ by $\pi:(\al,\om_\al)\mapsto\al$. Define an action $\cdot:\{\pm 1\}\t\Om(X)\ra\Om(X)$ by $\ep\cdot(\al,\om_\al)=(\al,\ep\cdot\om_\al)$ for $\ep=\pm 1$. Then $\pi,\cdot$ make $\Om(X)$ into a principal $\Z_2$-bundle over~$K^0(X)$.

Define a multiplication $\star:\Om(X)\t\Om(X)\ra\Om(X)$ by
\begin{gather*}
(\al,\om_\al)\star (\be,\om_\be)=(\al+\be,\om_{\al+\be}), \quad\text{where $\om_{\al+\be}$ is uniquely}\\
\text{determined by}\quad \phi_{\al,\be}(\om_\al\bt\om_\be)=\Phi_{\al,\be}^*(\om_{\al+\be}),
\end{gather*}
using the notation of Definitions \ref{or2def2}--\ref{or2def3} and Remark \ref{or2rem6}(b). Equation \eq{or2eq27} implies that $\star$ is associative. From the definition of $\bar\om_0$ we see that $(0,\bar\om_0)\star(\al,\om_\al)=(\al,\om_\al)\star(0,\bar\om_0)=(\al,\om_\al)$, so $(0,\bar\om_0)$ is the identity in $\Om(X)$. For any $(\al,\om_\al)$ in $\Om(X)$, we can easily show that some $\om_{-\al}$, one of the two possible orientations on $\cC_{-\al}$, satisfies $(-\al,\om_{-\al})\star(\al,\om_\al)=(\al,\om_\al)\star(-\al,\om_{-\al})=(0,\bar\om_0)$, so inverses exist in $\Om(X)$. Thus $\Om(X)$ is a group, which depends on $X,E_\bu$. The multiplication in $\Om(X)$ compares orientations on $\cC_\al,\cC_\be,\cC_{\al+\be}$ under the direct sum morphisms $\Phi_{\al,\be},\phi_{\al,\be}$ of~\S\ref{or242}.

Clearly the map $\pi:\Om(X)\ra K^0(X)$ is a surjective group morphism, with kernel $\{(0,\bar\om_0),(0,-\bar\om_0)\}\cong\{\pm 1\}$, so we have an exact sequence of groups
\e
\xymatrix@C=35pt{ 0 \ar[r] & \{1,-1\} \ar[r] & \Om(X) \ar[r]^\pi & K^0(X) \ar[r] & 0. }
\label{or2eq30}
\e
Equations \eq{or2eq26} and \eq{or2eq29} imply that for all $(\al,\om_\al),(\be,\om_\be)$ in $\Om(X)$ we have
\e
(\be,\om_\be)\star(\al,\om_\al)=(-1)^{\chi^{E_\bu}(\al,\be)+\chi^{E_\bu}(\al,\al)\chi^{E_\bu}(\be,\be)}\cdot (\al,\om_\al)\star(\be,\om_\be).
\label{or2eq31}
\e
So in general $\Om(X)$ may not be abelian.
\label{or2def4}	
\end{dfn}

The orientation group $\Om(X)$ is closely related to the problem of choosing canonical orientations on mapping spaces $\cC_\al$ for all $\al\in K^0(X)$, and hence canonical orientations on $\B_P$ for all principal $\U(m)$-bundles $P\ra X$ as in Definition \ref{or2def3}, with relations between these canonical orientations under direct sums, as in Problem~\ref{or1prob}(c).

Observe that choosing an orientation $\ti\om_\al$ on $\cC_\al$ for all $\al\in K^0(X)$ is equivalent to choosing a bijection $\La:\Om(X)\,{\buildrel\cong\over\longra}\, K^0(X)\t\{\pm 1\}$ compatible with \eq{or2eq30}. Then there are $\ti\ep_{\al,\be}\in\{\pm 1\}$ such that $(\al,\ti\om_\al)\star(\be,\ti\om_\be)=\ti\ep_{\al,\be}\cdot(\al+\be,\ti\om_{\al+\be})$, which encode the multiplication $\star$ on $\Om(X)$, and the signs $\ep_{P_1,P_2}$ in Problem \ref{or1prob}(c) are $\ep_{P_1,P_2}=\ti\ep_{\lb P_1\rb,\lb P_2\rb}$. So, Problem \ref{or1prob}(c) is really about understanding the group $\Om(X)$ and writing the multiplication $\star$ in an explicit form under a suitable trivialization $\La$ of the principal $\Z_2$-bundle~$\pi:\Om(X)\ra K^0(X)$.

The next theorem is just an exercise in group theory: it classifies groups $\Om(X)$ in an exact sequence \eq{or2eq30} satisfying \eq{or2eq31}, and uses no further properties of $\Om(X)$. It shows that $\Om(X)$ depends up to isomorphism only on the finitely generated abelian group $K^0(X)$, the Euler form $\chi^{E_\bu}:K^0(X)\t K^0(X)\ra\Z$, and a certain group morphism $\Xi:G\ra\{\pm 1\}$, where $G=\{\ga\in K^0(X):2\ga=0\}$ is the 2-torsion subgroup of $K^0(X)$. It provides explicit signs $\ti\ep_{\al,\be}$ as above, which are the signs $(-1)^{\sum_{1\le h<i\le 1}(\chi_{hi}^{E_\bu}+\chi_{hh}^{E_\bu}\chi_{ii}^{E_\bu})a'_ha_i}\cdot \Xi(\ga)$ in the third line of \eq{or2eq35}. Parts (b),(c) are at least a partial solution of Problem~\ref{or1prob}(c).

\begin{thm} 
\label{or2thm2}	
Let\/ $X$ be a compact\/ $n$-manifold and\/ $E_\bu$ an elliptic operator on $X,$ and use the notation of\/ {\rm\S\ref{or242}}. Suppose\/ $\cC_0$ is orientable, so Definition {\rm\ref{or2def4}} defines the orientation group $\Om(X),$ and a natural orientation $\bar\om_0$ for $\cC_0$. Since $K^0(X)$ is a finitely generated abelian group, we may choose an isomorphism
\e
K^0(X)\cong\Z^r\t\ts\prod_{j\in J}\Z_{2^{p_j}}\t\prod_{k\in K}\Z_{q_k},
\label{or2eq32}
\e
where $J,K$ are finite indexing sets, and\/ $p_j>0,$ $q_k>1$ for $j\in J,$ $k\in K$ with\/ $q_k$ odd. Under the isomorphism \eq{or2eq32} we write elements of\/ $K^0(X)$ as $\bigl((a_1,\ldots,a_r),(b_j)_{j\in J},(c_k)_{k\in K}\bigr)$ for $a_i\in\Z,$ $b_j\in\Z_{2^{p_j}},$ $c_k\in\Z_{q_k}$. We may write
\begin{align*}
&\chi^{E_\bu}\bigl[\bigl((a_1,\ldots,a_r),(b_j)_{j\in J},(c_k)_{k\in K}\bigr),
\bigl((a'_1,\ldots,a'_r),(b'_j)_{j\in J},(c'_k)_{k\in K}\bigr)\bigr]\\
&\qquad\ts =\sum_{h,i=1}^r\chi_{hi}^{E_\bu}\,a_ha'_i,
\end{align*}
where $\chi_{hi}^{E_\bu}\in\Z$ with\/ $\chi_{ih}^{E_\bu}=\chi_{hi}^{E_\bu}$. Write\/ $G=\{\ga\in K^0(X):2\ga=0\}$ for the $2$-torsion subgroup of\/ $K^0(X),$ so that in the representation \eq{or2eq32} we have
\e
G=\bigl\{\bigl((0,\ldots,0),(b_j)_{j\in J},(0)_{k\in K}\bigr):\text{$b_j=0+2^{p_j}\Z$ or\/ $2^{p_j-1}+2^{p_j}\Z$}\bigr\}.
\label{or2eq33}
\e
Then:
\begin{itemize}
\setlength{\itemsep}{0pt}
\setlength{\parsep}{0pt}
\item[{\bf(a)}] There is a unique group morphism $\Xi:G\ra\{1,-1\}$ depending on $X,E_\bu,$ such that if\/ $\ga\in G$ then for any orientation $\om_\ga$ on $\cC_\ga$ we have
\e
(\ga,\om_\ga)\star (\ga,\om_\ga)=\Xi(\ga)\cdot(0,\bar\om_0).
\label{or2eq34}
\e
\item[{\bf(b)}] There exists a bijection $\La:\Om(X)\,\smash{\buildrel\cong\over\longra}\,K^0(X)\t\{\pm 1\}$ such that using\/ $\La$ and\/ \eq{or2eq32} to identify\/ $\Om(X)$ with\/ $\Z^r\t\ts\prod_{j\in J}\Z_{2^{p_j}}\t\prod_{k\in K}\Z_{q_k}\t\{\pm 1\},$ the multiplication $\star$ in $\Om(X)$ is given explicitly by
\ea
&\bigl[\bigl((a_1,\ldots,a_r),(b_j)_{j\in J},(c_k)_{k\in K}\bigr),\ep\bigr]\star
\bigl[\bigl((a_1',\ldots,a_r'),(b_j')_{j\in J},(c_k')_{k\in K}\bigr),\ep'\bigr]
\nonumber\\
&=\bigl[\bigl((a_1+a_1',\ldots,a_r+a_r'),(b_j+b_j')_{j\in J},(c_k+c_k')_{k\in K}\bigr),
\nonumber\\
&\qquad (-1)^{\sum_{1\le h<i\le 1}(\chi_{hi}^{E_\bu}+\chi_{hh}^{E_\bu}\chi_{ii}^{E_\bu})a'_ha_i}\cdot \Xi(\ga)\cdot
\ep\ep'\bigr],
\label{or2eq35}
\ea
where $\ga\in G$ is constructed from $(b_j)_{j\in J},(b_j')_{j\in J}$ as follows: write $b_j=\bar b_j+2^{p_j}\Z,$ $b'_j=\bar b'_j+2^{p_j}\Z$ for $\bar b_j,\bar b_j'\in\{0,1,\ldots,2^{p_j}-1\}$. Then under \eq{or2eq33} we set\/ $\ga=\bigl((0,\ldots,0),(\ti b_j)_{j\in J},(0)_{k\in K}\bigr),$ where $\ti b_j=0+2^{p_j}\Z$ if\/ $\bar b_j+\bar b_j'<2^{p_j}$ and\/ $\ti b_j=2^{p_j-1}+2^{p_j}\Z$ if\/ $\bar b_j+\bar b_j'\ge 2^{p_j},$ for\/~$j\in J$.
\item[{\bf(c)}] Suppose $\La,\ti\La$ both satisfy {\bf(b)}. Then there exist unique signs $\eta_i,\ze_j\in\{\pm 1\}$ for $i=1,\ldots,r$ and\/ $j\in J$ such that for all\/ $[((a_1,\ldots,a_r),\ldots),\ep]$ we have
\e
\begin{split}
&\ti\La\ci\La^{-1}\bigl[\bigl((a_1,\ldots,a_r),(b_j)_{j\in J},(c_k)_{k\in K}\bigr),\ep\bigr]\\
&=\bigl[\bigl((a_1,\ldots,a_r),(b_j)_{j\in J},(c_k)_{k\in K}\bigr),
\ts\prod_{i=1}^r\eta_i^{a_i}\cdot\prod_{j\in J}\ze_j^{b_j}\cdot
\ep\bigr].
\end{split}
\label{or2eq36}
\e
Conversely, if\/ $\La$ satisfies {\bf(b)} and\/ $\eta_i,\ze_j\in\{\pm 1\}$ are given for all\/ $i,j,$ and we define $\ti\La:\Om(X)\,{\buildrel\cong\over\longra}\,K^0(X)\t\{\pm 1\}$ by {\rm\eq{or2eq36},} then $\ti\La$ satisfies\/~{\bf(b)}.
\end{itemize}
\end{thm}

\begin{proof} The first part of the theorem is immediate. For (a), if $\ga\in G$ and $\om_\ga$ is an orientation on $\cC_\ga$ then there is a unique $\Xi(\ga)=\pm 1$ satisfying \eq{or2eq34}, so as $(\ga,\om_\ga)\star (\ga,\om_\ga)=(\ga,-\om_\ga)\star(\ga,-\om_\ga)$, the map $\Xi:G\ra\{\pm 1\}$ is well defined. To see that $\Xi$ is a group morphism, note that if $\ga,\ga'\in G$ with $\ga''=\ga+\ga'$ and $\om_\ga,\om_{\ga'}$ are orientations on $\cC_\ga,\cC_{\ga'}$ with $(\ga,\om_\ga)\star(\ga',\om_{\ga'})=(\ga'',\om_{\ga''})$, then
\begin{align*}
&\Xi(\ga\!+\!\ga')\cdot(0,\bar\om_0)\!=\!(\ga'',\om_{\ga''})\star(\ga'',\om_{\ga''})\!=\!
(\ga,\om_\ga)\star(\ga',\om_{\ga'})\star(\ga,\om_\ga)\star(\ga',\om_{\ga'})\\
&=(\ga,\om_\ga)\star(\ga,\om_\ga)\star(\ga',\om_{\ga'})\star(\ga',\om_{\ga'})
=\bigl[\Xi(\ga)\cdot(0,\bar\om_0)\bigr]\star\bigl[\Xi(\ga')\cdot(0,\bar\om_0)\bigr]\\
&=\Xi(\ga)\Xi(\ga')\cdot(0,\bar\om_0),
\end{align*}
using \eq{or2eq34} in the fourth step and \eq{or2eq31} with $\chi^{E_\bu}\vert_{G\t G}=0$ in the third. Hence $\Xi(\ga+\ga')=\Xi(\ga)\Xi(\ga')$, and $\Xi$ is a morphism, proving~(a).

For (b), define elements $\la_i,\mu_j,\nu_k$ in $K^0(X)$ for $i=1,\ldots,r$, $j\in J$, $k\in K$ which are identified by \eq{or2eq32} with elements which have $a_k=1$, and $b_j=1$, and $c_k=1$ respectively, and all other entries zero, so that \eq{or2eq32} identifies $\bigl((a_1,\ldots,a_r),(b_j)_{j\in J},(c_k)_{k\in K}\bigr)$ with $\sum_ia_i\la_i+\sum_jb_j\mu_j+\sum_kc_k\nu_k$ in $K^0(X)$.

For all $i=1,\ldots,r$ choose an arbitrary orientation $\ti\om_{\la_i}$ for $\cC_{\la_i}$. For all $j\in J$ choose an arbitrary orientation $\ti\om_{\mu_j}$ for $\cC_{\mu_j}$. For each $k\in K$, let $\ti\om_{\nu_k}$ be the unique orientation for $\cC_{\nu_k}$ satisfying
\e
(\nu_k,\ti\om_{\nu_k})^{q_k}={\buildrel
{\ulcorner\,\,\,\text{$q_k$ copies } \,\,\,\urcorner} \over
{\vphantom{h}\smash{(\nu_k,\ti\om_{\nu_k})\star\cdots\star (\nu_k,\ti\om_{\nu_k})}}}=(0,\bar\om_0).
\label{or2eq37}
\e
This is well defined as $q_k\nu_k=0$ in $K^0(X)$, and replacing $\ti\om_{\nu_k}$ by $-\ti\om_{\nu_k}$ multiplies the left hand side of \eq{or2eq37} by $(-1)^{q_k}=-1$, as $q_k$ is odd, so exactly one of the two orientations on $\cC_{\nu_k}$ satisfies~\eq{or2eq37}.

Let $\al\in K^0(X)$. Then $\al$ corresponds under \eq{or2eq32} to some $\bigl((a_1,\ldots,a_r),\ab (b_j)_{j\in J},\ab (c_k)_{k\in K}\bigr)$. Write $b_j=\bar b_j+2^{p_j}\Z$ for unique $\bar b_j=0,\ldots,2^{p_j}-1$ and $c_j=\bar c_j+q_j\Z$ for unique $\bar c_j=0,\ldots,q_j-1$. Define an orientation $\ti\om_\al$ on $\cC_\al$ by
\e
(\al,\ti\om_\al)=(\la_1,\ti\om_{\la_1})^{a_1}\star\cdots\star (\la_r,\ti\om_{\la_r})^{a_r}\star\prod_{j\in J}(\mu_j,\ti\om_{\mu_j})^{\bar b_j}\star\prod_{k\in K}(\nu_k,\ti\om_{\nu_k})^{\bar c_k}.
\label{or2eq38}
\e
Here we should be careful as $\Om(X)$ may not be abelian by \eq{or2eq31}, and we have not specified orderings of $J,K$. But in fact $\chi^{E_\bu}$ is zero on the torsion factors of $K^0(X)$, so the elements $(\mu_j,\ti\om_{\mu_j}),(\nu_k,\ti\om_{\nu_k})$ in \eq{or2eq38} lie in the centre of $\Om(X)$, and only the order of the factors $(\la_1,\ti\om_{\la_1})^{a_1},\ldots,(\la_r,\ti\om_{\la_r})^{a_r}$ matters. Define $\La:\Om(X)\ra K^0(X)\t\{\pm 1\}$ in (b) by, for all $\al\in K^0(X)$ and~$\ep=\pm 1$
\e
\La:(\al,\ep\cdot\ti\om_\al)\longmapsto	\bigl[\bigl((a_1,\ldots,a_r),(b_j)_{j\in J},(c_k)_{k\in K}\bigr),\ep\bigr].
\label{or2eq39}
\e

Equation \eq{or2eq35} now follows from \eq{or2eq38}--\eq{or2eq39}, the fact that $(\mu_j,\ti\om_{\mu_j}),\ab(\nu_k,\ti\om_{\nu_k})$ lie in the centre of $\Om(X)$, and the next three equations
\ea
&\bigl[(\la_1,\ti\om_{\la_1})^{a_1}\star\cdots\star (\la_r,\ti\om_{\la_r})^{a_r}\bigr]\star\bigl[(\la_1,\ti\om_{\la_1})^{a'_1}\star\cdots\star (\la_r,\ti\om_{\la_r})^{a'_r}\bigr]
\label{or2eq40}\\
&=(-1)^{\sum_{1\le h<i\le 1}(\chi_{hi}^{E_\bu}+\chi_{hh}^{E_\bu}\chi_{ii}^{E_\bu})a'_ha_i}\cdot\bigl[(\la_1,\ti\om_{\la_1})^{a_1+a_1'}\star\cdots\star (\la_r,\ti\om_{\la_r})^{a_r+a_r'}\bigr],
\allowdisplaybreaks
\nonumber\\
\begin{split}
&\bigl[\ts\prod_{j\in J}(\mu_j,\ti\om_{\mu_j})^{\bar b_j}\bigr]\star\bigl[\ts\prod_{j\in J}(\mu_j,\ti\om_{\mu_j})^{\bar b'_j}\bigr]\\
&=\bigl[\ts\prod_{j\in J}(\mu_j,\ti\om_{\mu_j})^{\bar b_j+\bar b_j'-\ov {(b_j+b_j')}}\bigr]\star\bigl[\ts\prod_{j\in J}(\mu_j,\ti\om_{\mu_j})^{\ov {(b_j+b_j')}}\bigr]\\
&=(\ga,\ti\om_\ga)\star(\ga,\ti\om_\ga)\star\bigl[\ts\prod_{j\in J}(\mu_j,\ti\om_{\mu_j})^{\ov {(b_j+b_j')}}\bigr]\\
&=\Xi(\ga)\cdot\bigl[\ts\prod_{j\in J}(\mu_j,\ti\om_{\mu_j})^{\ov {(b_j+b_j')}}\bigr],
\end{split}
\label{or2eq41}
\allowdisplaybreaks\\
\begin{split}
&\bigl[\ts\prod_{k\in K}(\nu_k,\ti\om_{\nu_k})^{\bar c_k}\bigr]\star\bigl[\prod_{k\in K}(\nu_k,\ti\om_{\nu_k})^{\bar c'_k}\bigr]\\
&=\bigl[\ts\prod_{k\in K}(\nu_k,\ti\om_{\nu_k})^{\bar c_k+\bar c_k'-\ov{(c_k+c_k')}}\bigr]\star\bigl[\prod_{k\in K}(\nu_k,\ti\om_{\nu_k})^{\ov{(c_k+c_k')}}\bigr]\\
&=\bigl[\ts\prod_{k\in K}(\nu_k,\ti\om_{\nu_k})^{\ov{(c_k+c_k')}}\bigr],
\end{split}
\label{or2eq42}
\ea
where $\ga\in G$ in \eq{or2eq41} is defined as in (b). Here \eq{or2eq40} follows from \eq{or2eq31}. The first step of \eq{or2eq41} is immediate as the $(\mu_j,\ti\om_{\mu_j})$ commute in $\Om(X)$, the second holds as for each $j\in J$ either $\bar b_j+\bar b_j'-\ov {(b_j+b_j')}=0$ in which case $\ti b_j=0+2^{p_j}\Z$ in (b), or $\bar b_j+\bar b_j'-\ov {(b_j+b_j')}=2^{p_j}$ in which case $\ti b_j=2^{p_j-1}+2^{p_j}\Z$ in (b), and the third holds by \eq{or2eq34} as $(0,\bar\om_0)$ is the identity in $\Om(X)$. The first step of \eq{or2eq42} is immediate as the $(\nu_k,\ti\om_{\nu_k})$ commute in $\Om(X)$, and the second step holds by \eq{or2eq37} as $\bar c_k+\bar c_k'-\ov{(c_k+c_k')}=0$ or $q_k$ for each $k\in K$. This completes~(b). 

For (c), note that the only arbitrary choices we made in the proof of (b) were orientations $\ti\om_{\la_i}$ for $\cC_{\la_i}$ for $i=1,\ldots,r$ and $\ti\om_{\mu_j}$ for $\cC_{\mu_j}$ for all $j\in J$. Replacing $\ti\om_{\la_i}$ and $\ti\om_{\mu_j}$ by $\eta_i\cdot\ti\om_{\la_i}$ and $\ze_j\cdot\ti\om_{\mu_j}$ for all $i,j$ with $\eta_i,\ze_j\in\{\pm 1\}$ would yield an alternative bijection $\ti\La$ satisfying (b), where $\La,\ti\La$ are related by \eq{or2eq36}. This proves the last part of (c). 

For the first part, note that if $\ti\La$ satisfies (b) then we must have $\ti\La(\la_i,\ti\om_{\la_i})=(\la_i,\eta_i)$ for some $\eta_i=\pm 1$, all $i=1,\ldots,r$, and $\ti\La(\mu_j,\ti\om_{\mu_j})=(\mu_j,\ze_j)$ for some $\ze_j=\pm 1$, all $j\in J$. Using \eq{or2eq35} and \eq{or2eq37} we find that $\ti\La(\nu_k,\ti\om_{\nu_k})=(\nu_k,1)$ for all $k\in K$. Then for any $\al\in K^0(X)$, writing $(\al,\ti\om_\al)$ as in \eq{or2eq38}, we can use \eq{or2eq35} to determine $\ti\La(\al,\ti\om_\al)$ from $\ti\La(\la_i,\ti\om_{\la_i})=(\la_i,\eta_i)$, $\ti\La(\mu_j,\ti\om_{\mu_j})=(\mu_j,\ze_j)$ and $\ti\La(\nu_k,\ti\om_{\nu_k})=(\nu_k,1)$, and it must be the same as $\ti\La$ constructed above with $\eta_i\cdot\ti\om_{\la_i}$ and $\ze_j\cdot\ti\om_{\mu_j}$ in place of $\ti\om_{\la_i},\ti\om_{\mu_j}$, so \eq{or2eq36} holds.
\end{proof}

\begin{rem} The material of this section does not extend from unitary groups $\U(m)$ and mapping spaces $\cC_\al=\Map_{C^0}(X,B\U\t\Z)_\al$, to any of the families of Lie groups $\mathbin{\rm O}(m),\ab\SO(m),\ab\Spin(m)$ or $\Sp(m)$, and the corresponding mapping spaces $\Map_{C^0}(X,B{\rm O}\t\Z)_\al,\ldots.$ This is because Example \ref{or2ex4} does not extend to $\mathbin{\rm O}(m),\ldots,\Sp(m)$, as in Remark \ref{or2rem4}, so we have no way to compare orientations for these groups under direct sums.
\label{or2rem7}
\end{rem}

\section{Constructing orientations by excision}
\label{or3}

We now explain a method for orienting moduli spaces using `excision'. This was introduced by Donaldson \cite[\S II.4]{Dona1}, \cite[\S 3(b)]{Dona2}, \cite[\S 7.1.6]{DoKr} for moduli spaces of instantons on 4-manifolds.

\subsection{The Excision Theorem}
\label{or31}

The next theorem is proved by the last author \cite[Th.~2.13]{Upme}, based on Donaldson \cite[\S II.4]{Dona1}, \cite[\S 3(b)]{Dona2}, and~\cite[\S 7.1.6]{DoKr}.

\begin{thm}[Excision Theorem]
\label{or3thm1}
Suppose we are given the following data:
\begin{itemize}
\setlength{\itemsep}{0pt}
\setlength{\parsep}{0pt}
\item[{\bf(a)}] Compact\/ $n$-manifolds $X^+,X^-$.
\item[{\bf(b)}] Elliptic complexes $E_\bu^\pm$ on $X^\pm$.
\item[{\bf(c)}] A Lie group $G,$ and principal\/ $G$-bundles $P^\pm\ra X^\pm$ with connections\/~$\nabla_{P^\pm}$.
\item[{\bf(d)}] Open covers $X^+=U^+\cup V^+,$ $X^-=U^-\cup V^-$.
\item[{\bf(e)}] A diffeomorphism $\io:U^+\ra U^-,$ such that $E_\bu^+\vert_{U^+}$ and\/ $\io^*(E_\bu^-\vert_{U^-})$ are isomorphic elliptic complexes on $U^+$.
\item[{\bf(f)}] An isomorphism $\si:P^+\vert_{U^+}\ra \io^*(P^-\vert_{U^-})$ of principal\/ $G$-bundles over $U^+,$ which identifies $\nabla_{P^+}\vert_{U^+}$ with\/~$\io^*(\nabla_{P^-}\vert_{U^-})$.
\item[{\bf(g)}] Trivializations of principal\/ $G$-bundles\/ $\tau^\pm:P^\pm\vert_{V^\pm}\ra V^\pm\t G$ over\/ $V^\pm,$ which identify\/ $\nabla_{P^\pm}\vert_{V^\pm}$ with the trivial connections\/ $\nabla^0,$ and satisfy 
\begin{equation*}
\smash{\io\vert_{U^+\cap V^+}^*(\tau^-)\ci\si\vert_{U^+\cap V^+}=\tau^+\vert_{U^+\cap V^+}}.
\end{equation*}
\end{itemize}
Then we have a canonical identification of n-orientation\/ $\Z_2$-torsors from\/~{\rm\eq{or1eq4}:}
\e
\Om^{+-}:\check O_{P^+}^{E_\bu^+}\big\vert_{[\nabla_{P^+}]}\,{\buildrel\cong\over\longra}\,\check O_{P^-}^{E_\bu^-}\big\vert_{[\nabla_{P^-}]}.
\label{or3eq1}
\e

The isomorphisms \eq{or3eq1} are functorial in a very strong sense. For example:
\begin{itemize}
\setlength{\itemsep}{0pt}
\setlength{\parsep}{0pt}
\item[{\bf(i)}] If we vary any of the data in {\bf(a)}--{\bf(g)} continuously in a family over $t\in[0,1],$ then the isomorphisms $\Om^{+-}$ also vary continuously in $t\in[0,1]$. 
\item[{\bf(ii)}] The isomorphisms $\Om^{+-}$ are unchanged by shrinking the open sets $U^\pm,V^\pm$ such that\/ $X^\pm=U^\pm\cup V^\pm$ still hold, and restricting $\io,\si,\tau^\pm$.
\item[{\bf(iii)}] If we are also given a compact\/ $n$-manifold\/ $X^\t,$ elliptic complex $E_\bu^\t,$ bundle $P^\t\ra X^\t,$ connection $\nabla_{P^{\smash{\t}}},$ open cover $X^\t=U^\t\cup V^\t,$ diffeomorphism $\io':U^-\ra U^\t,$ and isomorphisms  $\si':P^-\vert_{U^-}\ra \io^{\prime*}(P^\t\vert_{U^\t}),$ $\tau^\t:P^\t\vert_{V^\t}\ra V^\t\t G$ satisfying the analogues of\/ {\bf(a)\rm--\bf(g)\rm,} then\/ $\Om^{+\t}=\Om^{-\t}\ci\Om^{+-},$ where $\Om^{+\t}$ is defined using $\io'\ci\io:U^+\ra U^\t$ and\/~$\io^*(\si')\ci\si:P^+\vert_{U^+}\ra (\io'\ci\io)^*(P^\t\vert_{U^\t})$.
\end{itemize}
\end{thm}

\begin{proof}[Sketch proof] On $X^\pm$, consider the elliptic operator $D^{\nabla_{\smash{\Ad(P^\pm)}}}\op (D^{\nabla^0_{\smash{\Ad(X^\pm\t G)}}})^*$, where $D^{\nabla^0_{\smash{\Ad(X^\pm\t G)}}}$ is the twisted elliptic operator \eq{or1eq2} from the trivial bundle $X^\pm\t G\ra X^\pm$ with the trivial connection $\nabla^0$, and $(D^{\nabla^0_{\smash{\Ad(X^\pm\t G)}}})^*$ is its formal adjoint. This has determinant line
\begin{equation*}
\det\bigl(D^{\nabla_{\smash{\Ad(P^\pm)}}}\bigr)\ot \det\bigl(D^{\nabla^0_{\smash{\Ad(X^\pm\t G)}}}\bigr)^*,
\end{equation*}
and thus by \eq{or1eq4} has orientation $\Z_2$-torsor
\begin{equation*}
\check O_{P^\pm}^{E_\bu^\pm}\big\vert_{[\nabla_{P^\pm}]}=O_{P^\pm}^{E_\bu^\pm}\big\vert_{[\nabla_{P^\pm}]}\ot_{\Z_2}O_{X^\pm\t G}^{E_\bu^\pm}\big\vert_{[\nabla^0]},
\end{equation*}
as in the left and right hand sides of \eq{or3eq1}. Using the isomorphisms $\tau^\pm$ in (g), we may deform $D^{\nabla_{\smash{\Ad(P^\pm)}}}\op (D^{\nabla^0_{\smash{\Ad(X^\pm\t G)}}})^*$ continuously through elliptic {\it pseudo\/}-differential operators on $X^\pm$ to operators supported on $U^\pm$, and arrange that these operators on $U^+$ and $U^-$ are identified by $\io:U^+\ra U^-$. Since orientation torsors also work for elliptic pseudo-differential operators, and are unchanged under continuous deformations, the identification \eq{or3eq1} follows.
\end{proof}  

Here is a refinement of Problem \ref{or1prob} for Gauge Orientation Problems:

\begin{prob}
\label{or3prob}	
Suppose we are given a Gauge Orientation Problem as in Definition\/ {\rm\ref{or1def3}} and Example\/ {\rm\ref{or1ex1}}. Then for all compact\/ $n$-manifolds $X$ with geometric structure $\T$ of the prescribed kind, and all principal\/ $G$-bundles $P\ra X$ for $G\in\cL,$ we should construct a canonical n-orientation on $\B_P,$ such that:
\begin{itemize}
\setlength{\itemsep}{0pt}
\setlength{\parsep}{0pt}
\item[{\bf(i)}] The n-orientations are functorial under isomorphisms of\/ $(X,\T,P),$ and change continuously under continuous deformations of\/~$\T$.
\item[{\bf(ii)}] In the situation of Theorem\/ {\rm\ref{or3thm1},} if the diffeomorphism $\io:U^+\ra U^-$ identifies the geometric structures $\T^+\vert_{U^+}$ and\/ $\T^-\vert_{U^-},$ then\/ $\Om^{+-}$ in\/ \eq{or3eq1} identifies the canonical n-orientations on $\B_{P^+}$ and\/ $\B_{P^-}$ at\/~$[\nabla_{P^\pm}]$. 
\end{itemize}
\end{prob}

Here part (ii) is a strong condition: in some cases it may determine the canonical n-orientations more-or-less uniquely for all $(X,\T)$ and $P\ra X$, though in other cases it can be overdetermined, so no such canonical n-orientations exist.

We have two powerful methods for trivializing bundles $\check O_P^{E_\bu}\ra\B_P$: when the symbol of $E_\bu$ is complex linear as in Theorem \ref{or2thm1}, and excision, Theorem~\ref{or3thm1}. The next theorem relates these methods. In Theorems \ref{or4thm2}, \ref{or4thm4}, and \ref{or4thm8} below we will use the two methods in combination, in a way we believe is new.

\begin{thm}{\bf(a)} Suppose we are given data\/ $X^+=U^+\cup V^+,E_\bu^+,G,\ab P^+,\ab\nabla_{P^+},\ab\tau^+$ as in Theorem\/ {\rm\ref{or3thm1}(a)--(g),} and we are also given a complex structure on $E_\bu^+\vert_{U^+},$ as in Theorem\/ {\rm\ref{or2thm1}}. Then there is a natural trivialization of\/ \hbox{$\Z_2$-torsors,} depending on the choice of complex structure on $E_\bu^+\vert_{U^+}\!:$
\e
\smash{\check O_{P^+}^{E_\bu^+}\big\vert_{[\nabla_{P^+}]}\cong \Z_2.}
\label{or3eq2}
\e
These isomorphisms \eq{or3eq2} are strongly functorial as in Theorem\/ {\rm\ref{or3thm1}(i)--(iv)}.
\smallskip

\noindent{\bf(b)} In {\bf(a)\rm,} suppose the complex structure on $E_\bu^+\vert_{U^+}$ is the restriction of a complex structure on $E_\bu^+$. Then Theorem\/ {\rm\ref{or2thm1}} gives a natural trivialization of\/ $\check O_{P^+}^{E_\bu^+},$ and\/ \eq{or3eq2} agrees with this at\/~$[\nabla_{P^+}]$.
\smallskip

\noindent{\bf(c)} In the situation of Theorem\/ {\rm\ref{or3thm1},} suppose we have complex structures on $E_\bu^\pm\vert_{U^\pm}$ which are identified by the isomorphism between $E_\bu^+\vert_{U^+}$ and\/ $\io^*(E_\bu^-\vert_{U^-})$ in Theorem\/ {\rm\ref{or3thm1}(e)}. Then \eq{or3eq2} for $X^\pm$ induce trivializations of the left and right hand sides of\/ {\rm\eq{or3eq1},} and\/ \eq{or3eq1} identifies these trivializations.

\label{or3thm2}	
\end{thm}

\begin{proof} For (a), in the sketch proof of Theorem \ref{or3thm1}, we explained that $\check O_{P^+}^{E_\bu^+}\vert_{[\nabla_{P^+}]}$ is the orientation $\Z_2$-torsor of the elliptic operator $D^{\nabla_{\smash{\Ad(P^\pm)}}}\op (D^{\nabla^0_{\smash{\Ad(X^\pm\t G)}}})^*$ on $X^+$, and we can deform this continuously through elliptic pseudo-differential operators on $X^+$ to an operator supported on $U^+$. Using the complex structure on $E_\bu^+\vert_{U^+}$, we can make this operator supported on $U^+$ $\C$-linear. Then as in the proof of Theorem \ref{or2thm1}, we get a trivialization of $\check O_{P^+}^{E_\bu^+}\big\vert_{[\nabla_{P^+}]}$, inducing the isomorphism \eq{or3eq2}. It is functorial as in the proofs of Theorems \ref{or2thm1} and~\ref{or3thm1}.

For (b), given a complex structure on $E_\bu^+$, we can take the continuous deformation from  $D^{\nabla_{\smash{\Ad(P^\pm)}}}\op (D^{\nabla^0_{\smash{\Ad(X^\pm\t G)}}})^*$ to a complex linear operator supported on $U^+$ to be the composition of two continuous deformations: first we deform $D^{\nabla_{\smash{\Ad(P^+)}}}$ and $D^{\nabla^0_{\smash{\Ad(X^\pm\t G)}}}$ through elliptic differential operators to $\C$-linear operators as used to construct the trivializations of $O_{P^+}^{E_\bu^+},O_{X^+\t G}^{E_\bu^+}\vert_{[\nabla^0]}$ in the proof of Theorem \ref{or2thm1}. Secondly, we deform through $\C$-linear elliptic pseudo-differential operators on $X^+$ to a $\C$-linear operator supported on $U^+$. Throughout the second deformation we have canonical orientations by $\C$-linearity, and (b) follows.

For (c), the trivializations of the left and right hand sides of \eq{or3eq1} from part (a) come from identifying them with the orientation $\Z$-torsors of $\C$-linear elliptic pseudo-differential operators supported on $U^+$ and $U^-$, where the $\C$-linearity is built using the  complex structures on $E_\bu^\pm\vert_{U^\pm}$. The isomorphism \eq{or3eq1} was proved by identifying both sides with orientation $\Z_2$-torsors of elliptic pseudo-differential operators supported on $U^+$ and $U^-$, and identifying these elliptic pseudo-differential operators under $\io:U^+\ra U^-$ using the isomorphism $E_\bu^+\vert_{U^+}\cong\io^*(E_\bu^-\vert_{U^-})$. Since this isomorphism identifies the complex structures on $E_\bu^\pm\vert_{U^\pm}$, we can take the isomorphism of elliptic pseudo-differential operators under $\io$ to identify the $\C$-linear structures, so \eq{or3eq1} identifies the corresponding trivializations from~\eq{or3eq2}.
\end{proof}

\subsection{\texorpdfstring{Trivializing principal bundles outside codimension $d$}{Trivializing principal bundles outside codimension d}}
\label{or32}

\begin{rem}
\label{or3rem1}
{\bf(a)} Suppose that $G$ is a Lie group, and $d\ge 2$ with homotopy groups $\pi_i(G)=0$ for $i=0,\ldots,d-2$. Then for $k=0,\ldots,d-1$, any principal $G$-bundle $P\ra\cS^k$ is trivial, as these are classified by $\pi_{k-1}(G)$. It follows that if $Z$ is a manifold or CW-complex of dimension $\le d-1$ then any principal $G$-bundle $P\ra Z$ is trivial.
\smallskip

\noindent{\bf(b)} Here are some facts about homotopy groups $\pi_i(G)$ for Lie groups $G$, which can be found in Borel~\cite{Bore}:
\begin{itemize}
\setlength{\itemsep}{0pt}
\setlength{\parsep}{0pt}
\item[{\bf(i)}]  $\pi_0(G)=0$ if $G$ is connected.
\item[{\bf(ii)}] $\pi_1(G)$ is abelian, and $\pi_1(G)=0$ if $G$ is simply-connected.
\item[{\bf(iii)}] $\pi_2(G)=0$ for any Lie group $G$.
\item[{\bf(iv)}] $\pi_3(G)\cong\Z^k$, where $k$ is the number of simple Lie group factors of $G$.
\end{itemize}

\noindent{\bf(c)} Combining {\bf(a)\rm,\bf(b)}, we see that for a Lie group $G$:
\begin{itemize}
\setlength{\itemsep}{0pt}
\setlength{\parsep}{0pt}
\item[{\bf(i)}] If $G$ is connected then $\pi_i(G)=0$ for $i=0,\ldots,d-2$ with $d=2$.
\item[{\bf(ii)}] If $G$ is connected and simply-connected then $\pi_i(G)=0$ for $i=0,\ldots,d-2$ with $d=4$. This does not hold for $d>4$ unless $G$ is contractible.
\end{itemize}
\end{rem}

We use this to show that we can trivialize principal $G$-bundles $P\ra X$ outside the $(n-d)$-skeleton of $X$ for $d=2$ or~4.

\begin{prop} 
\label{or3prop1}
Let\/ $X$ be a compact\/ $n$-manifold, $G$ a Lie group, and\/ $P\ra X$ a principal\/ $G$-bundle. Suppose that either\/ {\bf(i)} $G$ is connected and set\/ $d=2,$ or\/ {\bf(ii)} $G$ is connected and simply-connected and set\/ $d=4$. Then:
\begin{itemize}
\setlength{\itemsep}{0pt}
\setlength{\parsep}{0pt}
\item[{\bf(a)}] Let $\tau$ be any triangulation of\/ $X$ into smooth $n$-simplices $\si:\De_n\ra X,$ and let\/ $Y$ be the $(n-d)$-skeleton of\/ $\tau,$ that is, the union of all\/ $i$-dimensional faces of simplices in $\tau$ for $i\le n-d$. Then $Y$ is a closed subset of\/ $X,$ and a finite CW complex of dimension $n-d$.

We can find a trivialization $\Phi:P\vert_{X\sm Y}\,\smash{\buildrel\cong\over\longra}\, (X\sm Y)\t G$. 
\item[{\bf(b)}] Suppose $\tau_0,Y_0,\Phi_0$ and\/ $\tau_1,Y_1,\Phi_1$ are alternative choices in  {\bf(a)}. Then there exists a triangulation $\ti\tau$ of\/ $X\t[0,1]$ into smooth $(n+1)$-simplices, which restricts to $\tau_i$ on $X\t\{i\}$ for $i=0,1$. Let\/ $Z$ be the $(n+1-d)$-skeleton of\/ $\ti\tau$ relative to $X\t\{0,1\},$ i.e.\ the union of all\/ $i$-dimensional faces of simplices in $\ti\tau$ which have either $i\le n-d,$ or $i=n+1-d$ and do not lie wholly in\/ $X\t\{0,1\}$. Then $Z$ is closed in\/ $X\t[0,1],$ and a finite CW complex of dimension $n+1-d,$ with $Z\cap (X\t\{i\})=Y_i\t\{i\}$ for\/~$i=0,1$. 

We can find a trivialization $\Psi:\pi_X^*(P)\vert_{(X\t[0,1])\sm Z}\,\smash{\buildrel\cong\over\longra}\,((X\t[0,1])\sm Z)\t G,$ such that\/ $\Psi\vert_{(X\sm Y_i)\t\{i\}}=\Phi_i$ for\/~$i=0,1$.
\end{itemize}
\end{prop}

\begin{proof} For (a), given a triangulation $\tau$ of $X$, let $\tau'$ be the {\it barycentric subdivision\/} of $\tau$, that is, the subtriangulation that places an extra vertex at the barycentre of each $i$-simplex in $\tau$ for $i>0$, and divides each $k$-simplex $\si:\De_k\ra X$ in $\tau$ into $(k+1)!$ smaller $k$-simplices. Define $C\subset X$ to be the union of all $i$-simplices $\si'_i(\De_i)\subset X$ in $\tau'$ for $i=0,\ldots,d-1$ which meet an $(n-i)$-simplex $\si_{n-i}(\De_{n-i})$ in $\tau$ transversely at the barycentre of $\si_{n-i}(\De_{n-i})$. Then $C$ is closed in $X$, and is a CW complex of dimension~$d-1$. 

We can think of $C$ as the $(d-1)$-skeleton of the `dual triangulation' $\tau^*$ of $\tau$, though $\tau^*$ divides $X$ into polyhedra rather than simplices. For example, the icosahedron is a triangulation of $\cS^2$ into twenty 2-simplices, thirty 1-simplices and twelve 0 simplices. The `dual triangulation' is the dodecahedron, which divides $\cS^2$ into twenty 0-simplices, thirty 1-simplices, and twelve pentagons.

The important facts we need are that $C\cap Z=\es$, and $X\sm Y$ retracts onto $C$, since $X\sm Y$ is a union of interiors of $i$-simplices $\si'_i(\De_i)\subset X$ in $\tau'$ all of which have one face in $C$, and can be retracted onto $C$ in a natural way. As $C$ is a CW complex of dimension $d-1$, we see that $P\vert_C$ is trivial by Remark \ref{or3rem1}(a),(c). Since $X\sm Y$ retracts onto $C$, it follows that $P\vert_{X\sm Y}$ is trivial, proving~(a).
 
For (b), by standard facts about triangulations we can choose $\ti\tau$. Let $\ti\tau'$ be the barycentric subdivision of $\ti\tau$, and define $D\subset X\t[0,1]$ to be the union of all $i$-simplices $\ti\si'_i(\De_i)\subset X\t[0,1]$ in $\ti\tau'$ for $i=0,\ldots,d-1$ which meet an $(n+1-i)$-simplex $\ti\si_{n+1-i}(\De_{n+1-i})$ in $\ti\tau$ transversely at the barycentre of $\ti\si_{n+1-i}(\De_{n+1-i})$. Then $D$ is closed in $X\t[0,1]$, and is a CW complex of dimension $d-1$, with $D\cap X\t\{i\}=C_i\t\{i\}$ for $i=0,1$. As above we have $D\cap Z=\es$, and $(X\t[0,1])\sm Z$ retracts onto~$D$.

Since $D$ is a CW complex of dimension $d-1$, $P\vert_D$ is trivial by Remark \ref{or3rem1}(a),(c). Furthermore, as $(C_0\t\{0\})\amalg(C_1\t\{1\})$ is a CW-subcomplex of $D$, the trivializations $\Phi_i\vert_{C_i\t\{i\}}$ of $\pi_X^*(P)\vert_{C_i\t\{i\}}$ for $i=0,1$ can be extended to a single trivialization of $\pi_X^*(P)\vert_D$. As $(X\t[0,1])\sm Z$ retracts onto $D$, we can then extend the trivialization of $\pi_X^*(P)$ from $D$ to $\Psi:\pi_X^*(P)\vert_{(X\t[0,1])\sm Z}\ra ((X\t[0,1])\sm Z)\t G$, such that $\Psi\vert_{(X\sm Y_i)\t\{i\}}=\Phi_i$ for $i=0,1$.
\end{proof}

\subsection{A general method for solving Problem \ref{or3prob}}
\label{or33}

Suppose we are given a Gauge Orientation Problem, as in Definition \ref{or1def3} and Example \ref{or1ex1}. We will take the family $\cL$ of Lie groups $G$ to be either:
\begin{itemize}
\setlength{\itemsep}{0pt}
\setlength{\parsep}{0pt}
\item[(a)] all connected Lie groups $G$, so Proposition \ref{or3prop1} applies with $d=2$; or
\item[(b)] all connected, simply-connected Lie groups $G$, so Proposition \ref{or3prop1} applies with $d=4$.
\item[(c)] $\cL$ is $\bigl\{\SU(m):m=1,2,\ldots\bigr\}$, so Proposition \ref{or3prop1} applies with $d=4$, and we can use results on stabilization and K-theory.
\end{itemize}

We now outline a strategy for solving Problem~\ref{or3prob}:
\smallskip

\noindent{\bf Step 1.} Suppose for simplicity that $\cS^n$ admits a geometric structure $\T'$ of the prescribed kind, and that $\T'$ is unique up to isotopy. Prove that when $X=\cS^n$, all moduli spaces $\B_P$ are n-orientable. 
\smallskip

\noindent{\bf Step 2.} Choose n-orientations for all moduli spaces $\B_P$ when $X=\cS^n$.
\smallskip

\noindent{\bf Step 3.} Let $(X,\T)$, $G$ and $P\ra X$ be as in Definition \ref{or1def3}, and let $\nabla_P$ be a connection on $P$. By Proposition \ref{or3prop1}(a) we can choose an $(n-d)$-skeleton $Y\subset X$ and a trivialization $\Phi:P\vert_{X\sm Y}\ra(X\sm Y)\t G$. Choose a small open neighbourhood $U$ of $Y$ in $X$ such that $U$ retracts onto $Y$. Choose an open $V\subset X$ with $\ov V\subseteq X\sm Y$ and~$U\cup V=X$.

Choose a connection $\hat\nabla_P$ on $P\ra X$ which is trivial over $V\subset X\sm Y$, using the chosen trivialization $\Phi$ of $P\vert_{X\sm Y}$. Choose an embedding $\io:U\hookra \cS^n$ of $U$ as an open submanifold of $\cS^n$, if this is possible, and a geometric structure $\T'$ on $\cS^n$ of the prescribed kind, such that $\io^*(\T')\cong \T\vert_U$. Set $U'=\io(U)$ and $V'=\io(U\cap V)\amalg (\cS^n\sm U')$. Then $V'\subset\cS^n$ is open with~$U'\cup V'=\cS^n$.

Define a principal $G$-bundle $P'\ra\cS^n$, such that $P'\vert_{U'}$ is identified with $P\vert_U$ under $\io:U\ra U'$, and $P'\vert_{V'}$ is trivial, and on the overlap $U'\cap V'$, the identification matches the given trivialization $\Phi\vert_{U\cap V}$ of $P\vert_U$ on $U\cap V$ under~$\io$.

Define a connection $\hat\nabla_{P'}$ on $P'\ra\cS^n$, such that $\hat\nabla_{P'}\vert_{U'}$ is identified with $\hat\nabla_P\vert_U$ under the identification of $P'\vert_{U'}$ with $P\vert_U$ under $\io:U\ra U'$, and $\hat\nabla_{P'}$ is trivial over $V'$, using the chosen trivialization of $P'\vert_{V'}$. 

Theorem \ref{or3thm1} now gives an isomorphism of $\Z_2$-torsors
\e
\smash{\Om^{X\cS^n}:\check O_{P}^{E_\bu}\big\vert_{[\hat\nabla_{P}]}\,{\buildrel\cong\over\longra}\,\check O_{P'}^{E_\bu'}\big\vert_{[\hat\nabla_{P'}]}.}
\label{or3eq3}
\e
The right hand side has a chosen n-orientation by Step 2, and Problem \ref{or3prob}(ii) requires $\Om^{X\cS^n}$ to be n-orientation-preserving, so this gives an n-orientation of $O_P^{E_\bu}\big\vert_{[\hat\nabla_{P}]}$. We then determine the n-orientation of $O_{P}^{E_\bu}\big\vert_{[\nabla_{P}]}$ by choosing a smooth path from $\hat\nabla_P$ to $\nabla_P$ in the (contractible) space of connections on $P\ra X$, and deforming the n-orientation continuously along this path.

This constructs an n-orientation at any point $[\nabla_P]$ in $\B_P$, for any $(X,\T)$, Lie group $G$, and principal $G$-bundle $P\ra X$, which is uniquely determined by Step 2 and Problem~\ref{or3prob}(ii).
\smallskip

\noindent{\bf Step 4.} Prove that the n-orientation on $\B_P$ at $[\nabla_P]$ in Step 3 is independent of all arbitrary choices in its construction, and so is well defined.

Start with two sets of choices $Y_0,\Phi_0,U_0,\ab V_0,\ab\io_0,\ab\hat\nabla_{P,0}$ and $Y_1,\ab\Phi_1,\ab U_1,\ab V_1,\ab\io_1,\ab\hat\nabla_{P,1}$ in Step 3. We use Proposition \ref{or3prop1}(b) to get an $(n-d)$-skeleton $Z\subset X\t[0,1]$ interpolating between $Y_0$ and $Y_1$, and a trivialization $\Psi:\pi_X^*(P)\vert_{(X\t[0,1])\sm Z}\ra ((X\t[0,1])\sm Z)\t G$ interpolating between $\Phi_0$ and $\Phi_1$. We choose a small open neighbourhood $W$ of $Z$ in $X\t[0,1]$ which interpolates between $U_0$ and $U_1$ and retracts onto $Z$. Then we construct data on $X\t[0,1]$ and $\cS^n\t[0,1]$ interpolating between $V_0,\io_0,\hat\nabla_{P,0}$ and $V_1,\io_1,\hat\nabla_{P,1}$, if this is possible. This gives a continuous family of excision problems from $X$ to $\cS^n$ parametrized by $t\in[0,1]$, yielding a 1-parameter family of isomorphisms \eq{or3eq3}. Theorem \ref{or3thm1}(i) says that these depend continuously on $t\in[0,1]$. Thus the n-orientations on $\B_P$ at $[\nabla_P]$ determined by $Y_0,\ldots,\hat\nabla_{P,0}$ and $Y_1,\ldots,\hat\nabla_{P,1}$ are joined by a continuous family over $t\in[0,1]$, so they are equal, and independent of choices.
\smallskip

\noindent{\bf Step 5.} Steps 1--4 give canonical n-orientations $\check\om_P$ on all moduli spaces $\B_P$ in the Gauge Orientation Problem. Finally, we show that these n-orientations satisfy any other properties that we want, e.g.\ Problem \ref{or3prob}(i)--(ii), or comparison of n-orientations for $\U(m)$-bundles under direct sums with given signs, as in~\S\ref{or25}.

\smallskip

We will use versions of this method in the proofs of Theorems \ref{or4thm1}, \ref{or4thm2}, \ref{or4thm4}, and \ref{or4thm8} below (though not using $\cS^n$ as the model space in Steps 1 and 2), and in the sequels~\cite{CGJ,JoUp}.

\section{Application to orientations in gauge theory}
\label{or4}

We now apply the ideas of \S\ref{or2}--\S\ref{or3} to construct canonical orientations on several classes of gauge theory moduli spaces. Some of our results are new. We begin with a general discussion of gauge theory moduli spaces in~\S\ref{or41}.

\subsection{Orienting moduli spaces in gauge theory}
\label{or41}

In gauge theory one studies moduli spaces $\M_P^{\rm ga}$ of (irreducible) connections $\nabla_P$ on a principal bundle $P\ra X$ (perhaps plus some extra data, such as a Higgs field) satisfying a curvature condition. Under suitable genericity conditions, these moduli spaces $\M_P^{\rm ga}$ will be smooth manifolds, and the ideas of \S\ref{or2}--\S\ref{or3} can often be used to prove $\M_P^{\rm ga}$ is orientable, and construct a canonical orientation on $\M_P^{\rm ga}$. These orientations are important in defining enumerative invariants such as Casson invariants, Donaldson invariants, and Seiberg--Witten invariants. We illustrate this with the example of instantons on 4-manifolds,~\cite{DoKr}:

\begin{ex} Let $(X,g)$ be a compact, oriented Riemannian 4-manifold, and $G$ a Lie group (e.g.\ $G=\SU(2)$), and $P\ra X$ a principal $G$-bundle. For each connection $\nabla_P$ on $P$, the curvature $F^{\nabla_P}$ is a section of $\Ad(P)\ot\La^2T^*X$. We have $\La^2T^*X=\La^2_+T^*X\op\La^2_-T^*X$, where $\La^2_\pm T^*X$ are the subbundles of 2-forms $\al$ on $X$ with $*\al=\pm\al$. Thus $F^{\nabla_P}=F^{\nabla_P}_+\op F^{\nabla_P}_-$, with $F^{\nabla_P}_\pm$ the component in $\Ad(P)\ot\La^2_\pm T^*X$. We call $(P,\nabla_P)$ an ({\it anti-self-dual\/}) {\it instanton\/} if $F^{\nabla_P}_+=0$.

Write $\M_P^{\rm asd}$ for the moduli space of gauge isomorphism classes $[\nabla_P]$ of irreducible instanton connections $\nabla_P$ on $P$, modulo $\G_P/Z(G)$. The deformation theory of $[\nabla_P]$ in $\M_P^{\rm asd}$ is governed by the Atiyah--Hitchin--Singer complex \cite{AHS}: 
\e
\begin{gathered} 
\xymatrix@C=8pt@R=7pt{ 0 \ar[rr] && 
\Ga^{\iy} ( \Ad(P) \ot \La^0T^*X ) \ar[rrr]^{\d^{\nabla_P}} &&& 
\Ga^{\iy} ( \Ad(P) \ot \La^1T^*X  ) \\
&& {\qquad\qquad\qquad} \ar[rrr]^{\d^{\nabla_P}_+} &&&
\Ga^{\iy} ( \Ad(P) \ot \La^2_+T^*X  ) \ar[rr] && 0, } 
\end{gathered}
\label{or4eq1}
\e
where $\d^{\nabla_P}_+\ci\d^{\nabla_P}=0$ as $F^{\nabla_P}_+=0$. Write $\cH^0,\cH^1,\cH^2_+$ for the cohomology groups of \eq{or4eq1}. Then $\cH^0$ is the Lie algebra of $\Aut(\nabla_P)$, so $\cH^0=Z(\g)$, the Lie algebra of the centre $Z(G)$ of $G$, as $\nabla_P$ is irreducible. Also $\cH^1$ is the Zariski tangent space of $\M_P^{\rm asd}$ at $[\nabla_P]$, and $\cH^2_+$ is the obstruction space. If $g$ is generic then for non-flat connections $\cH^2_+=0$ for all $\nabla_P$, as in \cite[\S 4.3]{DoKr}, and $\M_P^{\rm asd}$ is a smooth manifold, with tangent space $T_{[\nabla_P]}\M_P^{\rm asd}=\cH^1$. Note that $\M_P^{\rm asd}\subset\ovB_P$ is a subspace of the topological stack $\ovB_P$ from Definition~\ref{or1def1}.

Take $E_\bu$ to be the elliptic operator on $X$
\e
D=\d+\d_+^*:\Ga^\iy(\La^0T^*X\op\La^2_+T^*X)\longra\Ga^\iy(\La^1T^*X).	
\label{or4eq2}
\e
Turning the complex \eq{or4eq1} into a single elliptic operator as in Remark \ref{or2rem2}(ii) yields the twisted operator $D^{\nabla_{\Ad(P)}}$ from \eq{or1eq2}. Hence the line bundle $\bar L^{E_\bu}_P\ra\ovB_P$ in Definition \ref{or1def2} has fibre at $[\nabla_P]$ the determinant line of \eq{or4eq1}, which (after choosing an isomorphism $\det Z(\g)\cong\R$) is $\det(\cH^1)^*=\det T^*_{[\nabla_P]}\M_P^{\rm asd}$. It follows that $\bar O_P^{E_\bu}\vert_{\M_P^{\rm asd}}$ is the orientation bundle of the manifold $\M_P^{\rm asd}$, and an orientation on $\ovB_P$ in Definition \ref{or1def2} (which is equivalent to an orientation on $\B_P$) restricts to an orientation on the manifold $\M_P^{\rm asd}$ in the usual sense of differential geometry. This is a very useful way of defining orientations on~$\M_P^{\rm asd}$.

\label{or4ex1}	
\end{ex}

There are several other important classes of gauge-theoretic moduli spaces $\M_P^{\rm ga}$ which have elliptic deformation theory, and so are generically smooth manifolds, for which orientations can be defined by pullback from $\ovB_P$. See Reyes Carri\'on \cite{ReCa} for a study of instanton-type equations governed by complexes generalizing \eq{or4eq1}. We can also generalize the programme above in three ways:

\begin{rem}{\bf(i)} For example, for `$G_2$-instantons' on a 7-manifold $(X,\vp,g)$ with holonomy $G_2$ we replace \eq{or4eq1} by a {\it four term\/} elliptic complex:
\e
\begin{gathered} 
\xymatrix@C=5pt@R=7pt{ 0 \ar[rr] && 
\Ga^{\iy} ( \Ad(P) \ot \La^0T^*X ) \ar[rrr]^{\d^{\nabla_P}} &&& 
\Ga^{\iy} ( \Ad(P) \ot \La^1T^*X  ) \\
\ar[rr]^(0.2){\d^{\nabla_P}_7} && {\Ga^{\iy} ( \Ad(P) \ot \La^2_7T^*X  )} \ar[rrr]^{*\vp\w\d^{\nabla_P}} &&&
\Ga^{\iy} ( \Ad(P) \ot \La^7T^*X  ) \ar[rr] && 0, } 
\end{gathered}
\label{or4eq3}
\e
where exactness follows from $\pi^2_7(F^{\nabla_P})=0$, $\d(*\vp)=0$, and the Bianchi identity. The cohomology at the fourth term is dual to the cohomology at the first term, and so is $Z(\g)^*$ for irreducible connections. Because of this, $G_2$-instanton moduli spaces $\M_P^{G_2}$ are generically manifolds with well-behaved orientations. Flat connections on 3-manifolds are similar.
\smallskip

\noindent{\bf(ii)} Many interesting problems involve moduli spaces $\M_P^{\rm Hi}$ of pairs $(\nabla_P,H)$, where $\nabla_P$ is a connection on $P\ra X$, and $H$ is some extra data, such as a {\it Higgs field}, a section of a vector bundle on $X$ defined using $P$, where $(\nabla_P,H)$ satisfy some p.d.e. Under good conditions $\M_P^{\rm Hi}$ is a manifold, and the orientation bundle of $\M_P^{\rm Hi}$ is the pullback of an orientation bundle $O_P^{E_\bu}\ra\B_P$ or $\bar O_P^{E_\bu}\ra\ovB_P$ under the forgetful map $\M_P^{\rm Hi}\ra\B_P$ or $\M_P^{\rm Hi}\ra\ovB_P$, $[\nabla_P,H]\mapsto[\nabla_P]$.
\smallskip

\noindent{\bf(iii)} If we omit the genericness/transversality conditions, gauge theory moduli spaces $\M_P^{\rm ga}$ are generally not smooth manifolds. However, as long as their deformation theory is given by an elliptic complex similar to \eq{or4eq1} or \eq{or4eq3} whose cohomology is constant except at the second and third terms, $\M_P^{\rm ga}$ will still be a {\it derived smooth manifold\/} ({\it d-manifold}, or {\it m-Kuranishi space\/}) in the sense of Joyce \cite{Joyc1,Joyc3,Joyc4,Joyc5}. Orientations for derived manifolds are defined and well behaved, and we can define orientations on $\M_P^{\rm ga}$ by pullback of orientations on $\ovB_P$ exactly as in the case when $\M_P^{\rm ga}$ is a manifold.
\label{or4rem1}
\end{rem}

\subsection{Examples of orientation problems}
\label{or42}

We now give a series of examples of gauge theory moduli spaces we can orient using our techniques, in dimensions~$n=2,\ldots,6$.

\subsubsection{Flat connections on 2-manifolds}
\label{or421}

Let $X$ be a compact 2-manifold, $G$ a Lie group, and $P\ra X$ a principal $G$-bundle. Consider the moduli space $\M_P^{\rm fl}$ of irreducible flat connections $\nabla_P$ on $P$. Then (at least if $X$ is orientable) $\M_P^{\rm fl}$ is a smooth manifold. The deformation theory of $[\nabla_P]$ in $\M_P^{\rm fl}$ is controlled by the elliptic complex
\begin{equation*} 
\xymatrix@C=8pt@R=7pt{ 0 \ar[rr] && 
\Ga^{\iy} ( \Ad(P) \ot \La^0T^*X ) \ar[rrr]^{\d^{\nabla_P}} &&& 
\Ga^{\iy} ( \Ad(P) \ot \La^1T^*X  ) \\
&& {\qquad\qquad\qquad} \ar[rrr]^{\d^{\nabla_P}} &&&
\Ga^{\iy} ( \Ad(P) \ot \La^2T^*X  ) \ar[rr] && 0, } 
\end{equation*}
with $\det T^*\M_P^{\rm fl}$ the determinant line of this complex.

Choose a Riemannian metric $g$ on $X$, and take $E_\bu$ to be the elliptic operator
\begin{equation*}
D=\d\op\d^*:\Ga^\iy(\La^0T^*X\op\La^2T^*X)\longra\Ga^\iy(\La^1T^*X).	
\end{equation*}
Then the orientation bundle of the manifold $\M_P^{\rm fl}$ is the pullback under the inclusion $\M_P^{\rm fl}\hookra\ovB_P$ of the bundle $\bar O_P^{E_\bu}\ra\ovB_P$ from Definition~\ref{or1def2}. 

In the next theorem, part (a) is proved by Freed, Hopkins and Teleman \cite[\S 3]{FHT} as part of their construction of a 2-d TQFT, but (b) may be new.

\begin{thm}{\bf(a)} For compact, oriented\/ $2$-manifolds $X,$ the moduli spaces $\M_P^{\rm fl}$ above have canonical orientations for all\/ $G$ and\/ $P\ra X$.
\smallskip

\noindent{\bf(b)} If\/ $X$ is not oriented, then after choosing orientations for\/ $\g$ and\/ $\det D$ we can define a canonical orientation on $\M_P^{\rm fl}$ if\/ $G$ is any connected, simply-connected Lie group, or if\/ $G=\U(m)$. Also $\M_P^{\rm fl}$ is orientable if\/~$G=\SO(3)$.

\label{or4thm1}
\end{thm}

\begin{proof} Part (a) holds by Theorem \ref{or2thm1}, as if $X$ is oriented then there are complex structures on $E_0\cong X\t\C$ and $E_1\cong T^{*(0,1)}X$ for which the symbol of $D\cong\bar\partial$ is complex linear. Part (b) for $G$ connected and simply-connected works by the method of \S\ref{or33} with $d=4$ in a trivial way, as $Y=Z=\es$ in Steps 3 and 4 for dimensional reasons, so Steps 1 and 2 are unnecessary. Part (b) for $G=\U(m)$ then follows from Example \ref{or2ex5}, and for $G=\SO(3)$ from Proposition \ref{or2prop2}, noting that $H^3(X,\Z)=0$ as~$\dim X=2$. 
\end{proof}

\subsubsection{Flat connections on 3-manifolds, and Casson invariants}
\label{or422}

Let $X$ be a compact 3-manifold, $G$ a Lie group, and $P\ra X$ a principal $G$-bundle. Consider the moduli space $\M_P^{\rm fl}$ of irreducible flat connections $\nabla_P$ on $P$. In contrast to the 2-dimensional case, $\M_P^{\rm fl}$ is generally {\it not\/} a smooth manifold. However, (at least if $X$ is orientable) $\M_P^{\rm fl}$ is a {\it derived manifold\/} of virtual dimension 0, as in Remark \ref{or4rem1}(iii), so orientations for $\M_P^{\rm fl}$ make sense. As in Remark \ref{or4rem1}(i), the deformation theory of $[\nabla_P]$ in $\M_P^{\rm fl}$ is controlled by the four term elliptic complex
\begin{equation*} 
\xymatrix@C=5pt@R=7pt{ 0 \ar[rr] && 
\Ga^{\iy} ( \Ad(P) \ot \La^0T^*X ) \ar[rrr]^{\d^{\nabla_P}} &&& 
\Ga^{\iy} ( \Ad(P) \ot \La^1T^*X  ) \\
\ar[rr]^(0.2){\d^{\nabla_P}} && {\Ga^{\iy} ( \Ad(P) \ot \La^2T^*X  )} \ar[rrr]^{\d^{\nabla_P}} &&&
\Ga^{\iy} ( \Ad(P) \ot \La^3T^*X  ) \ar[rr] && 0, } 
\end{equation*}
which should have constant cohomology at the first and fourth terms for $\M_P^{\rm fl}$ to be a derived manifold.

The moduli spaces $\M_P^{\rm fl}$ are studied in connection with the Casson invariant of 3-manifolds, as in Akbulut and McCarthy \cite{AkMc}. Casson originally defined a $\Z$-valued invariant $\Cass(X)$ of an oriented integral homology 3-sphere $X$ using a Heegard splitting of $X$. Later, Taubes \cite{Taub1} provided an alternative definition of $\Cass(X)$ as a virtual count of $\M_P^{\rm fl}$ for $P\ra X$ the trivial $\SU(2)$-bundle.

The theory of Casson invariants has been generalized in several directions. As in Donaldson \cite{Dona3}, $\Cass(X)$ is the Euler characteristic of the $\SU(2)$-instanton Floer homology groups of $X$. Boden and Herald \cite{BoHe} defined an invariant for homology 3-spheres as a virtual count of $\M_P^{\rm fl}$ for $P\ra X$ the trivial $\SU(3)$-bundle, and there are extensions to 3-manifolds other than homology 3-spheres using flat connections on $\U(2)$-bundles and $\SO(3)$-bundles,~\cite[\S 5.6]{Dona3}.

Choose a Riemannian metric $g$ on $X$, and take $E_\bu$ to be the elliptic operator
\begin{equation*}
D=\d+\d^*:\Ga^\iy(\La^0T^*X\op\La^2T^*X)\longra\Ga^\iy(\La^1T^*X\op\La^3T^*X).	
\end{equation*}
Then the orientation bundle of the derived manifold $\M_P^{\rm fl}$ is the pullback under the inclusion $\M_P^{\rm fl}\hookra\ovB_P$ of the bundle $\bar O_P^{E_\bu}\ra\ovB_P$ from Definition~\ref{or1def2}. 

Note that $X$ need not be orientable in the next theorem.

\begin{thm}{\bf(a)} In the situation above, suppose $\al\in\Ga^\iy(\La^2T^*X)$ is a nonvanishing $2$-form with\/ $\md{\al}_g\equiv 1$. Such\/ $\al$ exist for any compact Riemannian $3$-manifold\/ $(X,g)$. Then there are unique complex vector bundle structures on $E_0,E_1$ such that\/ $i\cdot 1=\al$ in $E_0$ and the symbol of\/ $D$ is complex linear. 

Thus for any Lie group $G$ and principal\/ $G$-bundle $P\ra X,$ Theorem\/ {\rm\ref{or2thm1}} defines a canonical n-orientation on\/ $\B_P$.
\smallskip

\noindent{\bf(b)} If\/ $G$ is connected then the n-orientation on $\B_P$ in {\bf(a)} is independent of\/~$\al$.

\label{or4thm2}	
\end{thm}

\begin{proof} For (a), at a point $x\in X$, choose an orthonormal basis $e_1,e_2,e_3$ for $T_x^*X$ with $\al\vert_x=e_1\w e_2$. Then define $\C$-vector space structures on $E_0\vert_x,E_1\vert_x$ by
\begin{equation*}
i\cdot 1=e_1\w e_2=\al\vert_x, \;\> i\cdot e_1\w e_3=e_2\w e_3, \;\> i\cdot e_1=e_2, \>\; i\cdot e_3=e_1\w e_2\w e_3.
\end{equation*}
It is easy to check that these are independent of the choice of $(e_1,e_2,e_3)$, so over all $x\in X$ they extend to complex vector bundle structures on $E_0,E_1$, and the symbol of $D$ is complex linear. Part (a) follows.

For (b), let $(X,g)$ be a compact Riemannian 3-manifold, $G$ a connected Lie group, and $P\ra X$ a principal $G$-bundle. By Proposition \ref{or3prop1}(a) with $d=2$ we can choose a 1-skeleton $Y\subset X$, a CW complex of dimension 1, such that $P\vert_{X\sm Y}$ is trivial. As in Step 3 of \S\ref{or33}, choose a small open neighbourhood $U$ of $Y$ in $X$ such that $U$ retracts onto $Y$, and an open $V\subset X$ with $\ov V\subseteq X\sm Y$ and $U\cup V=X$, and a connection $\nabla_P$ on $P\ra X$ which is trivial over $V\subset X\sm Y$, using the chosen trivialization of~$P\vert_{X\sm Y}$. 

Suppose $\al^+,\al^-\in \Ga^\iy(\La^2T^*X)$ are nonvanishing 2-forms with $\md{\al^\pm}_g=1$. In general, $\al^+,\al^-$ are not isotopic through nonvanishing 2-forms on $X$. However, $\al^+\vert_Y,\al^-\vert_Y$ are isotopic through nonvanishing sections of $\La^2T^*X\vert_Y$ over the 1-skeleton $Y$, as $\La^2T^*X$ has rank 3. As $U$ retracts onto $Y$, it follows that $\al^+\vert_U$ and $\al^-\vert_U$ are isotopic through nonvanishing sections of $\La^2T^*X\vert_U$. So after a continuous deformation of $\al^-$, we can suppose that~$\al^+\vert_U=\al^-\vert_U$.

Part (a) gives complex structures $J^\pm$ on $E_\bu$ coming from $\al^\pm$, with $J^+\vert_U=J^-\vert_U$ as $\al^+\vert_U=\al^-\vert_U$. Write $\check\om_P^+,\check\om_P^-$ for the n-orientations on $\B_P$ given by Theorem \ref{or2thm1} using $J^+,J^-$. Theorem \ref{or3thm2}(a) with $X,U,V,E_\bu,P,\nabla_P$ in place of $X^+,\ab U^+,\ab V^+,\ab P^+,\ab\nabla_{P^+}$ now gives an isomorphism $\check O_P^{E_\bu}\vert_{[\nabla_P]}\cong \Z_2$, depending only on the complex structure $J^+\vert_U=J^-\vert_U$ on $E_\bu\vert_U$. Theorem \ref{or3thm2}(b) for $J^+$ and $J^-$ implies that this agrees with $\check\om_P^+\vert_{[\nabla_P]}$ and $\check\om_P^-\vert_{[\nabla_P]}$, so $\check\om_P^+=\check\om_P^-$ as $\B_P$ is connected. This proves~(b).
\end{proof}

To pass from an n-orientation of $\B_P$ to an orientation of $\B_P$, by \eq{or1eq4} and \eq{or1eq7} we need to choose an orientation for $\det D$, noting that as $\ind D=0$ on any 3-manifold $X$ we do not need an orientation on $\g$. Since orientations on $\B_P$ induce orientations on $\ovB_P$, which restrict to orientations of $\M_P^{\rm fl}$, we deduce:

\begin{thm} Let\/ $(X,g)$ be a compact Riemannian $3$-manifold, and choose an orientation for $\det D$ (equivalently, an orientation on $\bigop_{k=0}^3H^k(X,\R)$). Then for any connected Lie group $G$ and principal\/ $G$-bundle $P\ra X$ we can construct a canonical orientation for the derived manifold\/~$\M_P^{\rm fl}$.
\label{or4thm3}	
\end{thm}

Here is how this relates to results in the literature: when $X$ is oriented and $G=\SU(2)$, Taubes \cite[Prop.~2.1]{Taub1} shows $\B_P$ is orientable, and then uses the standard orientation to get canonical orientations as any $\SU(2)$-bundle $P\ra X$ is trivial, \cite[p.~554-5]{Taub1}. Boden and Herald \cite[\S 4]{BoHe} prove the analogue for $G=\SU(3)$. We believe Theorem \ref{or4thm3} may be new for non-orientable $X$, and also for non-simply-connected $G$, when $P\ra X$ need not be trivial, so standard orientations do not suffice to define canonical orientations.

\subsubsection{Anti-self-dual instantons on 4-manifolds}
\label{or423}

Let $(X,g)$ be a compact, oriented Riemannian 4-manifold, $G$ a Lie group, and $P\ra X$ a principal $G$-bundle. In Example \ref{or4ex1} we defined the moduli space $\M_P^{\rm asd}$ of irreducible anti-self-dual instantons on $P$, and explained that for generic $g$ it is a smooth manifold, whose orientation bundle is the pullback of $\bar O_P^{E_\bu}\ra\ovB_P$ from Definition \ref{or1def2} under the inclusion $\M_P^{\rm asd}\hookra\ovB_P$, for $E_\bu$ as in~\eq{or4eq2}. 

Instanton moduli spaces $\M_P^{\rm asd}$ are used to define Donaldson invariants of 4-manifolds $X$, as in \cite{Dona1,Dona2,DoKr}, which can distinguish different smooth structures on homeomorphic 4-manifolds $X_1,X_2$. Most work in the area focusses on $G=\SU(2)$ and $G=\SO(3)$, although Kronheimer \cite{Kron} extends the definition to $G=\SU(m)$ and $\mathop{\rm PSU}(m)$. Orientations on $\M_P^{\rm asd}$ are needed to determine the sign of the Donaldson invariants, and have been well studied. Many of the methods of \S\ref{or2}--\S\ref{or3} were first used in Donaldson theory. 

\begin{thm}
\label{or4thm4}
Let\/ $(X,g)$ be a compact, oriented Riemannian $4$-manifold, and\/ $D,E_\bu$ be as in~\eq{or4eq2}.
\begin{itemize}
\setlength{\itemsep}{0pt}
\setlength{\parsep}{0pt}
\item[{\bf(a)}] Suppose $J$ is an almost complex structure on $X$ which is Hermitian with respect to $g$ and compatible with the orientation. Then there are unique complex vector bundle structures on $E_0,E_1$ such that\/ $J$ acts on $E_1=T^*X$ by multiplication by $i,$ and the symbol of\/ $D$ is complex linear. 

Thus for any Lie group $G$ and principal\/ $G$-bundle $P\ra X,$ Theorem\/ {\rm\ref{or2thm1}} defines a canonical n-orientation on\/ $\B_P$.

\item[{\bf(b)}] Now let\/ $G$ be a connected Lie group, and choose a\/ $\Spinc$-structure $\mathfrak s$ on $X,$ noting that\/ $\Spinc$-structures exist for any $(X,g)$. Then for all principal\/ $G$-bundles $P\ra X,$ we can construct a canonical n-orientation on $\B_P$.

In the situation of\/ {\bf(a)\rm,} the almost complex structure $J$ induces a $\Spinc$-structure ${\mathfrak s}_J$ on $X,$ and the n-orientation on $\B_P$ from {\bf(a)} agrees with the n-orientation constructed above from the $\Spinc$-structure~${\mathfrak s}_J$.
\item[{\bf(c)}] If\/ $G$ is also simply-connected, or if\/ $G=\U(m),$ then the n-orientation on $\B_P$ in {\bf(b)} is independent of the choice of\/ $\Spinc$-structure~$\mathfrak s$.
\end{itemize}
\end{thm}

\begin{proof} For (a), given $J$ we can put complex structures on $E_0,E_1$, such that the symbol of $E_\bu$ agrees with that of $\bar\partial+\bar\partial^*:\Ga^\iy(\La^{0,0}\op\La^{0,2})\ra\Ga^\iy(\La^{0,1})$, which is complex linear. So (a) holds by Theorem \ref{or2thm1}.

For (b), choose $G,{\mathfrak s},P\ra X$ as in the theorem. Let $\nabla_P$ be any connection on $P$. Proposition \ref{or3prop1}(a) with $d=2$ gives a 2-skeleton $Y\subset X$, which is a CW-complex of dimension 2, and a trivialization $\Phi:P\vert_{X\sm Y}\ra (X\sm Y)\t G$. As in Step 3 of \S\ref{or33}, choose a small open neighbourhood $U$ of $Y$ in $X$ such that $U$ retracts onto $Y$, an open $V\subset X$ with $\ov V\subseteq X\sm Y$ and $U\cup V=X$, and a connection $\hat\nabla_P$ on $P\ra X$ which is trivial over $V\subset X\sm Y$, using the chosen trivialization $\Phi$ of $P\vert_{X\sm Y}$. 

Recall from \cite{Morg,Nico} that a $\Spinc$-structure $\mathfrak s$ on $(X,g)$ consists of rank 2 Hermitian vector bundles $S_\pm^{\mathfrak s}\ra X$ and a Clifford multiplication map $T^*X\ra \Hom_\C(S_+^{\mathfrak s},S_-^{\mathfrak s})$. They are related to almost complex structures in the following way. The projective space bundle ${\mathbb P}_\C(S_+^{\mathfrak s})\ra X$ is naturally isomorphic to the bundle of oriented Hermitian almost complex structures $J$ on $(X,g)$, as used in (a). Hence, oriented Hermitian almost complex structures $J$ on $X$ correspond to complex line subbundles $L_J\subset S_+^{\mathfrak s}$. Such a $J$ determines a $\Spinc$-structure ${\mathfrak s}_J$, unique up to isomorphism, for which $L_J$ is the trivial line bundle. Thus, if $S_+^{\mathfrak s}$ has a nonvanishing section $\si\in\Ga^\iy(S_+^{\mathfrak s})$, this determines a unique oriented Hermitian almost complex structure $J$ on $(X,g)$, and~${\mathfrak s}\cong{\mathfrak s}_J$.

For general $X$ there need not exist nonvanishing sections $\si\in\Ga^\iy(S_+^{\mathfrak s})$, and if they exist they need not be unique up to isotopy. But $S_+^{\mathfrak s}\vert_Y$ has nonvanishing sections which are unique up to isotopy over the 2-skeleton $Y$, as $S_+^{\mathfrak s}$ has real rank $4>2+1$. Since $U$ retracts onto $Y$, it follows that $S_+^{\mathfrak s}\vert_U$ has nonvanishing sections $\si\in\Ga^\iy(S_+^{\mathfrak s}\vert_U)$ which are unique up to isotopy. These yield oriented Hermitian almost complex structures $J_\si$ on $(U,g\vert_U)$, which are unique up to isotopy, and determine ${\mathfrak s}\vert_U$ up to isomorphism.

Apply Theorem \ref{or3thm2}(a) with $X,U,V,E_\bu,P,\hat\nabla_P,\Phi\vert_V$ in place of $X^+,\ab U^+,\ab V^+,\ab E^+_\bu,\ab P^+,\ab \nabla_{P^+},\ab \tau^+$, and using the complex structure on $E_\bu\vert_U$ induced by $J_\si$ for some nonvanishing $\si\in\Ga^\iy(S_+^{\mathfrak s}\vert_U)$. This gives an isomorphism 
\e
\smash{\check O_P^{E_\bu}\big\vert_{[\hat\nabla_P]}\cong \Z_2.}
\label{or4eq4}
\e
We then determine the isomorphism $\check O_P^{E_\bu}\vert_{[\nabla_P]}\cong\Z_2$ by choosing a smooth path from $\hat\nabla_P$ to $\nabla_P$ in the (contractible) space of connections on $P\ra X$, and deforming the n-orientation continuously along this path. This constructs an n-orientation at any point $[\nabla_P]$ in~$\B_P$.

We will show this n-orientation at $[\nabla_P]$ is independent of choices in its construction, following Step 4 of \S\ref{or33}. Suppose that $Y_0,\Phi_0,U_0,V_0,\hat\nabla_{P,0},\si_0$ and $Y_1,\Phi_1,U_1,V_1,\hat\nabla_{P,1},\si_1$ are alternative choices above. Then Proposition \ref{or3prop1}(b) with $d=2$ gives a 3-skeleton $Z\subset X\t[0,1]$ interpolating between $Y_0\t\{0\}$ and $Y_1\t\{1\}$, and a trivialization $\Psi:\pi_X^*(P)\vert_{(X\t[0,1])\sm Z}\ra ((X\t[0,1])\sm Z)\t G$ interpolating between $\Phi_0$ and $\Phi_1$. Choose a small open neighbourhood $W$ of $Z$ in $X\t[0,1]$ which interpolates between $U_0$ and $U_1$ and retracts onto~$Z$. 

As $Z$ is a CW-complex of dimension 3, there exist nonvanishing sections of the rank 4 bundle $\pi_X^*(S_+^{\mathfrak s})\vert_Z$, and we can choose them to interpolate between given nonvanishing sections on $Y_0$ and $Y_1$. Since $W$ retracts onto $Z$, it follows that there exist nonvanishing sections $\tau$ of $\pi_X^*(S_+^{\mathfrak s})\vert_W$, and we can choose $\tau$ with $\tau\vert_{Y_i\t\{i\}}=\si_i$ for~$i=0,1$.

We can now choose data $\Phi_t,U_t,V_t,\hat\nabla_{P,t},\si_t$ as above, depending smoothly on $t\in[0,1]$, and interpolating between $\Phi_0,U_0,V_0,\hat\nabla_{P,0},\si_0$ and $\Phi_1,U_1,V_1,\hat\nabla_{P,1},\si_1$, where $\Phi_t(x)=\Psi(x,t)$ for $(x,t)\in(X\t[0,1])\sm Z$, and $U_t=\bigl\{x\in X:(x,t)\in W\bigr\}$, and $\si_t(x)=\tau(x,t)$ for $x\in U_t$. So as in \eq{or4eq4} we get isomorphisms for $t\in[0,1]$
\e
\check O_P^{E_\bu}\big\vert_{[\hat\nabla_{P,t}]}\cong \Z_2,
\label{or4eq5}
\e
which depend continuously on $t\in[0,1]$ by strong functoriality in Theorem \ref{or3thm2}(a). We choose a smooth path from $\hat\nabla_{P,t}$ to $\nabla_P$ in the (contractible) space of connections on $P\ra X$, depending smoothly on $t\in[0,1]$, and interpolating between the previous choices when $t=0$ and~$t=1$. 

Deforming the trivialization of $\check O_P^{E_\bu}\vert_{[\hat\nabla_{P,t}]}$ along this path gives a continuous family of trivializations of $\check O_{P}^{E_\bu}\vert_{[\nabla_{P}]}$ for $t\in[0,1]$ which interpolate between the two previous choices at $t=0$ and $t=1$. Hence the trivialization of $\check O_{P}^{E_\bu}\vert_{[\nabla_{P}]}$ defined in the first part of the proof is independent of choices in its construction, and is well defined. These trivializations clearly depend continuously on $[\nabla_P]$ in $\B_P$, and so define a canonical n-orientation on $\B_P$, for all principal $G$-bundles $P\ra X$. This completes the first part of~(b).

For the second part, let $J$ be an almost complex structure on $X$, inducing a $\Spinc$-structure ${\mathfrak s}_J$. Then in the definition of the orientation on $\B_P$ above, we can take $J_\si=J\vert_U$. So by Theorem \ref{or3thm2}(b), the isomorphism \eq{or4eq4} agrees with natural n-orientation defined by Theorem \ref{or2thm1} at $[\hat\nabla_P]$, so deforming along the path from $[\hat\nabla_P]$ to $[\nabla_P]$, part (b) follows.

For (c), if $G$ is simply-connected, then in the proof above we can apply Proposition \ref{or3prop1}(a) with $d=4$ instead of $d=2$. So $Y\subset X$ is a 0-skeleton, and the almost complex structure $J_\si$ on $U$ is unique up to isotopy, and independent of $\mathfrak s$. Thus the orientation on $\B_P$ is independent of $\mathfrak s$. The case $G=\U(m)$ follows from Example \ref{or2ex5} and $G=\SU(m+1)$, which is simply-connected.
\end{proof}

As in Definition \ref{or1def2}, we can convert n-orientations on $\B_P$ to orientations on $\B_P$ by choosing orientations on $\det D$ and $\g$. Since orientations on $\B_P$ induce orientations on $\ovB_P$, which pull back to orientations of $\M_P^{\rm asd}$, we deduce:

\begin{thm}
\label{or4thm5}
Let\/ $(X,g)$ be a compact, oriented Riemannian\/ $4$-manifold.
\begin{itemize}
\setlength{\itemsep}{0pt}
\setlength{\parsep}{0pt}
\item[{\bf(a)}] Let\/ $G$ be a connected Lie group, and choose an orientation on\/ $\det D$ (equivalently, an orientation on $H^0(X)\op H^1(X)\op H^2_+(X)$) and on\/ $\g,$ and a\/ $\Spinc$-structure $\mathfrak s$ on $X$. Then for all principal\/ $G$-bundles $P\ra X,$ we can construct a canonical orientation on $\M_P^{\rm asd}$.
\item[{\bf(b)}] If\/ $G$ is also simply-connected, or if\/ $G=\U(m),$ then the orientation on $\M_P^{\rm asd}$ in {\bf(a)} is independent of the choice of\/ $\Spinc$-structure~$\mathfrak s$.
\end{itemize}	
\end{thm}

Here is how this relates to results in the literature: part (b) is proved for $G=\SU(m)$ and $X$ simply-connected by Donaldson \cite[II.4]{Dona1} using the methods of Lemma \ref{or2lem1} and \S\ref{or24}, for $G=\U(m),\SU(m)$ and $X$ arbitrary by Donaldson \cite[\S 3(d)]{Dona2} using the method of \S\ref{or3}, and for $X$ simply-connected and $G$ a simply-connected, simple Lie group by Donaldson and Kronheimer \cite[\S 5.4]{DoKr}. So far as the authors know, part (a) is new, both the orientability of $\B_P,\M_P^{\rm asd}$, and the use of $\Spinc$-structures in constructing canonical orientations.

\subsubsection{Seiberg--Witten theory on 4-manifolds}
\label{or424}

We will explain the Seiberg--Witten $\U(m)$-monopole equations, following Zentner \cite{Zent}. Most of the literature on Seiberg--Witten theory (see e.g.\ Morgan \cite{Morg} and Nicolaescu \cite{Nico}) is concerned with the case $m=1$ of these. Let $(X,g)$ be a compact, oriented Riemannian 4-manifold. Fix a $\Spinc$-{\it structure\/} $\mathfrak s$ on $(X,g)$, consisting of rank 2 Hermitian vector bundles $S_\pm^{\mathfrak s}\ra X$ with a given connection $\nabla_{S_+^{\mathfrak s}}$ on $S_+^{\mathfrak s}$, and a Clifford multiplication map~$T^*X\ra \Hom_\C(S_+^{\mathfrak s},S_-^{\mathfrak s})$. 

Let $P\ra X$ be a principal $\U(m)$-bundle, and $E=(P\t\C^m)/\U(m)$ the associated rank $m$ complex vector bundle $E\ra X$. A {\it Seiberg--Witten $\U(m)$-monopole\/} on $P$ is a pair $(\nabla_P,\Psi)$ of a connection $\nabla_P$ on $P\ra X$, and a section $\Psi\in \Ga^\iy(S_+^{\mathfrak s}\ot_\C E)$, satisfying the {\it Seiberg--Witten equations\/}
\e 
\slashed{D}{}^{\nabla_P}\Psi=0 , \qquad \smash{F^{\nabla_P}_+} = q(\Psi), 
\label{or4eq6}
\e 
where $\slashed{D}{}^{\nabla_P}:\Ga^\iy(S_+^{\mathfrak s}\ot_\C E)\ra \Ga^\iy(S_-^{\mathfrak s}\ot_\C E)$ is a twisted Dirac operator defined using $\nabla_{S_+^{\mathfrak s}}$ and $\nabla_P$, and $q:S_+^{\mathfrak s}\ot_\C E\ra \Ad(P)\ot\La^2_+T^*X$ is a certain real quadratic bundle morphism.

The gauge group $\G_P=\Aut(P)$ acts on the family of Seiberg--Witten mono\-poles $(\nabla_P,\Psi)$. We call $(\nabla_P,\Psi)$ {\it irreducible\/} if its stabilizer group in $\G_P$ is trivial. Define $\M_{P,\mathfrak s}^{\rm SW}$ to be the moduli space of gauge equivalence classes $[\nabla_P,\Psi]$ of irreducible $\U(m)$-monopoles $(\nabla_P,\Psi)$. If $g,\nabla_{S_+^{\mathfrak s}}$ are generic then $\M_{P,\mathfrak s}^{\rm SW}$ is a smooth manifold. There is a forgetful map $\M_{P,\mathfrak s}^{\rm SW}\ra\B_P$ taking $[\nabla_P,\Psi]\mapsto[\nabla_P]$, for $\B_P$ as in Definition \ref{or1def1}. The moduli spaces $\M_{P,\mathfrak s}^{\rm SW}$ for $m=1$ (with orientations as below) are used to define the {\it Seiberg--Witten invariants\/} of $X$, \cite{Morg,Nico}, which are closely related to Donaldson invariants in \S\ref{or423}, but are technically less demanding.

To understand orientations for $\M_{P,\mathfrak s}^{\rm SW}$, observe that the linearization of the Seiberg--Witten equations \eq{or4eq6} is isotopic to the direct sum of the Atiyah--Hitchin--Singer complex \eq{or4eq1}, and $\slashed{D}{}^{\nabla_P}:\Ga^\iy(S_+^{\mathfrak s}\ot_\C E)\ra \Ga^\iy(S_-^{\mathfrak s}\ot_\C E)$. Hence the orientation bundle of $\M_{P,\mathfrak s}^{\rm SW}$ is the tensor product of orientation bundles from these two factors. The first is the pullback of the orientation bundle $O_P^{E_\bu}\ra\B_P$ for anti-self-dual instantons in \S\ref{or423} under the forgetful map $\M_{P,\mathfrak s}^{\rm SW}\ra\B_P$. The second is trivial by Theorem \ref{or2thm1}, as the symbol of $\slashed{D}{}^{\nabla_P}$ is complex linear. Thus, by Theorem \ref{or4thm4}(c), if we choose an orientation on $\det D$ (equivalently, an orientation on $H^0(X)\op H^1(X)\op H^2_+(X)$), we have canonical orientations on all moduli spaces $\M_{P,\mathfrak s}^{\rm SW}$. When $m=1$ this is proved by Morgan~\cite[Cor.~6.6.3]{Morg}.

\subsubsection{The Vafa--Witten equations on 4-manifolds}
\label{or425}

Let $(X,g)$ be a compact, oriented Riemannian 4-manifold, $G$ a Lie group, and $P\ra X$ a principal $G$-bundle. The {\it Vafa--Witten equations\/} concern triples $(\nabla_P,\al,\be)$ of a connection $\nabla_P$ on $P$ and sections $\al\in\Ga^\iy(\Ad(P)))$ and $\be\in\Ga^\iy(\Ad(P)\ot\La^2_+T^*X)$ satisfying
\begin{gather*}
\d_{\nabla_P} \al + \d_{\nabla_P}^{*} \be = 0 , \qquad
F^{\nabla_P}_++\ts\frac{1}{2}[\be,\al]+\frac{1}{8}[\be.\be]=0, 
\end{gather*}
where $[\be.\be]\in\Ga^\iy(\Ad(P)\ot\La^2_+T^*X)$. They were introduced by Vafa and Witten \cite{VaWi} in the study of the $S$-duality conjecture in $\mathcal{N}=4$ super Yang--Mills theory. See Mares \cite[\S A.1.6]{Mare} or Tanaka \cite[\S 2]{Tana3} for more details.  

The gauge group $\G_P=\Aut(P)$ acts on the family of Vafa--Witten solutions $(\nabla_P,\al,\be)$. We call $(\nabla_P,\al,\be)$ {\it irreducible\/} if its stabilizer group in $\G_P$ is $Z(G)$. Define $\M_P^{\rm VW}$ to be the moduli space of gauge equivalence classes $[\nabla_P,\al,\be]$ of irreducible Vafa--Witten solutions $(\nabla_P,\al,\be)$. It is a derived manifold. There is a forgetful map $\M_P^{\rm VW}\ra\ovB_P$ taking $[\nabla_P,\al,\be]\mapsto[\nabla_P]$, for $\ovB_P$ as in Definition~\ref{or1def1}. 

The deformations of $[\nabla_P,\al,\be]$ in $\M_P^{\rm VW}$ are controlled by the elliptic complex 
\begin{equation*} 
\xymatrix@!0@C=34.5pt@R=40pt{ 0 \ar[rr] &&
\Ga^{\iy} \bigl( \Ad(P) \!\ot\! \La^0T^*X \bigr) \ar[rrrrr]^(0.36){(\d_{\nabla_P}\, 0\, 0)} &&&&& \Ga^{\iy} \bigl( \Ad(P) \!\ot\! (\La^1T^*X\!\op\!\La^0T^*X\!\op\!\La^2_+T^*X) \bigr) \\
\ar[rrrrrr]^(0.28){\begin{pmatrix} 0 & \d_{\nabla_P} & \d_{\nabla_P}^* \\ \d_{\nabla_P} & 0 & 0 \end{pmatrix}} &&&&&&
\Ga^{\iy} ( \Ad(P) \ot (\La^1T^*X\op\La^2_+T^*X)  ) \ar[rrr] &&& 0, } 
\end{equation*}
modulo degree 0 operators. As in \S\ref{or421}--\S\ref{or424}, the orientation bundle of $\M_P^{\rm VW}$ is the pullback of the orientation bundle $\bar O_P^{E_\bu}\ra\ovB_P$ under the forgetful map $\M_P^{\rm VW}\ra\ovB_P$, where $E_\bu$ is the elliptic operator
\begin{equation*} 
\begin{pmatrix} \d & \d^* & 0 \\
0 & 0 & \d^* \\ 0 & 0 & \d_+ \end{pmatrix}\begin{aligned}[h]: \Ga^{\iy} \bigl( \La^0T^*X\op\La^2_+T^*X\op\La^1T^*X\bigr)\longra &\\
\Ga^{\iy} \bigl(\La^1T^*X\op\La^0T^*X\op\La^2_+T^*X\bigr)&. \end{aligned}
\end{equation*}
Then $E_\bu=\ti E_\bu\op\ti E_\bu^*$, for $\ti E_\bu$ as in \eq{or4eq2}. Therefore \S\ref{or225} shows that $O_P^{E_\bu}$ and hence $\bar O_P^{E_\bu}$ are canonically trivial. This proves:

\begin{thm} The Vafa--Witten moduli spaces $\M_P^{\rm VW}$ have canonical orientations for all\/ $G$ and\/~$P\ra X$.
\label{or4thm6}	
\end{thm}

The authors believe Theorem \ref{or4thm6} is new.

\subsubsection{The Kapustin--Witten equations on 4-manifolds}
\label{or426}

Let $(X,g)$ be a compact, oriented Riemannian 4-manifold, $G$ a Lie group, and $P\ra X$ a principal $G$-bundle. The {\it Kapustin--Witten equations\/} concern pairs $(\nabla_P,\phi)$ of a connection $\nabla_P$ on $P$ and a section $\phi\in\Ga^\iy(\Ad(P)\ot T^*X)$ satisfying
\begin{gather*}
\d_{\nabla_P}^* \phi = 0 \quad \text{in }\Ga^\iy(\Ad(P)), \qquad  \d_{\nabla_P}^{-} \phi = 0 \quad\text{in }\Ga^\iy(\Ad(P)\ot\La^2_-T^*X),\\
\text{and}\qquad F^{\nabla_P}_{+} + [\phi,\phi]_{+} = 0 \quad\text{in }\Ga^\iy(\Ad(P)\ot\La^2_+T^*X). 
\end{gather*}
They were introduced by Kapustin and Witten in \cite[\S 3.3]{KaWi}, coming from topologically twisted $\mathcal{N}=4$ super Yang--Mills theory, and are studied in Gagliardo and Uhlenbeck \cite{GaUh}, Taubes \cite{Taub2}, and Tanaka~\cite{Tana4}.

The gauge group $\G_P=\Aut(P)$ acts on the family of Kapustin--Witten solutions $(\nabla_P,\phi)$. We call $(\nabla_P,\phi)$ {\it irreducible\/} if its stabilizer group in $\G_P$ is $Z(G)$. Define $\M_P^{\rm KW}$ to be the moduli space of gauge equivalence classes $[\nabla_P,\phi]$ of irreducible Kapustin--Witten solutions $(\nabla_P,\phi)$, as a derived manifold. There is a forgetful map $\M_P^{\rm KW}\ra\ovB_P$ taking $[\nabla_P,\phi]\mapsto[\nabla_P]$, for $\ovB_P$ as in Definition~\ref{or1def1}. 

The deformations of $[\nabla_P,\phi]$ in $\M_P^{\rm KW}$ are controlled by the elliptic complex 
\begin{equation*} 
\xymatrix@!0@C=34.5pt@R=50pt{ 0 \ar[rr] &&
\Ga^{\iy} \bigl( \Ad(P) \!\ot\! \La^0T^*X \bigr) \ar[rrrrr]^(0.42){(\d_{\nabla_P}\, 0)} &&&&& \Ga^{\iy} \bigl( \Ad(P) \!\ot\! (\La^1T^*X\!\op\!\La^1T^*X) \bigr) \\
\ar[rrrrr]^(0.24){\begin{pmatrix} 0 & \d^*_{\nabla_P}  \\ \d^+_{\nabla_P} & 0 \\ 0 & \d^-_{\nabla_P} \end{pmatrix}} &&&&&
\Ga^{\iy} ( \Ad(P) \ot (\La^0T^*X\op\La^2_+T^*X\op\La^2_-T^*X) ) \ar[rrrr] &&&& 0, } 
\end{equation*}
modulo degree 0 operators. As in \S\ref{or421}--\S\ref{or425}, the orientation bundle of $\M_P^{\rm KW}$ is the pullback of the orientation bundle $\bar O_P^{E_\bu}\ra\ovB_P$ under the forgetful map $\M_P^{\rm KW}\ra\ovB_P$, where $E_\bu$ is the elliptic operator
\begin{equation*} 
\begin{pmatrix} 0 & \d_+ & 0 & \d^* \\
\d^* & 0 & \d_- & 0 
\end{pmatrix}
\begin{aligned}[h]:\, &\Ga^{\iy} \bigl( \La^0T^*X\!\op\!\La^2_+T^*X\!\op\!\La^2_-T^*X\op\La^0T^*X\bigr)\\
&\qquad\longra \Ga^{\iy} \bigl(\La^1T^*X\op\La^1T^*X\bigr). \end{aligned}
\end{equation*}
This is a direct sum $E_\bu=E^+_\bu\op E^-_\bu$, where $E^+_\bu$ is as in \eq{or4eq2}, and $E^-_\bu$ is $E^+_\bu$ for the opposite orientation on $X$. Thus $\bar O_P^{E_\bu}\cong\bar O_P^{E^+_\bu}\ot_{\Z_2}\bar O_P^{E^-_\bu}$. As Theorem \ref{or4thm4} describes orientations for $O_P^{E^\pm_\bu}$, and hence $\bar O_P^{E^\pm_\bu}$, we deduce:

\begin{thm}
\label{or4thm7}
Let\/ $(X,g)$ be a compact, oriented Riemannian\/ $4$-manifold.
\begin{itemize}
\setlength{\itemsep}{0pt}
\setlength{\parsep}{0pt}
\item[{\bf(a)}] Let\/ $G$ be a connected Lie group, and choose an orientation on\/ $\det D$ (equivalently, an orientation on $H^2(X)$) and on\/ $\g,$ and a\/ $\Spinc$-structure $\mathfrak s$ on $X$. Then for all principal\/ $G$-bundles $P\ra X,$ we can construct a canonical orientation on $\M_P^{\rm KW},$ as a derived manifold.
\item[{\bf(b)}] If\/ $G$ is also simply-connected, or if\/ $G=\U(m),$ then the orientation on $\M_P^{\rm KW}$ in {\bf(a)} is independent of the choice of\/ $\Spinc$-structure~$\mathfrak s$.
\end{itemize}
\end{thm}

The authors believe Theorem \ref{or4thm7} is new.

\subsubsection{The Haydys--Witten equations on 5-manifolds}
\label{or427}

The {\it Haydys--Witten equations\/} are a 5-dimensional gauge theory introduced independently by Haydys \cite[\S 3]{Hayd} and Witten \cite[\S 5.2.6]{Witt}. Both Haydys and Witten work on oriented 5-manifolds. We will give a different presentation of the Haydys--Witten equations to \cite{Hayd,Witt}, which also works for non-oriented 5-manifolds, and is equivalent to \cite{Hayd,Witt} in the oriented case.

Let $(X,g)$ be a compact Riemannian 5-manifold, and $\al\in\Ga^\iy(\La^4T^*X)$ be a nonvanishing 4-form with $\md{\al}_g\equiv 1$. Then there are orthogonal splittings
\e
\begin{gathered}
\La^1T^*X=\La^1_\t T^*X\op\La^1_0T^*X,\;\>
\La^2T^*X=\La^2_+T^*X\op\La^2_-T^*X\op\La^2_0T^*X, \\ 
\La^3T^*X=\La^3_+T^*X\op\La^3_-T^*X\op\La^3_0T^*X,	
\end{gathered}
\label{or4eq7}
\e
where $\La^1_\t T^*X$ has rank 1, $\La^j_\pm T^*X$ have rank 3, and $\La^j_0T^*X$ have rank 4. They may be described explicitly as follows: if $x\in X$ and $(e_1,\ldots,e_5)$ is an orthonormal basis of $T_x^*X$ with $\al\vert_x=e_1\w e_2\w e_3\w e_4$, then
\e
\begin{aligned}
\La^1_\t T^*_xX&=\an{e_5}_\R, \;\> \La^1_0T^*_xX=\an{e_1,e_2,e_3,e_4}_\R, \;\>
\La^2_0T^*_xX=\an{e_{15},e_{25},e_{35},e_{45}}_\R, \\
\La^2_+T^*_xX&=\an{e_{12}+e_{34},e_{13}+e_{42},e_{14}+e_{23}}_\R, \\
\La^2_-T^*_xX&=\an{e_{12}-e_{34},e_{13}-e_{42},e_{14}-e_{23}}_\R,  \\
\La^3_+T^*_xX&=\an{e_{125}+e_{345},e_{135}+e_{425},e_{145}+e_{235}}_\R, \\
\La^3_-T^*_xX&=\an{e_{125}+e_{345},e_{135}+e_{425},e_{145}+e_{235}}_\R, \\
\La^3_0T^*_xX&=\an{e_{234},e_{341},e_{412},e_{123}}_\R,
\end{aligned}\!\!\!\!\!\!\!\!{}
\label{or4eq8}
\e
where $e_{ij\cdots l}$ means $e_i\w e_j\w\cdots \w e_l$.

Let $G$ be a Lie group and $P\ra X$ be a principal $G$-bundle. The {\it Haydys--Witten equations\/} concern pairs $(\nabla_P,\psi)$ of a connection $\nabla_P$ of $P$ and a section $\psi\in\Ga^\iy(\Ad(P)\ot \La^3_+T^*X)$ satisfying
\begin{align*}
F^{\nabla_P}_{+} + (\d^*_{\nabla_P}\psi)_++q(\psi) &= 0 \quad\text{in }\Ga^\iy(\Ad(P)\ot\La^2_+T^*X),\\
F^{\nabla_P}_0 + (\d^*_{\nabla_P}\psi)_0 &= 0  \quad \text{in }\Ga^\iy(\Ad(P)\ot \La^2_0T^*X),
\end{align*}
for $q:\Ad(P)\ot \La^3_+T^*X\ra\Ad(P)\ot \La^2_+T^*X$ a certain quadratic bundle map. They were introduced independently by Haydys \cite[\S 3]{Hayd} and Witten~\cite[\S 5.2.6]{Witt}. 

The gauge group $\G_P=\Aut(P)$ acts on the family of Haydys--Witten solutions $(\nabla_P,\psi)$. We call $(\nabla_P,\psi)$ {\it irreducible\/} if its stabilizer group in $\G_P$ is $Z(G)$. Define $\M_P^{\rm HW}$ to be the moduli space of gauge equivalence classes $[\nabla_P,\psi]$ of irreducible Haydys--Witten solutions $(\nabla_P,\phi)$. It is a derived manifold of virtual dimension 0. There is a forgetful map $\M_P^{\rm HW}\ra\ovB_P$ taking $[\nabla_P,\psi]\mapsto[\nabla_P]$, for $\ovB_P$ as in Definition~\ref{or1def1}. 

The deformations of $[\nabla_P,\psi]$ in $\M_P^{\rm HW}$ are controlled by the elliptic complex 
\begin{equation*} 
\xymatrix@!0@C=34.5pt@R=38pt{ 0 \ar[rr] &&
\Ga^{\iy} \bigl( \Ad(P) \!\ot\! \La^0T^*X \bigr) \ar[rrrrr]^(0.42){(\d_{\nabla_P}\, 0)} &&&&& \Ga^{\iy} \bigl( \Ad(P) \!\ot\! (\La^1T^*X\!\op\!\La^3_+T^*X) \bigr) \\
\ar[rrrrr]^(0.27){\begin{pmatrix} (\d_{\nabla_P})_+ & (\d^*_{\nabla_P})_+  \\ (\d_{\nabla_P})_0 & (\d^*_{\nabla_P})_0 \end{pmatrix}} &&&&&
\Ga^{\iy} ( \Ad(P) \ot (\La^2_+T^*X\op\La^2_0T^*X)  ) \ar[rrr] &&& 0, } 
\end{equation*}
modulo degree 0 operators. As in \S\ref{or421}--\S\ref{or426}, the orientation bundle of $\M_P^{\rm HW}$ is the pullback of the orientation bundle $\bar O_P^{E_\bu}\ra\ovB_P$ under the forgetful map $\M_P^{\rm HW}\ra\ovB_P$, where $E_\bu$ is the elliptic operator
\e
\begin{gathered} 
D\!=\!\begin{pmatrix} \d & \d^*_0 & \d^*_+  \\ 0 & \d_+ & \d_+ \end{pmatrix}\begin{aligned}[h]:\, &\Ga^{\iy} \bigl( \La^0T^*X\!\op\!\La^2_0T^*X\!\op\!\La^2_+T^*X \bigr)\\
&\quad\longra \Ga^{\iy} \bigl(\La^1T^*X\op\La^3_+T^*X\bigr). 
\end{aligned}
\end{gathered}
\label{or4eq9}
\e

\begin{thm}
\label{or4thm8}
Let\/ $(X,g)$ be a compact Riemannian $5$-manifold, and\/ $\al\in\Ga^\iy(\La^4T^*X)$ a unit length\/ $4$-form.
\begin{itemize}
\setlength{\itemsep}{0pt}
\setlength{\parsep}{0pt}
\item[{\bf(a)}] Suppose there exists a Hermitian complex structure $J$ on the fibres of\/ $\La^1_0T^*X$ in \eq{or4eq7} compatible with the natural orientation on $\La^1_0T^*X$. (This is equivalent to choosing an \begin{bfseries}almost CR structure\end{bfseries} on $X$.) Then using $J$ we can define a complex structure on $E_\bu,$ as in Theorem\/~{\rm\ref{or2thm1}}.

Thus for any Lie group $G$ and principal\/ $G$-bundle $P\ra X,$ Theorem\/ {\rm\ref{or2thm1}} defines a canonical n-orientation on\/~$\B_P$.
\item[{\bf(b)}] Now let\/ $G$ be a connected, simply-connected Lie group. Then for all principal\/ $G$-bundles $P\ra X,$ we can construct a canonical n-orientation on $\B_P$. It agrees with that defined in\/ {\bf(a)} when part\/ {\bf(a)} applies.
\end{itemize}	
\end{thm}

\begin{proof} For (a), at a point $x\in X$, choose an orthonormal basis $(e_1,\ldots,e_5)$ for $T_x^*X$ with $\al\vert_x=e_1\w e_2\w e_3\w e_4$ and $J(e_1)=e_2$, $J(e_3)=e_4$ in the basis \eq{or4eq8} for $\La^1_0T_x^*X$. Define complex vector space structures on $E_0\vert_x,E_1\vert_x$ in \eq{or4eq9} by
\begin{align*}
i\cdot 1&=e_{12}+e_{34}, \; i\cdot e_{15}=e_{25}, \; i\cdot e_{35}=e_{45}, \;
i\cdot (e_{13}+e_{42}) = e_{14}+e_{23}, \\
i\cdot e_5&=e_{125}+e_{345}, \; i\cdot e_1=e_2, \; i\cdot e_3=e_4, \;  i\cdot (e_{135}+e_{425})=e_{145}+e_{235}.
\end{align*}
It is easy to check that these are independent of the choice of $(e_1\ldots,e_5)$, so over all $x\in X$ they extend to complex vector bundle structures on $E_0,E_1$, and the symbol of $D$ is complex linear. Part (a) follows.

For (b), let $G$ be a connected, simply-connected Lie group, and $P\ra X$ a principal $G$-bundle. By Proposition \ref{or3prop1}(a) with $d=4$ we can choose a 1-skeleton $Y\subset X$, a CW complex of dimension 1, and a trivialization $\Phi:P\vert_{X\sm Y}\ra (X\sm Y)\t G$. As in Step 3 of \S\ref{or33}, choose a small open neighbourhood $U$ of $Y$ in $X$ such that $U$ retracts onto $Y$, and an open $V\subset X$ with $\ov V\subseteq X\sm Y$ and $U\cup V=X$, and a connection $\hat\nabla_P$ on $P\ra X$ which is trivial over $V\subset X\sm Y$, using the chosen trivialization $\Phi$ of~$P\vert_{X\sm Y}$. 

For general $X$ there need not exist oriented Hermitian complex structures $J$ on $\La^1_0T^*X$, and if they exist they need not be unique up to isotopy. But such complex structures exist on $\La^1_0T^*X\vert_Y$ over the 1-skeleton $Y$, and are unique up to isotopy. As $U$ retracts onto $Y$, it follows that $\La^1_0T^*X\vert_U$ admits oriented Hermitian complex structures $J$, uniquely up to isotopy. Thus as in (a) we can define a complex structure on $E_\bu\vert_U$, unique up to isotopy.

The rest of the proof of (b) closely follows the proof of Theorem \ref{or4thm4}(b) from the paragraph containing \eq{or4eq4}, except that in the paragraph before that containing \eq{or4eq5}, as $Z$ is now a CW-complex of dimension 2, there exist oriented Hermitian complex structures on $\pi_X^*(\La^1_0T^*X)\vert_Z$, and we can choose them to interpolate between given complex structures over $Y_0$ and~$Y_1$. 
\end{proof}

To pass from an n-orientation of $\B_P$ to an orientation of $\B_P$, by \eq{or1eq4} and \eq{or1eq7} we need to choose an orientation for $\det D$, noting that as $\ind D=0$ on any 5-manifold $X$ we do not need an orientation on $\g$. Since orientations on $\B_P$ induce orientations of $\ovB_P$, which pull back to orientations of $\M_P^{\rm HW}$, we deduce:

\begin{thm} Let\/ $(X,g)$ be a compact Riemannian $5$-manifold, and\/ $\al\in\Ga^\iy(\La^4T^*X)$ a unit length\/ $4$-form. Suppose\/ $G$ is a connected, simply-connected Lie group, and choose an orientation on\/ $\det D$. Then for all principal\/ $G$-bundles $P\ra X,$ we can construct a canonical orientation on\/~$\M_P^{\rm HW}$.
\label{or4thm9}	
\end{thm}

The authors believe Theorems \ref{or4thm8} and \ref{or4thm9} are new.

\subsubsection{Donaldson--Thomas instantons on symplectic 6-manifolds}
\label{or428}

Let $(X,\om)$ be a symplectic 6-manifold, $J$ a compatible almost complex structure on $X$, and $g=\om(-,J-)$ the associated Hermitian metric. The second author \cite{Tana2} introduced `Donaldson--Thomas instantons' \cite{Tana2}, a gauge theory on $X$, in the hope of defining `analytic Donaldson--Thomas invariants'. 

Let $G$ be a Lie group and $P\ra X$ a principal $G$-bundle. A {\it Donaldson--Thomas instanton\/} on $P$ is a pair $(\nabla_P,\xi)$ of a connection $\nabla_P$ on $P$ and a section $\xi\in\Ga^\iy(\Ad(P)\ot_\R\La^{0,3}T^*X)$ satisfying
\begin{equation*} 
F^{\nabla_P}_{0,2} = \bar{\partial}_{ \nabla_P}^{*} \xi, \qquad F^{ \nabla_P}_{1,1} \w \om = 0. 
\end{equation*} 
The gauge group $\G_P=\Aut(P)$ acts on the family of Donaldson--Thomas instantons $(\nabla_P,\xi)$. We call $(\nabla_P,\xi)$ {\it irreducible\/} if its stabilizer group in $\G_P$ is $Z(G)$. Write $\M_P^{\rm DT}$ for the moduli space of gauge equivalence classes $[\nabla_P,\xi]$ of irreducible Donaldson--Thomas instantons $(\nabla_P,\xi)$. It is a derived manifold. There is a forgetful map $\M_P^{\rm DT}\ra\ovB_P$ taking~$[\nabla_P,\xi]\mapsto[\nabla_P]$. 

The deformations of $[\nabla_P,\xi]$ in $\M_P^{\rm DT}$ are controlled by the elliptic complex 
\begin{equation*} 
\xymatrix@!0@C=34.5pt@R=38pt{ 0 \ar[rr] &&
\Ga^{\iy} \bigl( \Ad(P) \!\ot\! \La^0T^*X \bigr) \ar[rrrrr]^(0.42){(\d_{\nabla_P}\, 0)} &&&&& \Ga^{\iy} \bigl( \Ad(P) \!\ot\! (\La^1T^*X\!\op\!\La^{0,3}T^*X) \bigr) \\
\ar[rrrrr]^(0.27){\begin{pmatrix} \bar\partial^{0,2}_{\nabla_P} & \bar\partial_{\nabla_P}^*  \\ \d^{\an{\om}}_{\nabla_P} & 0 \end{pmatrix}} &&&&&
\Ga^{\iy} ( \Ad(P) \ot (\La^{0,2}T^*X\op\an{\om})  ) \ar[rrr] &&& 0, } 
\end{equation*}
modulo degree 0 operators. The orientation bundle of $\M_P^{\rm DT}$ is the pullback of the orientation bundle $\bar O_P^{E_\bu}\ra\ovB_P$ under the forgetful map $\M_P^{\rm DT}\ra\ovB_P$, where up to isotopy we may take $E_\bu$ to be the elliptic operator
\begin{equation*} 
D=\bar\partial+\bar\partial^*:\Ga^{\iy} \bigl( \La^{0,0}T^*X\op\La^{0,2}T^*X\bigr)\longra \Ga^{\iy} \bigl(\La^{0,1}T^*X\op\La^{0,3}T^*X\bigr). 
\end{equation*}
As $D$ is complex linear, Theorem \ref{or2thm1} gives canonical orientations on all moduli spaces $\B_P,\ovB_P$ and~$\M_P^{\rm DT}$.

\subsubsection{\texorpdfstring{$G_2$-instantons on $G_2$-manifolds}{G₂-instantons on G₂-manifolds}}
\label{or429}

In the sequel \cite{JoUp} we solve Problem \ref{or3prob} for the Dirac operator on 7-manifolds and $G=\U(m)$ or $\SU(m)$, using a variation on the method of \S\ref{or33}. Here {\it flag structures\/} \cite[\S 3.1]{Joyc2} are an algebro-topological structure on 7-manifolds $X$, related to `linking numbers' of disjoint homologous 3-submanifolds~$Y_1,Y_2\subset X$.

\begin{thm}[Joyce and Upmeier \cite{JoUp}] Suppose\/ $(X,g)$ is a compact, oriented, spin Riemannian $7$-manifold, and take $E_\bu$ to be the Dirac operator $D:\Ga^\iy(S)\ra\Ga^\iy(S)$ on $X$. Fix an orientation on $\det D$ and a \begin{bfseries}flag structure\end{bfseries} on $X,$ as in Joyce {\rm\cite[\S 3.1]{Joyc2}}. Let\/ $G$ be $\U(m)$ or $\SU(m)$ and\/ $P\ra X$ be a principal\/ $G$-bundle. Then we can construct a canonical orientation on\/~$O_P^{E_\bu}\ra\B_P$.

\label{or4thm10}
\end{thm}

The orientability of $O_P^{E_\bu}$ was previously proved by Walpuski \cite[\S 6.1]{Walp1}, but the canonical orientations are new. For Lie groups $G$ other than $\U(m)$ or $\SU(m)$, the analogue of Theorem \ref{or4thm10} may be false. In \cite[\S 2.4]{JoUp} the authors give a compact, oriented, spin 7-manifold $X$ such that $O_P^{E_\bu}\ra\B_P$ is not orientable when $P=X\t\Sp(m)\ra X$ is the trivial $\Sp(m)$-bundle for any~$m\ge 2$.

Theorem \ref{or4thm10} is related to a 7-dimensional gauge theory discussed by Donaldson and Thomas \cite{DoTh} and Donaldson and Segal \cite{DoSe}. Let $(X,\vp,g)$ be a compact $G_2$-manifold with $\d(*\vp)=0$, $G$ a Lie group, and $P\ra X$ a principal $G$-bundle. A $G_2$-{\it instanton\/} on $P$ is a connection $\nabla_P$ on $P$ with $F^{\nabla_P}\w*\vp=0$ in $\Ga^\iy(\Ad(P)\ot\La^6T^*X)$. Write $\M_P^{G_2}$ for the moduli space of irreducible $G_2$-instantons on $P$. Then $\M_P^{G_2}$ is a derived manifold of virtual dimension 0. Examples and constructions of $G_2$-instantons are given in~\cite{MNS,SaEa,SaWa,Walp2,Walp3,Walp4}. 

As in \S\ref{or41}, we may orient $\M_P^{G_2}$ by restricting orientations on $\bar O^{E_\bu}_P\ra\ovB_P$, for $E_\bu$ the Dirac operator of the spin structure on $X$ induced by $(\vp,g)$. Thus Theorem \ref{or4thm10} implies:

\begin{cor} Let\/ $(X,\vp,g)$ be a compact\/ $G_2$-manifold with\/ $\d(*\vp)=0,$ and fix an orientation on $\det D$ and a flag structure on $X$. Then for any principal\/ $G$-bundle\/ $P\ra X$ for\/ $G=\U(m)$ or $\SU(m),$ we can construct a canonical orientation on\/~$\M_P^{G_2}$. 
\label{or4cor1}
\end{cor}

This confirms a conjecture of the first author~\cite[Conj.~8.3]{Joyc2}.

Donaldson and Segal \cite{DoSe} propose defining enumerative invariants of $(X,\vp,g)$ by counting $\M_P^{G_2}$, {\it with signs}, and adding correction terms from associative 3-folds in $X$. To determine the signs we need an orientation on $\M_P^{G_2}$. Thus, Corollary \ref{or4cor1} contributes to the Donaldson--Segal programme.

\subsubsection{\texorpdfstring{$\Spin(7)$-instantons on $\Spin(7)$-manifolds}{Spin(7)-instantons on Spin(7)-manifolds}}
\label{or4210}

In the sequel \cite{CGJ}, using a variation on the method of \S\ref{or33} in which $P\ra X$ is the trivial bundle, Cao, Gross and Joyce prove:

\begin{thm}[Cao, Gross and Joyce {\cite[Th.~1.11]{CGJ}}] Let\/ $(X,g)$ be a compact, oriented, spin Riemannian $8$-manifold, and\/ $E_\bu$ be the positive Dirac operator $D_+:\Ga^\iy(S_+)\ra\Ga^\iy(S_-)$ on $X$. Suppose\/ $P\ra X$ is a principal\/ $G$-bundle for $G=\U(m)$ or $\SU(m)$. Then $\B_P$ is orientable. The mapping spaces $\cC,\cC_\al$ in {\rm\S\ref{or242}} are also orientable.
\label{or4thm11}
\end{thm}

This extends results of Cao and Leung \cite[Th.~1.2]{CaLe2}, who proved Theorem \ref{or4thm11} if $G=\U(m)$ and $H_{\rm odd}(X,\Z)=0$, and Mu\~noz and Shahbazi \cite{MuSh}, who proved Theorem \ref{or4thm11} if $G=\SU(m)$ and $\Hom(H^3(X,\Z),\Z)=0$. As for Theorem \ref{or4thm10}, the analogue of Theorem \ref{or4thm11} for Lie groups $G$ other than $\U(m)$ or $\SU(m)$ may be false. In \cite{CGJ} the authors give an example of a compact, oriented, spin 8-manifold $X$ such that $O_P^{E_\bu}\ra\B_P$ is not orientable when $P=X\t\Sp(m)\ra X$ is the trivial $\Sp(m)$-bundle for any~$m\ge 2$.

Again, Theorem \ref{or4thm10} is related to an 8-dimensional gauge theory discussed by Donaldson and Thomas \cite{DoTh}. Let $(X,\Om,g)$ be a compact $\Spin(7)$-manifold. Then there is a natural splitting $\La^2T^*X=\La^2_7T^*X\op\La^2_{21}T^*X$ into vector subbundles of ranks 7 and 21. Suppose $G$ is a Lie group and $P\ra X$ a principal $G$-bundle. A $\Spin(7)$-{\it instanton\/} on $P$ is a connection $\nabla_P$ on $P$ with $\pi^2_7(F^{\nabla_P})=0$ in $\Ga^\iy(\Ad(P)\ot\La^2_7T^*X)$. Write $\M_P^{\Spin(7)}$ for the moduli space of irreducible $\Spin(7)$-instantons on $P$. Then $\M_P^{\Spin(7)}$ is a derived manifold. Examples of $\Spin(7)$-instantons were given by Lewis \cite{Lewi}, Tanaka \cite{Tana2}, and Walpuski~\cite{Walp5}. 

As in \S\ref{or41}, we may orient $\M_P^{\Spin(7)}$ by restricting orientations on $\bar O^{E_\bu}_P\ra\ovB_P$, for $E_\bu$ the positive Dirac operator of the spin structure on $X$ induced by $(\Om,g)$. Thus Theorem \ref{or4thm11} implies:

\begin{cor} Let\/ $(X,\Om,g)$ be a compact\/ $\Spin(7)$-manifold. Then $\M_P^{\Spin(7)}$ is orientable for any principal\/ $G$-bundle\/ $P\ra X$ with\/ $G=\U(m)$ or\/~$\SU(m)$.
\label{or4cor2}
\end{cor}

Borisov and Joyce \cite{BoJo} and Cao and Leung \cite{CaLe1} set out a programme to define Donaldson--Thomas type invariants `counting' moduli spaces of  (semi)stable coherent sheaves $\M_\al^{\coh}$ on a Calabi--Yau 4-fold $X$. To do this requires an `orientation' on $\M_\al^{\coh}$ in the sense of \cite[\S 2.4]{BoJo}. Cao, Gross and Joyce \cite[Cor.~1.17]{CGJ} use Theorem \ref{or4thm11} to prove that all such moduli spaces $\M_\al^{\coh}$ are orientable.

\medskip

\noindent{\small\sc The Mathematical Institute, Radcliffe
Observatory Quarter, Woodstock Road, Oxford, OX2 6GG, U.K.

\noindent E-mails: {\tt joyce@maths.ox.ac.uk, tanaka@maths.ox.ac.uk, \\
upmeier@maths.ox.ac.uk.}}


\begin{thebibliography}{99}
\addcontentsline{toc}{section}{References}

\bibitem{AkMc} S. Akbulut and J. McCarthy, {\it Casson's invariant for oriented homology $3$-spheres --- an exposition}, Math. Notes 36, Princeton University Press, 1990.

\bibitem{Atiy} M.F. Atiyah, {\it K-Theory}, W.A. Benjamin inc., 1967 / Addison-Wesley, 1989.

\bibitem{AtBo} M.F. Atiyah and R. Bott, {\it The Yang--Mills equations over Riemann surfaces}, Phil. Trans. Roy. Soc. London A308 (1982), 523--615.

\bibitem{AHS} M.F. Atiyah, N.J. Hitchin, and I.M. Singer, 
{\it Self-duality in four-dimensional Riemannian geometry}, 
Proc. Roy. Soc. London Ser. A 362 (1978), 425--461. 

\bibitem{AtSi1} M.F. Atiyah and I.M. Singer, {\it The Index of Elliptic Operators: I}, Ann. of Math. 87 (1968), 484--530.

\bibitem{AtSe} M.F. Atiyah and G.B. Segal, {\it The Index of Elliptic Operators: II}, Ann. of Math. 87 (1968), 531--545.

\bibitem{AtSi3} M.F. Atiyah and I.M. Singer, {\it The Index of Elliptic Operators: III}, Ann. of Math. 87 (1968), 546--604.

\bibitem{AtSi4} M.F. Atiyah and I.M. Singer, {\it The Index of Elliptic Operators: IV}, Ann. of Math. 92 (1970), 119--138.

\bibitem{AtSi5} M.F. Atiyah and I.M. Singer, {\it The Index of Elliptic Operators: V}, Ann. of Math. 93 (1971), 139--149.

\bibitem{AtSi6} M.F. Atiyah and I.M. Singer, {\it Dirac operators coupled to vector potentials}, Proc. Nat. Acad. Sci. U.S.A. 81 (1984), Phys. Sci., 2597--2600. 

\bibitem{BoHe} H.U. Boden and C.M. Herald, {\it The ${\rm SU}(3)$ Casson invariant for integral homology $3$-spheres}, J. Diff. Geom. 50 (1998), 147--206. 

\bibitem{Bore} A. Borel, {\it Topology of Lie groups and characteristic classes}, Bull. A.M.S. 61 (1955), 397--432.

\bibitem{BoJo} D. Borisov and D. Joyce, {\it Virtual fundamental classes for moduli spaces of sheaves on Calabi--Yau four-folds}, Geometry and Topology 21 (2017), 3231--3311. \href{http://arxiv.org/abs/1504.00690}{arXiv:1504.00690}.

\bibitem{CGJ} Y. Cao, J. Gross and D. Joyce, {\it Orientability of moduli spaces of\/ $\Spin(7)$-instantons and coherent sheaves on Calabi--Yau $4$-folds}, \href{http://arxiv.org/abs/1811.09658}{arXiv:1811.09658}, 2018.

\bibitem{CaLe1} Y. Cao and N.C. Leung, {\it Donaldson--Thomas theory for Calabi--Yau $4$-folds}, \href{http://arxiv.org/abs/1407.7659}{arXiv:1407.7659}, 2014.

\bibitem{CaLe2} Y. Cao and N.C. Leung, {\it Orientability for gauge theories on Calabi--Yau manifolds}, Adv. Math. 314 (2017), 48--70. \href{http://arxiv.org/abs/1502.01141}{arXiv:1502.01141}.

\bibitem{Dona1} S.K. Donaldson, {\it An application of gauge theory to four-dimensional topology}, J. Diff. Geom. 18 (1983), 279--315.

\bibitem{Dona2} S.K. Donaldson, {\it The orientation of Yang--Mills moduli spaces and\/ $4$-manifold topology}, J. Diff. Geom. 26 (1987), 397--428. 

\bibitem{Dona3} S.K. Donaldson, {\it Floer homology groups in Yang--Mills theory}, Cambridge Tracts in Mathematics, 147, CUP, Cambridge, 2002. 

\bibitem{DoKr} S.K. Donaldson and P.B. Kronheimer, {\it The Geometry of Four-Manifolds}, OUP, 1990.

\bibitem{DoSe} S.K. Donaldson and E. Segal, {\it Gauge Theory in Higher Dimensions, II}, Surveys in Diff. Geom. 16 (2011), 1--41. \href{http://arxiv.org/abs/0902.3239}{arXiv:0902.3239}.

\bibitem{DoTh} S.K. Donaldson and R.P. Thomas, {\it Gauge Theory in Higher Dimensions}, Chapter 3 in S.A. Huggett et al., editors, {\it The Geometric Universe}, OUP, Oxford, 1998.

\bibitem{FHT} D.S. Freed, M.J. Hopkins, and C. Teleman, {\it Consistent orientation of moduli spaces}, pages 395--419 in O. Garc\'\i a-Prada et al., editors, {\it The many facets of geometry}, OUP, Oxford, 2010. \href{http://arxiv.org/abs/0711.1909}{arXiv:0711.1909}.

\bibitem{GaUh} M. Gagliardo and K. Uhlenbeck, {\it Geometric Aspects of the Kapustin--Witten Equations}, J. Fixed Point Theory Appl. 11 (2012), 185--198.  \href{http://arxiv.org/abs/1401.7366}{arXiv:1401.7366}.

\bibitem{Hayd} A. Haydys, {\it Fukaya--Seidel category and gauge theory}, 
J. Symplectic Geom. 13 (2015), 151--207. \href{http://arxiv.org/abs/1010.2353}{arXiv:1010.2353}.

\bibitem{HJJS} D. Husem\"oller, M. Joachim, B. Jur\v co, and M. Schottenloher, {\it Basic Bundle Theory and K-Cohomology Invariants}, Lecture Notes in Physics, Springer, 2008. 

\bibitem{Joyc1} D. Joyce, {\it An introduction to d-manifolds and
derived differential geometry}, pages 230--281 in L. Brambila-Paz et al., editors, {\it Moduli spaces}, L.M.S. Lecture Notes 411, Cambridge University Press, 2014. \href{http://arxiv.org/abs/1206.4207}{arXiv:1206.4207}.

\bibitem{Joyc2} D. Joyce, {\it Conjectures on counting associative $3$-folds in $G_2$-manifolds}, pages 97--160 in V. Mu\~noz et al., editors, {\it Modern Geometry: A Celebration of the Work of Simon Donaldson}, Proc. Symp. Pure Math. 99, A.M.S., Providence, RI, 2018. \href{http://arxiv.org/abs/1610.09836}{arXiv:1610.09836}.

\bibitem{Joyc3} D. Joyce, {\it Kuranishi spaces as a\/ $2$-category},  to appear in J. Morgan, editor, {\it Virtual Fundamental Cycles in Symplectic Topology}, book to be published by the A.M.S. in 2019. \href{http://arxiv.org/abs/1510.07444}{arXiv:1510.07444}.

\bibitem{Joyc4} D. Joyce, {\it D-manifolds and d-orbifolds: a
theory of derived differential geometry}, to be published by OUP, 2019. Preliminary version (2012) available at \url{http://people.maths.ox.ac.uk/~joyce/dmanifolds.html}.

\bibitem{Joyc5} D. Joyce, {\it Kuranishi spaces and Symplectic Geometry}, multiple volume book in progress, 2017--2027. Preliminary versions of volumes I, II available at \url{http://people.maths.ox.ac.uk/~joyce/Kuranishi.html}.

\bibitem{Joyc6} D. Joyce, {\it Ringel--Hall style Lie algebra structures on the homology of moduli spaces}, preprint, 2019.

\bibitem{JoUp} D. Joyce and M. Upmeier, {\it Canonical orientations for moduli spaces of\/ $G_2$-instantons with gauge group $\SU(m)$ or $\U(m)$}, \href{http://arxiv.org/abs/1811.02405}{arXiv:1811.02405}, 2018.

\bibitem{KaWi} A. Kapustin and E. Witten, {\it Electric-magnetic duality and the geometric Langlands program}, Commun. Number Theory Phys. 1 (2007), 1--236. \href{https://lanl.arxiv.org/abs/hep-th/0604151}{hep-th/0604151}.

\bibitem{Karo} M. Karoubi, {\it K-Theory, an introduction}, Grundlehren der math. Wiss. 226, Springer-Verlag, Berlin, 1978.

\bibitem{KnMu} F.F. Knudsen and D. Mumford, {\it The projectivity of the moduli space of stable curves. I. Preliminaries on ``det'' and ``Div''}, Math. Scand. 39 (1976), 19--55. 

\bibitem{Kron} P. Kronheimer, {\it Four-manifold invariants from higher rank bundles}, J. Diff. Geom. 70 (2005), 59--112.

\bibitem{Lewi} C. Lewis, {\it $\Spin(7)$ instantons}, D.Phil.\ thesis, Oxford University, 1998.

\bibitem{Mare} B. Mares, {\it Some Analytic Aspects of Vafa--Witten Twisted\/ $\mathcal{N}=4$ Supersymmetric Yang--Mills theory}, PhD thesis, M.I.T., 2010. 

\bibitem{May1} J.P. May, {\it $E_\iy$-ring spaces and 
$E_\iy$-ring spectra}, Springer Lecture Notes in Math. 577, Springer-Verlag, Berlin, 1977. 

\bibitem{May2} J.P. May, {\it A concise course in Algebraic Topology}, University of Chicago Press, Chicago, 1999.

\bibitem{MNS} G. Menet, J. Nordstr\"om and H.N. S\'a Earp, {\it Construction of\/ $G_2$-instantons via twisted connected sums}, \href{http://arxiv.org/abs/1510.03836}{arXiv:1510.03836}, 2015.

\bibitem{Metz} D.S. Metzler, {\it Topological and smooth stacks}, \href{http://arxiv.org/abs/math/0306176}{math.DG/0306176}, 2003.

\bibitem{MiSt} J.W. Milnor and J.D. Stasheff, {\it Characteristic classes}, Annals of Math. Studies 76, Princeton University Press, Princeton, NJ, 1974.

\bibitem{Morg} J.W. Morgan, {\it The Seiberg--Witten equations and applications to the topology of smooth of smooth four-manifolds}, Mathematical Notes 44, Princeton University Press, 1996. 

\bibitem{MuSh} V. Mu\~noz and C.S. Shahbazi, {\it Orientability of the moduli space of\/ $\Spin(7)$-instantons}, Pure. Appl. Math. Quarterly 13 (2017), 453--476. \hfil\break
\href{http://arxiv.org/abs/1707.02998}{arXiv:1707.02998}.

\bibitem{Nico} L. Nicolaescu, {\it Notes on Seiberg--Witten theory}, 
Graduate Studies in Math, 28, A.M.S., Providence, RI, 2000.  

\bibitem{Nooh1} B. Noohi, {\it Foundations of topological stacks, I}, \href{http://arxiv.org/abs/math/0503247}{math.AG/0503247}, 2005.

\bibitem{Nooh2} B. Noohi, {\it Homotopy types of topological stacks}, \href{http://arxiv.org/abs/0808.3799}{arXiv:0808.3799}, 2008.

\bibitem{Quil} D. Quillen, {\it Determinants of Cauchy--Riemann operators on Riemann surfaces}, Functional Anal. Appl. 19 (1985), 31--34.

\bibitem{ReCa} R. Reyes Carri\'on, {\it A generalization of the notion of instanton}, Differential Geom. Appl. 8 (1998), 1--20.

\bibitem{SaEa} H.N. S\'a Earp, {\it $G_2$-instantons over asymptotically cylindrical manifolds}, Geom. Topol. 19 (2014), 61--111. \href{http://arxiv.org/abs/1101.0880}{arXiv:1101.0880}.

\bibitem{SaWa} H.N. S\'a Earp and T. Walpuski, {\it $G_2$-instantons over twisted connected sums}, Geom. Topol. 19 (2015), 1263--1285. \href{http://arxiv.org/abs/1310.7933}{arXiv:1310.7933}.

\bibitem{Sega} G. Segal, {\it Categories and cohomology theories}, Topology 13 (1974), 293-312.

\bibitem{Swit} R.M. Switzer, {\it Algebraic Topology -- Homotopy and Homology}, Grundlehren der math. Wiss. 212, Springer-Verlag, New York, 1975.

\bibitem{Tana1} Y. Tanaka, {\it A construction of\/ $\Spin(7)$-instantons}, Ann. Global Anal. Geom. 42 (2012), 495--521. \href{http://arxiv.org/abs/1201.3150}{arXiv:1201.3150}.

\bibitem{Tana2} Y. Tanaka, {\it On the moduli space of the Donaldson--Thomas instantons}, Extracta Math. 31 (2016), 89--107. \href{http://arxiv.org/abs/0805.2192}{arXiv:0805.2192}. 

\bibitem{Tana3} Y. Tanaka, {\it Some boundedness properties of the Vafa--Witten equations on closed four-manifolds}, Q. J. Math. 68 (2017), 1203--1225. 
\href{http://arxiv.org/abs/1308.0862}{arXiv:1308.0862}. 

\bibitem{Tana4} Y. Tanaka, {\it On the singular sets of solutions to the Kapustin--Witten equations and the Vafa--Witten ones on compact K\"ahler surfaces}, to appear in Geometriae Dedicata, 2018. \href{http://arxiv.org/abs/1510.07739}{arXiv:1510.07739}.

\bibitem{Taub1} C.H. Taubes, {\it Casson's invariant and gauge theory}, 
J. Diff. Geom. 31 (1990), 547--599. 

\bibitem{Taub2} C.H. Taubes, {\it Compactness theorems for $\SL(2,\C)$ generalizations of the $4$-dimensional anti-self dual equations}, \href{http://arxiv.org/abs/1307.6447}{arXiv:1307.6447}, 2013.

\bibitem{Upme} M. Upmeier, {\it A categorified excision principle for elliptic symbol families}, \href{http://arxiv.org/abs/1901.10818}{arXiv:1901.10818}, 2019.

\bibitem{VaWi} C. Vafa and E. Witten, {\it A strong coupling test of $S$-duality}, Nucl. Phys. B, 431 (1994), 3--77. \href{https://lanl.arxiv.org/abs/hep-th/9408074}{hep-th/9408074}.

\bibitem{Walp1} T. Walpuski, {\it Gauge theory on $G_2$-manifolds}, PhD Thesis, Imperial College London, 2013.

\bibitem{Walp2} T. Walpuski, {\it $G_2$-instantons on generalized Kummer constructions}, Geom. Topol. 17 (2013), 2345--2388. \href{http://arxiv.org/abs/1109.6609}{arXiv:1109.6609}.

\bibitem{Walp3} T. Walpuski, {\it $G_2$-instantons over twisted connected sums: an example}, Math. Res. Lett. 23 (2016), 529--544. \href{http://arxiv.org/abs/1505.01080}{arXiv:1505.01080}.

\bibitem{Walp4} T. Walpuski, {\it $G_2$-instantons, associative submanifolds and Fueter sections}, Comm. Anal. Geom. 25 (2017), 847--893. \href{http://arxiv.org/abs/1205.5350}{arXiv:1205.5350}.

\bibitem{Walp5} T. Walpuski, {\it $\Spin(7)$-instantons, Cayley submanifolds and Fueter sections}, Comm. Math. Phys. 352 (2017), 1--36. \href{http://arxiv.org/abs/1409.6705}{arXiv:1409.6705}.

\bibitem{Witt} E. Witten, {\it Fivebranes and Knots}, Quantum Topol. 3 (2012), 1--137. \href{http://arxiv.org/abs/1101.3216}{arXiv:1101.3216}.

\bibitem{Zent} R. Zentner, {\it On higher rank instantons and the monopole cobordism program}, Q. J. Math. 63 (2012), 227--256. \href{http://arxiv.org/abs/0911.5146}{arXiv:0911.5146}.

\end{thebibliography}
\end{document}